\documentclass[a4paper,twoside,11pt]{article}
\usepackage{a4wide}
\usepackage[utf8]{inputenc}
\usepackage[T1]{fontenc}
\usepackage[english]{babel} 
\usepackage{amsmath,amssymb,amsthm,amsfonts}
\usepackage{mathabx} 
\usepackage{mathrsfs}
\usepackage{bbm}
\usepackage{graphicx}
\usepackage{hyperref}
\usepackage{fancyhdr}
\usepackage{enumitem}

\theoremstyle{plain}
\newtheorem{theorem}{Theorem}[section]
\newtheorem{lemma}[theorem]{Lemma}
\newtheorem{corollary}[theorem]{Corollary}
\newtheorem{proposition}[theorem]{Proposition}

\theoremstyle{definition}

\theoremstyle{remark}
\newtheorem{remark}[theorem]{Remark}

\numberwithin{equation}{section}

\newcommand{\tconst}[1]{\ensuremath{t_{\textnormal{\ref{#1}}}}}
\newcommand{\const}[1]{\ensuremath{c_{\textnormal{\ref{#1}}}}}
\newcommand{\Const}[1]{\ensuremath{C_{\textnormal{\ref{#1}}}}}

\def\C{\ensuremath{\mathbf{C}}}
\def\N{\ensuremath{\mathbf{N}}}
\def\Q{\ensuremath{\mathbf{Q}}}
\def\R{\ensuremath{\mathbf{R}}}
\def\Z{\ensuremath{\mathbf{Z}}}

\def\Znn{\ensuremath{\Z_{\ge 0}}}

\def\ep{\varepsilon}

\def\E{\ensuremath{\mathbf{E}}}
\def\P{\ensuremath{\mathbf{P}}}
\DeclareMathOperator{\Var}{Var}
\def\F{\ensuremath{\mathscr{F}}}

\def\Ind{\ensuremath{\mathbf{1}}}

\renewcommand\Re{\operatorname{Re}}
\renewcommand\Im{\operatorname{Im}}

\def\to{\rightarrow}

\def\tand{\ensuremath{\text{ and }}}
\def\tif{\ensuremath{\text{ if }}}
\def\tas{\ensuremath{\text{ as }}}
\def\ton{\ensuremath{\text{ on }}}
\def\tor{\ensuremath{\text{ or }}}

\newcommand{\dd}{\mathrm{d}} 

\author{\textsc{Pascal Maillard}\thanks{Laboratoire de Probabilités et Modèles Aléatoires, UMR~7599, Université~Paris~VI, Case~courrier~188, 4~Place~Jussieu, 75252~PARIS~Cedex~05, FRANCE, e-mail: \texttt{pascal.maillard@upmc.fr}}}
\title{Branching Brownian motion with selection of the $N$ right-most particles: An approximate model}
\date{February 3, 2012}

\def\G{\ensuremath{\mathscr{G}}}
\def\H{\ensuremath{\mathscr{H}}}
\def\p{\ensuremath{\mathbf{p}}}
\def\Ptilde{\widetilde{\P}}
\def\Etilde{\widetilde{\E}}

\def\Phat{\widehat{\P}}
\def\Ehat{\widehat{\E}}
\def\Qhat{\widehat{\Q}}
\def\Varhat{\widehat{\Var}}

\begin{document}

\maketitle

{\leftskip=2truecm \rightskip=2truecm \baselineskip=15pt \small

\noindent{\bfseries Abstract.} We present an approximation to the Brunet--Derrida model of supercritical branching Brownian motion on the real line with selection of the $N$ right-most particles, valid when the population size $N$ is large. It consists of introducing a random space-time barrier at which particles are instantaneously killed in such a way that the population size stays almost constant over time. We prove that the suitably recentered position of this barrier converges at the $\log^3 N$ timescale to a L\'evy process, which we identify. This validates the physicists' predictions about the fluctuations in the Brunet--Derrida model.

\bigskip

\noindent{\bfseries Keywords.} Branching Brownian motion, Brunet--Derrida particle system.

%

}


\section{Introduction}

Consider one-dimensional supercritical branching Brownian motion (BBM): particles diffuse on the real line according to standard Brownian motions and split independently with rate 1 into a random number of particles distributed according to the reproduction law $(q(k))_{k\ge 0}$, with mean greater than 1 and finite second moment. The physicists Brunet and Derrida have introduced a model of BBM with \emph{selection}: Fix a (large) parameter $N$, and as soon as the number of particles exceeds $N$, instantaneously kill the \emph{left-most}, in order to have at most $N$ particles at any time. This model, which we call the $N$-BBM, has been studied by them and coauthors in extraordinary detail \cite{Brunet1997,Brunet2001,Brunet2006,Brunet2006a}.
In a first approximation, they model the system by a deterministic traveling wave of an FKPP-type equation with \emph{cutoff} \cite{Brunet1997} (FKPP stands for Fisher, Kolmogorov, Petrovskii and Piskounov after \cite{Fisher1937,KPPshort}). Assuming the validity of this approximation, they find that the linear speed of $N$-BBM differs from the speed of the right-most particle in BBM without selection by a quantity of the order of $\log^2 N$. In a subsequent work \cite{Brunet2006}, they introduce a better, semi-deterministic approximation, which does not only yield a better asymptotic for the speed, but also the complete set of cumulants of the position of the front, all of them scaling as $\log^{-3}N$. Moreover, this approximate model together with numerical simulations suggests \cite{Brunet2006a} that the genealogy of the system can be described on a timescale $\log^3 N$ by the celebrated Bolthausen--Sznitman coalescent \cite{Bolthausen1998}.

In order to explain the presence of the Bolthausen--Sznitman coalescent, Berestycki, Beres\-tycki and Schweinsberg \cite{Berestycki2010} approximate the $N$-BBM by BBM with a linear space-time barrier at which particles are instantaneously killed and whose slope is exactly the conjectured speed of the $N$-BBM (this idea was also present in \cite{Berard2010} and indeed already in \cite{Brunet2006}). They show for their system that the number of particles and the genealogy of the system converge in the $\log^3 N$ timescale to Neveu's continuous-state branching process and the Bolthausen--Sznitman coalescent, respectively.
Here, we will push their results further and {\em present an approximation of the $N$-BBM by BBM with a random absorbing space-time barrier (the model is defined properly in Section \ref{sec:moving_barrier}) with the following properties.
\begin{enumerate}
 \item The number of particles in the approximate model stays almost constant over the time\-scale $\log^3 N$.
 \item The position of the random barrier, suitably recentered, converges at the $\log^3 N$ timescale to a L\'evy process $(L_t)_{t\ge 0}$, given by
\begin{equation}
\label{eq:laplace_levy}
 \log E[e^{i\lambda L_1}] = i\lambda c + c_0\int_0^\infty e^{i\lambda x} - 1 - i\lambda x\Ind_{(x\le 1)}\,\Lambda(\dd x).
\end{equation}
Here, $c_0 = \sqrt{2\sum_k(k-1)q(k)}$, $c\in\R$ is a constant and $\Lambda(\dd x)$ is the image of the measure $x^{-2}\dd x$ by the map $x\mapsto c_0^{-1}\log(1+x)$.
\end{enumerate}
}
These properties are stated in precise form in Section \ref{sec:moving_barrier}.

In a next work, we plan to show that this approximation can be coupled with the $N$-BBM in such a way that our results can be transferred to that model.

\subsection{Related work}

The author is aware of only two mathematically rigorous articles on the $N$-BBM or the $N$-BRW (branching random walk): B\'erard and Gou\'er\'e \cite{Berard2010} prove the $\log^2 N$ correction of the linear speed of $N$-BRW, thereby showing the validity of the approximation by a deterministic traveling wave with cutoff. Durrett and Remenik \cite{Durrett2009} study the empirical distribution of $N$-BRW and show that it converges to a system of integro-differential equations with moving boundary. BBM with absorption at a linear space-time barrier however is a well-studied process (see for example \cite{Derrida2007,Gantert2011,HHK2006,Kesten1978,Neveu1988}) and is much more tractable than $N$-BBM due to the greater independence between the particles and its connection with some differential equations \cite{HHK2006,Maillard2010,Neveu1988}.

In addition to its intrinsic interest, the $N$-BBM is believed to be representative for general \emph{noisy traveling waves} (see \cite{Brunet2001} or the review articles \cite{VanSaarloos2003}, Chapter 7, or \cite{Panja2004}). There is indeed an exact duality relation between the FKPP equation with Gaussian white noise and a system of branching and coalescing Brownian motions, discovered by Shiga \cite{Shiga1988} in the context of stepping stone models. Recently, an estimate for the speed of a traveling wave of the noisy FKPP equation was established \cite{Mueller2010} which partly confirms the physicists' predictions. We believe that the present paper will we useful in the study of its dual branching-coalescing system, which could potentially lead to an improvement of the results for the noisy FKPP equation.

Let us also note that branching Brownian motion \emph{without} selection has a long history: Starting with \cite{Sevastyanov1958} it has been studied by many authors and under various aspects, along with its discrete counterpart, the branching random walk. Since \cite{McKean1975}, its connection to the FKPP equation has raised very fruitful interactions between analysis and probability theory (see for example \cite{Kyprianou2004} and the references therein). BBM has been used in applications, for example to model ecological and epidemic spread \cite{Mollison1977} or directed polymers on disordered trees \cite{Derrida1988}. During the last years, there has been renewed interest in the behavior of its extremal particles, be it the right-most only \cite{Addario-Berry2009,Hu2009,Aidekon2011b} or the whole point process formed by the particles at the right edge \cite{Aidekon2011c,Arguin2011, Arguin2011a,Madaule2011}. The extremal statistics of several other models have actually been shown or are conjectured to belong to the same universality class as BBM, such as the Gaussian Free Field on a two-dimensional lattice \cite{Bolthausen2001,Bolthausen2010,Bramson2010}, or the cover time of a 2D box by a random walk (see e.g.\ \cite{Ding2011b} and the references therein). 

\subsection{Heuristic ideas and overview of the results}
\label{sec:heuristic}

We recall the heuristic semi-deterministic description of $N$-BBM established in \cite{Brunet2006}:
\begin{enumerate}
 \item Most of the time, the particles are in a meta-stable state, in which the diameter of the cloud of particles (also called the \emph{front}) is approximately $L = c_0^{-1} \log N$, the empirical density of the particles proportional to $e^{-c_0 x} \sin(\pi x/L)$, and the system moves at a linear speed $v_{\mathrm{cutoff}} = c_0 - c_0\pi^2/(2\log^2N)$. This is the description provided by the \emph{cutoff} approximation from \cite{Brunet1997}.
 \item This meta-stable state is perturbed from time to time by particles moving far to the right and thus spawning a big number of descendants, causing a shift of the front to the right after a relaxation time which is of the order of $\log^2 N$. To make this precise, we fix a point in the bulk, for example the barycenter of the cloud of particles, and shift our coordinate system such that this point becomes its origin. Playing with the initial conditions of the FKPP equation with cutoff, the authors of \cite{Brunet2006} found that a particle moving up to the point $L + \delta$ causes a shift of the front by
\[
 R(\delta) = \frac 1 {c_0} \log \Big(1+\frac{Ce^{c_0\delta}}{L^3}\Big),
\]
for some constant $C>0$. In particular, in order to have an effect on the position of the front, a particle has to reach $L_{\mathrm{eff}} + O(1)$, where $L_{\mathrm{eff}} = L + c_0^{-1} 3 \log \log N$.
\item Assuming that such an event, where a particle ``escapes'' to the point $L + \delta$, happens with rate $C e^{-c_0 \delta}$, one sees that the time it takes for a particle to come close to $L_{\mathrm{eff}}$, and thus causing shifts of the front, is of the order of $\log^3 N$, which is much longer than the relaxation time when $N$ is large.
\item With this information, the full statistics of the position of front (the speed $v$ and the cumulants of order $n\ge 2$) were found to be
\begin{equation}
 \label{eq:statistics}
\begin{split}
 v - v_{\mathrm{cutoff}} &\approx \pi^2 c_0\frac{3\log\log N}{\log^3 N}\\
\frac{\text{[$n$-th cumulant]}}{t} &\approx \frac{\pi^2 c_0^{2-n} n! \zeta(n)}{\log^3 N},\quad n\ge 2,
\end{split}
\end{equation}
where $\zeta$ denotes the Riemann zeta-function.
\end{enumerate}

Berestycki, Berestycki and Schweinsberg \cite{Berestycki2010} put this description onto a rigorous foundation. They study BBM with absorption at the origin and with drift $-\mu$, where $\mu = \sqrt{c_0^2 - \pi^2/L_{\mathrm{eff}}^2} = v + o(L^{-3})$. Their starting point is to introduce a second barrier at the point $L_A = L_{\mathrm{eff}} - c_0^{-1} A$, for some large positive constant $A$, and divide the particles into two parts: One the one hand those that stay inside the interval $(0,L_A)$ or get absorbed at $0$, on the other hand those that eventually hit the second barrier. This corresponds roughly to the division of the process into a deterministic and a stochastic part. Indeed, killing the particles at the right barrier prevents the number of particles to grow fast and thus permits to calculate expectations and variances of various quantities. The quantities one is interested in, for example the number of particles at the time $\log^3 N$, will then have variances of order $e^{-A}$, such that for large $A$, this system behaves almost deterministically at the macroscopic scale. Moreover, the shape of the front predicted by the physicists, with a density proportional to $e^{-c_0 x} \sin(\pi x/L)$, follows simply from the transition density of Brownian motion with drift killed at the border of the interval $(0,L_A)$. As for the particles that hit the right barrier, the number of descendants of such a particle will be at a later time of the order of $e^{-A}NW$, where $W$ is a random variable with tail $P(W > x) \sim 1/x$, as $x$ goes to infinity. Moreover, the rate at which particles hit the right barrier is of the order of $e^A/\log^3 N$. Putting the pieces together, the authors of \cite{Berestycki2010} then show that the process which counts the number of particles of the system converges in the $\log^3 N$ timescale to Neveu's continuous-state branching process and its genealogy to the Bolthausen--Sznitman coalescent.

In this article, we validate the physicists' predictions concerning the fluctuations of the position of the $N$-BBM. Similar to \cite{Berestycki2010}, we approximate the $N$-BBM by BBM with negative drift and absorption at a barrier, but instead of keeping the barrier fixed at the origin, \emph{we will make the barrier move along with process, in such a way that the number of particles stays almost constant}. The movement of the barrier is very simple. Most of the time it does not move at all. Only when a particle hits a point $a>0$ \emph{and} spawns a lot of descendants, we move the barrier to the right in order to kill particles and thus make the population size stay almost constant. After this system has relaxed (which takes a time of order $a^2$), the barrier stays fixed again and we repeat this process, with the point $a$ shifted by the amount the barrier has moved.

Let us go into the details. Our system is defined in terms of the three positive parameters $a$, $A$ and $\kappa$, and we define $N$ by $a = c_0^{-1} (\log N + 3\log \log N - A)$.\footnote{We use the letter $a$ instead of $L_A$ for typographical reasons.} Initially, we have a barrier located at the origin and a set of particles in the interval $(0,a)$, such that $Z_0 \approx \kappa e^A$, where
\[
 Z_t = \sum ae^{\mu(X_u(t)-a)}\sin \frac{\pi x}{a}.
\]
Here, we sum over all the particles $u$ alive at time $t$ and $X_u(t)$ denotes the distance of the particle $u$ from the barrier at time $t$. We then let the particles evolve as branching Brownian motions with branching rate 1, reproduction law $q(k)$ and drift $-\mu$, where $\mu = \sqrt{c_0^2 - \pi^2/a^2}$. Furthermore, particles are killed as soon as they touch the barrier. We recall that by hypothesis the reproduction law $q(k)$ has mean greater than $1$ and finite second moment.\footnote{This last condition is only technical and we believe our results to be true for more general reproduction laws. In fact, in Section \ref{sec:critical_line}, we show that the asymptotic results on the random variable $W$ obtained in \cite{Berestycki2010} still hold if $\sum_{k\ge 1} k\log^3 kq(k)$ is finite, but we don't know whether this condition is sufficient for Theorems \ref{th:barrier} and \ref{th:barrier2}.}

The process $(Z_t)_{t\ge 0}$ is important for two reasons: Firstly, when $a$ is large, the number of particles at a time $t+s$, where $a^2 \ll s \ll a^3$, is approximately $(2\pi c_0) N e^{-A} Z_t$ \cite{Berestycki2010}. Hence, the initial condition is chosen in such a way that the number of particles is proportional to $N$. Secondly, if we kill particles at $0$ and $a$, then $Z_t$ is a martingale and therefore very easy to handle.

When a particle hits the right barrier at the time $T$, say, we absorb its descendants at the space-time line $\mathscr L = a-y + (c_0-\mu)(t-T)$, where $y$ is a large constant depending on $A$ only (this idea comes from \cite{Berestycki2010}). In doing so, the number of particles absorbed at the barrier has the same law as in BBM with absorption at a \emph{critical} line, i.e.\ a line with slope $c_0$. Defining then $Z'$ as $Z_t$, but summing only over these descendants, we know that at a later time the number of descendants of this particle will be of the order of $e^{-A} N Z'$. Consequently, we say that a \emph{breakout} occurs, whenever $Z' > \ep e^A$, where $\ep$ will be chosen such that in particular $\ep \ll 1/A$. Looking at the definition of $Z_t$, it is easy to guess by which amount $\Delta$ we have to move the barrier in order to counterbalance the breakout: Choosing $\Delta = c_0^{-1}\log(1+Z'/(\kappa e^A))$, the value of $Z_t$ is approximately divided by $1+Z'/(\kappa e^A)$, such that after the relaxation time, the value of $Z_t$ and the number of particles should again be approximately $\kappa e^A$ and $(2\pi c_0\kappa)N$. This is basically true, but we also have to take into account the fluctuations of $Z_t$ between the times $0$ and $T$, which are mostly due to the particles hitting the point $a$ without producing a breakout. For this reason, the actual definition of $\Delta$ in \eqref{eq:Delta_def} differs from the one given here. Nevertheless, the above considerations already explain the convergence of the barrier to the L\'evy process given by \eqref{eq:laplace_levy}: One the one hand, we have $Z' \approx (\pi/c_0) W$, where $W$ is the random variable mentioned above, such that the law of $e^{-A}Z'$ conditioned on $Z'>\ep e^A$ is approximately $\ep x^{-2}\Ind_{(x\ge\ep)}\,\dd x$ for large $A$ and $a$.\footnote{The statement ``for large $A$ and $a$'' means that we let first $a$, then $A$ go to infinity, see Section \ref{sec:before_breakout_definitions}.} On the other hand, we will show that breakouts occur at a rate proportional to $\ep^{-1}a^{-3}$. Together with the definition of $\Delta$, this explains the L\'evy measure $\Lambda(\dd x)$ of \eqref{eq:laplace_levy}. One easily checks that the cumulants of this L\'evy process coincide with \eqref{eq:statistics}.

We want to stress two more points. First, in \cite{Berestycki2010}, the authors cut the interval $[0,a^3]$ into tiny pieces of size $\theta a^3$, with $\theta \ll e^{-A}$, in order to make sure that with high probability at most one breakout occurs during a single piece. In adapting this approach to our system with the moving barrier, we found it however difficult to control the fluctuations of the process $Z_t$ over the whole interval of time $[0,a^3]$. We therefore chose another approach, which also has the advantage of giving more information about the history of the particle that causes the breakout. Namely, we will classify the particles into \emph{tiers}, according to the number of times they have hit the point $a$ and come back to the space-time line $\mathscr L$ mentioned above. Thus, when a tier $0$ particle hits the point $a$, it advances to tier $1$, and its descendants have a second chance to break out after having come back to $\mathscr L$. We can then define \emph{the time $T$ of the first breakout} and will indeed show that $T$ is approximately exponentially distributed with rate proportional to $\ep^{-1}a^{-3}$. Interestingly, we will see that with high probability breakouts only occur from particles of tier $0$ or $1$, the number of breakouts occuring from particles of tier $1$ between the times $0$ and $a^3$ being approximately proportional to $A$ (and the remaining $\approx \ep^{-1}$ breakouts occurring from particles of tier $0$). In order to study the system \emph{up to the time $T$}, we will then study BBM \emph{conditioned to break out at time $t$} for every $t\ge 0$, which can be formulated in terms of a Doob transform of the process.

The second point concerns the shape of the barrier. If we were only interested in the state of the system \emph{after} is has relaxed, we could shift the barrier instantaneously by $\Delta$. However, since in a second work we plan to couple the model of this article with the $N$-BBM, we need to move the barrier continuously and over the timescale $a^2$, which is the relaxation time of the system. In this article, we will allow the shape of the barrier to be given by an arbitrary family of ``barrier functions'' $(f_\Delta)_{\Delta\ge0}$, which are non-decreasing, twice differentiable functions with $f_\Delta(0) = 0$ and $f_\Delta(+\infty) = \Delta$, plus some uniformity conditions. First-moment calculations then suggest that the right barrier function to choose for the coupling with the $N$-BBM is
\[
 f_\Delta(t) = c_0^{-1} \log\Big(1+(e^{c_0\Delta}-1)\pi^{-2}e^{\pi^2 t/2}\frac{\dd}{\dd t}\theta(1,t)\Big),
\]
where $\theta(x,t)$ is defined in \eqref{eq:theta}.


We finally remark that although parts of this article (mostly in Sections \ref{sec:critical_line} and \ref{sec:interval}) are heavily based on \cite{Berestycki2010}, it is entirely self-contained. This means that we will reprove some results of \cite{Berestycki2010}, often because we need stronger or different versions, but also sometimes because we found simpler proofs. We think that this is for the benefit of the reader.

\tableofcontents
\newpage

\section{Brownian motion in an interval}
\label{sec:bm}

In this section, we recall some explicit formulas concerning real-valued Brownian motion killed upon exiting an interval. These formulae naturally involve Jacobi theta functions, since these are fundamental solutions of the heat equation with periodic boundary conditions. We will therefore first review their definition and some of their properties.

\subsection{Jacobi theta functions}

In \cite{BatemanII}, p355, the Jacobi theta function $\theta_3(v|\tau)$ is defined for $v\in\C$, $\tau\in\C$ with $\Im \tau > 0$ as 
\begin{equation}
\label{eq:theta3}
\theta_3(v|\tau) = \sum_{n\in\Z} \exp\Big(i \pi (n^2\tau + 2 n v)\Big) = 1 + 2 \sum_{n=1}^\infty e^{i \pi n^2 \tau} \cos(2 \pi n v).
\end{equation}
For our purposes, the following definition will be handier: For $x\in\C$, $t\in\C$ with $\Re t>0$, we define
\begin{equation}
\label{eq:theta}
\begin{split}
\theta(x,t) = \theta_3\Big(\frac x 2\Big|\frac{i \pi t}{2}\Big) &= \sum_{n\in\Z} \exp\Big(-\tfrac{\pi^2}{2} n^2 t + i \pi n x\Big)\\
&= 1 + 2 \sum_{n=1}^\infty e^{- \tfrac{\pi^2}{2} n^2 t} \cos(\pi n x).
\end{split}
\end{equation}

The definition \eqref{eq:theta} is a representation of $\theta$ as a Fourier series, which is particularly well suited for large $t$, but which does not reveal its behaviour as $\Re t\to 0$. This is where the following representation comes in, which is related to \eqref{eq:theta} by the Poisson summation formula (see \cite{Bellman}, §9):
\begin{equation}
\label{eq:theta_gaussian}
\theta(x,t) = \sum_{n\in\Z} \frac{1}{\sqrt{2\pi t}} \exp\Big(-\frac{(x-2n)^2}{2t}\Big).
\end{equation}
One recognizes immediately that for real $x$ and $t$, $\theta(x,t)$ is the probability density at time $t$ of Brownian motion on the circle $\R/2\Z$ started at $0$. In other words, $\theta(x,t)$ is the unique solution to the PDE
\begin{equation*}
\label{eq:pde_theta}
\begin{cases}
\frac{\partial}{\partial t} u(x,t) = \frac 1 2 \left(\frac{\partial}{\partial x}\right)^2 u(x,t) & \text{(PDE)}\\
u(x,t) = u(x+2,t) & \text{(BC)}\\
u(x,0+) = \sum_{n\in\Z} \delta(x-2n) & \text{(IC)},
\end{cases}
\end{equation*}
where $\delta(x)$ denotes the Dirac Delta-function. This is the heat equation with periodic boundary condition and the Dirac comb as initial condition. Note that (PDE) and (BC) also follow directly from \eqref{eq:theta}.

\subsection{Brownian motion killed upon exiting an interval}

Various quantities of Brownian motion killed upon exiting an interval can be expressed by theta functions. For $x\in\R$, let $W^x$ be the law of Brownian motion started at $x$, let $(X_t)_{t\ge0}$ be the canonical process and let $H_y = \inf\{t\ge0: X_t = y\}$. For $a>0$ and $x\in[0,a]$, denote by $W^x_{\text{killed},a}$ the law of Brownian motion started at $x$ and killed upon leaving the interval $(0,a)$. Let $p^a_t(x,y)$ be its transition density, i.e.
\begin{equation}
\label{eq:def_p}
p^a_t(x,y) = W^x_{\text{killed},a}(X_t\in\dd y) = W^x(X_t \in \dd y,\ H_0\wedge H_a < t)/\dd y,\quad x,y\in [0,a].
\end{equation}
Then $p^a_t(x,y)$ is the fundamental solution to the heat equation (PDE) with boundary condition
\[u(0,t) = u(a,t) = 0,\quad t\ge0.\]
Hence (see also \cite{ItoMcKean}, Problem 1.7.8 or \cite{BorodinSalminen}, formula 1.1.15.8),
\begin{equation}
\label{eq:p_theta}
p^a_t(x,y) = a^{-1}\left(\theta\Big(\frac{x-y}{a},\frac t {a^2}\Big) - \theta\Big(\frac{x+y}{a},\frac t {a^2}\Big)\right).
\end{equation}
Equation \eqref{eq:theta} then yields
\begin{equation}
\label{eq:p_sin}
p^a_t(x,y) = \frac{2}{a} \sum_{n=1}^\infty e^{-\frac{\pi^2}{2a^2} n^2 t} \sin(\pi n \tfrac x a) \sin(\pi n \tfrac y a).
\end{equation}
This representation is particularly useful for large $t$: Define 
\begin{equation}
\label{eq:def_E}
 E_t = \sum_{n=2}^{\infty} n^2e^{-\pi^2/2\,(n^2-1) t}.
\end{equation}
By \eqref{eq:p_sin} and the inequality $|\sin nx| \le n \sin x$, $x\in [0,\pi]$, one sees that
\begin{equation}
 \label{eq:p_estimate}
 p^a_t(x,y) = \frac 2 a \sin(\pi x/a) \sin(\pi y/a)e^{-\frac{\pi^2}{2a^2} t}(1+ O(1) E_{t/a^2}).
\end{equation}
Note that the potential kernel is given by
\begin{equation}
\label{eq:p_potential}
 \int_0^\infty p^a_t(x,y)\,\dd t = W^x\Big(\int_0^{H_0\wedge H_a} \Ind_{(X_t\in\dd y)}\,\dd t\Big)/\dd y = 2 a^{-1} (x\wedge y)(a-x\vee y),
\end{equation}
by the formula for the Green function of Brownian motion (see e.g.\ \cite{Kallenberg1997}, Lemma 20.10, p379).

Set $H=H_0\wedge H_a$ and define
\begin{equation}
\label{eq:def_r}
r^a_t(x) = W^x(H \in \dd t,\ X_H = a)/\dd t.
\end{equation}
Then (see \cite{BorodinSalminen}, formula 1.3.0.6),
\begin{equation}
\label{eq:r_theta}
r^a_t(x) = \frac{1}{2a^2} \theta'\left(\frac{x}{a}-1,\frac t {a^2}\right),
\end{equation}
where $\theta'$ denotes the derivative of $\theta$ with respect to $x$.

The following two integrals are going to appear several times throughout the article, which is why we give some useful estimates here. For a measurable subset $S\subset \R$, define
\begin{equation}
 \label{eq:Ia} 
  I^a(x,S) = W^x\Big(e^{\frac{\pi^2}{2a^2}H_a} \Ind_{(H_0 > H_a\in S)}\Big) = \int_{S\cap (0,\infty)} e^{\frac{\pi^2}{2a^2} s} r_s^a(x)\dd s,
\end{equation}
and
\begin{equation}
 \label{eq:Ja}
  J^a(x,y,S) = \int_{S\cap (0,\infty)} e^{\frac{\pi^2}{2a^2} s} p_s^a(x,y)\dd s,
\end{equation}
which satisfy the scaling relations 
\begin{equation}
 \label{eq:IJ_scaling}
 I^a(x,S) = I(x/a,S/a^2),\quad J^a(x,y,S) = a J(x/a,y/a,S/a^2),
\end{equation}
with $I = I^1$ and $J=J^1$. The following lemma provides estimates on $I(x,S)$ and $J(x,y,S)$.

\begin{lemma}
 \label{lem:I_estimates}
There exists a universal constant $C$, such that for every $x\in[0,1]$ and every measurable $S\subset \R_+$, we have
\[
\begin{split}
 |I(x,S) - \pi \lambda(S)\sin(\pi x)| &\le C\Big(1 \wedge E_{\inf S} (1\wedge \lambda(S))\sin(\pi x)\Big),\quad\tand\\
|J(x,y,S) - 2 \lambda(S)\sin(\pi x)\sin(\pi y)| &\le C\Big([(x\wedge y)(1-(x\vee y))] \wedge E_{\inf S} \sin(\pi x) \sin(\pi y)\Big),
\end{split}
\]
where $\lambda(S)$ denotes the Lebesgue measure of $S$ and $E_{\inf S}$ is defined in \eqref{eq:def_E}.
\end{lemma}
\begin{proof}
First note that $I(x,\cdot)$ is a positive measure on $\R_+$ for every $x\in[0,1]$, such that
\[
 0\le I(x,S\cap[0,1]) \le I(x,[0,1]) \le W^x\Big(e^{\frac{\pi^2}{2}H_1} \Ind_{(H_1 \le 1)}\Big) \le e^{\frac{\pi^2}{2}},
\]
by \eqref{eq:Ia}. Furthermore, decomposing $I(x,S)$ into
\[
  I(x,S) = I(x,S\cap[0,1]) + I(x,S\cap (1,\infty)),
\]
it is enough to prove that $|I(x,S) - \pi \lambda(S)\sin(\pi x)| \le C (1\wedge \lambda(S))E_{\inf S} \sin(\pi x)$ for all $S$. Now, by \eqref{eq:r_theta} and \eqref{eq:theta},
\[
\begin{split}
 I(x,S) &= \frac 1 2 \int_S e^{\frac{\pi^2}{2} s} \theta'(x,s) \dd s\\
  &= \pi \int_S \sum_{n=1}^\infty e^{-\frac{\pi^2}{2}(n^2-1)s} (-1)^{n-1} n \sin(\pi n x)\dd s\\
  &= \pi \lambda(S) \sin(\pi x) + \pi \sum_{n=2}^\infty\Big(\int_S e^{-\frac{\pi^2}{2}(n^2-1)s} \dd s\Big) n(-1)^{n-1} \sin(\pi n x),
\end{split}
\]
where the exchange of integral and sum is justified by the uniform convergence of the sum for $s\ge 1$.
We now have for $n\ge 2$,
\[
  \int_S e^{-\frac{\pi^2}{2}(n^2-1)s}\dd s \le \int_{\inf S}^\infty e^{-\frac{\pi^2}{2}(n^2-1)s}\dd s = \frac{2}{\pi^2(n^2-1)} e^{-\frac{\pi^2}{2}(n^2-1)\inf S},
\]
as well as
\[
 \int_S e^{-\frac{\pi^2}{2}(n^2-1)s}\dd s \le \lambda(S)e^{-\frac{\pi^2}{2}(n^2-1)\inf S}
\]
Furthermore, we have for $n\ge 2$,
\[
 |n(-1)^{n-1} \sin(\pi n x)| \le n^2 \sin(\pi x) \le 2 (n^2 - 1) \sin(\pi x).
\]
It follows that
\[
 |I(x,S) - \pi \lambda(S)\sin(\pi x)| \le (\frac{4}{\pi} \wedge \pi \lambda(S)) E_{\inf S} \sin(\pi x).
\]
This proves the statement about $I$. The proof of the statement about $J$ is similar, drawing on \eqref{eq:p_sin} instead and on the following estimate:
\[
 J(x,y,[0,1]) = \int_0^1 e^{\frac{\pi^2}{2} t} p_t(x,y)\,\dd t \le e^{\frac{\pi^2}{2}} \int_0^\infty p_t(x,y)\,\dd t = e^{\frac{\pi^2}{2}} (x\wedge y)(1-(x\vee y)),
\]
by \eqref{eq:p_potential}.
\end{proof}

\subsection{The Brownian taboo process}
\label{sec:taboo}
The Markov process on $(0,a)$ with infinitesimal generator
\[\frac 1 2 \left(\frac \dd {\dd x}\right)^2 + \frac \pi a \cot \frac {\pi x} a \frac \dd {\dd x}\]
is called the \emph{Brownian taboo process} on $(0,a)$. It is a diffusion with scale function $s(x)$ and speed measure $m(\dd x)$, where
\[
s(x) = \frac \pi a \cot \frac {\pi x} a \quad\quad\tand\quad\quad m(\dd x) = \frac {2 a^2} {\pi^2} \sin^2\left(\frac {\pi x} a\right)\ \dd x.
\]
The singular points $0$ and $a$ are therefore entrance-not-exit. For $x\in [0,a]$ we denote the law of the Brownian taboo process on $(0,a)$ started from $x$ by $W^x_{\text{taboo},a}$. Often we will drop the $a$ if its value is clear from the context.

The name of this process was coined by F.\ Knight \cite{Knight1969} who showed that it can be interpreted as Brownian motion conditioned to stay inside the interval $(0,a)$ (hence, $0$ and $a$ are \emph{taboo states}). When $a=\pi$, the Brownian taboo process is also known as the \emph{three-dimensional Legendre process}, because of its relation to Brownian motion on the 3-sphere (see \cite{ItoMcKean}, p270). Readers familiar with the 3-dimensional Bessel process will notice that it can be obtained from the Brownian taboo process as the limit in law when $a\to\infty$. Note that the normalisation of the scale function and speed measure from the last paragraph was chosen in such a way that they converge, respectively, to the scale function and speed measure of the 3-dimensional Bessel process, as $a\to\infty$.

Below we list some useful properties of the Brownian taboo process:
\begin{enumerate}[nolistsep]
 \item It satisfies the following \emph{scaling relation}: If $X_t$ is a Brownian taboo process on $(0,1)$, then $a X_{t/a^2}$ is a Brownian taboo process on $(0,a)$.
 \item It is the Doob transform of Brownian motion killed at $0$ and $a$, with respect to the space-time harmonic function $h(x,t) = \sin(\pi x/a) \exp(\pi^2t/(2a^2))$. In other words, for $x\in(0,a)$, $W^x_{\text{taboo}}$ is obtained from $W^x_{\text{killed}}$ by a Cameron--Martin--Girsanov change of measure with the martingale
\[
 Z_t = \left(\sin\frac {\pi x} a\right)^{-1} \sin\frac {\pi X_t} a \ \exp\frac {\pi^2}{2a^2} t.
\]
 \item As a consequence, its transition probabilities are given by
\begin{equation}
\label{eq:taboo_transition}
 p^{\text{taboo}(0,a)}_t(x,y) = W^x_{\text{taboo},a}\left(X_t \in \dd y\right)/\dd y = \frac {\sin (\pi y/a)} {\sin (\pi x/a)}e^{\frac {\pi^2}{2a^2} t}\, p^a_t(x,y).
\end{equation}
Equation \eqref{eq:p_estimate} now implies that
\begin{equation}
 \label{eq:ptaboo_estimate}
 p^{\text{taboo}(0,a)}_t(x,y) = \frac 2 a \sin^2(\pi y/a)(1+ O(1) E_{t/a^2}),\quad \text{for all }x,y\in[0,a],
\end{equation}
 \item As can be seen from above or directly, it admits the stationary probability measure
\[
(m(0,a))^{-1} m(\dd x) = 2/a \sin^2(\pi x/a)\ \dd x. 
\]
 \item It is self-dual in the sense that for a measurable functional $F$ and $t>0$, we have
\[
 W^{x,t,y}_{\text{taboo}}[F((X_s;0\le s\le t))] = W^{y,t,x}_{\text{taboo}}[F((X_{t-s};0\le s\le t))].
\]
Here $W^{x,t,y}_{\text{taboo}}$ denotes the taboo bridge from $x$ to $y$ of length $t$. This follows from the self-duality of killed Brownian motion.
\end{enumerate}

The following lemma will be needed in Sections \ref{sec:before_breakout} and \ref{sec:moving_barrier}.

\begin{lemma}
\label{lem:k_integral}
Define $k(x)=e^{-c x}$. There exists a constant $C$, depending only on $c$, such that we have for every $x,y\in[0,a]$,
\begin{equation}
 \label{eq:taboo_integral_error_1}
 W^{x}_{\mathrm{taboo}}\Big[\int_0^t k(X_s)\,\dd s\Big] \le C \Big(t/a^3 + \operatorname{err}(x)\Big),
\end{equation}
and for $t\ge a^2$, 
\begin{equation}
 \label{eq:taboo_integral_error_2}
 W^{x,t,y}_{\mathrm{taboo}}\Big[\int_0^t k(X_s)\,\dd s\Big] \le C \Big(t/a^3 + \operatorname{err}(x) + \operatorname{err}(y)\Big),
\end{equation}
with $\operatorname{err}(z) = (1\wedge z^{-1}) + (1+z)e^{-cz}$. If $t\le a^2$, we still have for $x,y\le a/2$,
\begin{equation}
 \label{eq:taboo_integral_error_3}
 W^{x,t,y}_{\mathrm{taboo}}\Big[\int_0^t k(X_s)\,\dd s\Big] \le C.
\end{equation}
\end{lemma}

\begin{proof}
We first show that \eqref{eq:taboo_integral_error_1} implies \eqref{eq:taboo_integral_error_2}. By the self-duality of the taboo bridge, we have
\[
 W^{x\to y}_{\mathrm{taboo}}\Big[\int_0^t k(X_s)\,\dd s\Big] = W^{x,t,y}_{\mathrm{taboo}}\Big[\int_0^{t/2} k(X_s)\,\dd s\Big] + W^{y,t,x}_{\mathrm{taboo}}\Big[\int_0^{t/2} k(X_s)\,\dd s\Big].
\]
It therefore remains to prove that
\[
 E(x,y) = W^{x,t,y}_{\mathrm{taboo}}\Big[\int_0^{t/2} k(X_s)\,\dd s\Big] \le C(t/a^3 + \operatorname{err}(x)).
\]
Conditioning on $\sigma(X_s; 0\le t\le t/2)$, this integral equals
\[
 E(x,y) = W^x_{\mathrm{taboo}}\Big[\frac{p^{\mathrm{taboo}}_{t/2}(X_{t/2},y)}{p^{\mathrm{taboo}}_T(x,y)}\int_0^{t/2} k(X_s)\,\dd s\Big].
\]
By \eqref{eq:ptaboo_estimate}, there exists a universal constant $C$, such that for $t\ge a^2$,
\[
 E(x,y) \le C\, W^x_{\mathrm{taboo}}\Big[\int_0^{t/2} k(X_s)\,\dd s\Big].
\]
Equation \eqref{eq:taboo_integral_error_1} therefore implies \eqref{eq:taboo_integral_error_2}.

Heuristically, one can estimate the left side of \eqref{eq:taboo_integral_error_1} in the following way: Since $k(x)$ is decreasing very fast, only the times at which $X_s$ is of order $1$ contribute to the integral. When started from the stationary distribution, the process takes a time of order $a^3$ to reach a point at distance $O(1)$ from $0$ \cite{Lambert2000} and it stays there for a time of order $1$, hence the integral is of order $t/a^3$. When started from the point $x$, an additional error is added, which is of order $1$, when $x$ is at distance of order $1$ away from $0$. Adding both terms gives the bound appearing in the statement of the lemma.

The exact calculations are most easily performed in the following way. Let $Y$ be a random variable with values in $(0,a)$ distributed according to $\widetilde{m}(\dd x) := 2/a \sin^2(\pi x/a)\,\dd x$, which is the stationary probability measure of the taboo process. Let $H_Y = \inf\{t>0: X_s = Y\}$. We then have
\[
\begin{split}
 W^x_{\mathrm{taboo}}\Big[\int_0^t k(X_s)\,\dd s\Big] &= W^x_{\mathrm{taboo}}\Big[\int_0^{H_Y} k(X_s)\,\dd s+\int_{H_Y}^t k(X_s)\,\dd s\Big]\\
&\le W^x_{\mathrm{taboo}}\Big[\int_0^{H_Y} k(X_s)\,\dd s\Big] + W^{\widetilde m}_{\mathrm{taboo}}\Big[\int_0^t k(X_s)\,\dd s\Big]\\
& =: I_1+I_2.
\end{split}
\]
The second term is simply equal to
\[
 I_2 = t \int_0^a \widetilde{m}(\dd y)\, k(y)\,\dd y \le 2 \pi^2 t/a^3 \int_0^\infty e^{-c y}(1+y) y^2\,\dd y \le CT/a^3,
\]
for some constant $C$ depending only on $c$.

The first term is equal to (see e.g.\ \cite{Revuz1999} Chapter 3, Corollary 3.8)
\begin{equation}
 \label{eq:proof_k_integral_1}
\begin{split}
 I_1 &= \int_0^x \widetilde{m}(\dd y) \int_y^a m(\dd z) G_{y,a}(x,z) k(z) + \int_x^a \widetilde{m}(\dd y) \int_0^y m(\dd z) G_{0,y}(x,z) k(z)\\
 &=: I_{11}+I_{12},
\end{split}
\end{equation}
where the Green functions are defined by
\[
 G_{y,a}(x,z) = s(x\wedge z) - s(y)\quad\tand\quad G_{0,y}(x,z) = s(y) - s(x\vee z).
\]
By Fubini's theorem, the first term in \eqref{eq:proof_k_integral_1} is easily bounded by
\[
 I_{11} \le \int_0^a m(\dd z) k(z) \int_0^z \widetilde{m}(\dd y) [s(z) - s(y)],
\]
and noticing that $\operatorname{sign}(s(z)) = -\Ind_{(z<a/2)} + \Ind_{(z>a/2)}$, we get
\[
\begin{split}
 I_{11} &\le \int_{a/2}^a m(\dd z) s(z) k(z)\int_0^z \widetilde{m}(\dd y) + \int_0^a m(\dd z) k(z) \int_0^z \widetilde{m}(\dd y) (-s(y))\\
 &\le C/a^3\Big(\int_{a/2}^a z^4 k(z)\,\dd z + \int_0^a z^4 k(z)\,\dd z\Big)\\
 &\le C/a^3,
\end{split}
\]
where again we made use of the inequality $\sin x \le x$ for $x\in[0,\pi]$.

For the term $I_{12}$ a little bit more care is needed. Using the fact that $\int_x^a \widetilde{m}(\dd y) \le 1$, we have
\[
\begin{split}
 I_{12} &\le \int_x^a \widetilde{m}(\dd y) s(y) \int_0^y m(\dd z) k(z) + (-s(x)\vee 0) \int_0^x m(\dd z) k(z)+ \int_x^{a/2} m(\dd z) |s(z)| k(z)\\
&=: I_{121} + I_{122} + I_{123}.
\end{split}
\]
To estimate the first two terms, note that
\[
 \int_0^y m(\dd z) k(z) \le C(1\wedge y^3),\quad\tand\quad\int_x^a \widetilde{m}(\dd y) s(y) \le C/a.
\]
such that
\[
 I_{121} + I_{122} \le C\Big(1/a + (1\wedge x^3)(-s(x)\vee 0)\Big) \le C\Big(1/a + (1\wedge x^{-1})\Big),
\]
because $-s(x) \le 1/x$ for $x\in[0,a]$. The third term is seen to be bounded by
\[
 I_{123} \le C \int_x^\infty z k(z)\,\dd z \le C (1+x)e^{-cx}.
\]
Altogether, we get
\[
 W^x_{\mathrm{taboo}}\Big[\int_0^t k(X_s)\,\dd s\Big] \le C\Big(t/a^3 + 1/a + \operatorname{err}(x)\Big),
\]
and the $1/a$ term can be dropped, because $t \ge a^2$ by hypothesis. This proves \eqref{eq:taboo_integral_error_2} and therefore \eqref{eq:taboo_integral_error_1}.

When $t\le a^2$ a different method of proof is needed. First we note that for $0 < x,y \le a/2$, the transition density of the taboo bridge can be written
\[
 W^{x,t,y}_{\mathrm{taboo}}\Big[X_s\in \dd z\Big] = \frac{p_s^a(x,z)p_{t-s}^a(z,y)}{p_t^a(x,y)}\,\dd z.
\]
If we denote by $p_t^0(x,y) = (2\pi t)^{-1/2} \exp(-(z^2+x^2)/2t)2\sinh(zx/t)$ the transition density of Brownian motion killed at $0$, then we have the trivial inequality $p_t^a(x,y) \le p_t^0(x,y)$ and furthermore by scaling we see that $p_t^a(x,y) \ge C p_t^0(x,y)$, since $x,y\le a/2$ and $t\le a^2$. It follows that
\[
 W^{x,t,y}_{\mathrm{taboo}}\Big[\int_0^t k(X_s)\,\dd s\Big] \le C R^{x,t,y}\Big[\int_0^t k(X_s)\,\dd s\Big],
\]
where $R^{x,t,y}$ denotes the law of the Bessel bridge of dimension 3. This Bessel bridge is the Doob transform of the Bessel process started at $x$ with respect to the space-time harmonic function $h_y(z,s) = p_{t-s}^0(z,y)/p_t^0(x,y)$. By the standard theory of Doob transforms, this is the Bessel process with additional drift
\[
 \frac{\dd}{\dd z}(\log h_y(z,s)) = -\frac{z^2}{t-s}+\frac{\dd}{\dd z}\log \sinh\frac{zy}{t-s} = -\frac{z^2}{t-s} + \frac{y}{t-s} \coth \frac{zy}{t-s}.
\]
Now, this in an increasing function in $y$, and standard comparison theorems for diffusions (see e.g.\ \cite{Revuz1999}, Theorem IX.3.7) now yield that for $y_1\le y_2$, we have
\[
 R^{x,t,y_2}[k(X_s)] \le R^{x,t,y_1}[k(X_s)],
\]
since $k$ is a decreasing function. This is true in particular for $y_1 = 0$. Using the self-duality of the Bessel bridge, we can repeat the same reasoning with $x$. We thus have altogether
\[
 W^{x,t,y}_{\mathrm{taboo}}\Big[\int_0^t k(X_s)\,\dd s\Big] \le C R^{0,t,0}\Big[\int_0^t k(X_s)\,\dd s\Big],
\]
for any $x,y\le a/2$. This calculation can be done explicitly and yields \eqref{eq:taboo_integral_error_3}.
\end{proof}

\section{Preliminaries on branching Markov processes}
\label{sec:preliminaries}

In this section we recall some known results about branching Brownian motion and branching Markov processes in general.

\subsection{Definition and notation}
\label{sec:preliminaries_definition}

Branching Brownian motion can be formally defined using Neveu's \emph{marked trees} \cite{Neveu1986} as in \cite{Chauvin1988} and \cite{Chauvin1991}. We will follow this path here, but with slight differences, because we will need to consider more general branching Markov processes and the definition of branching Brownian motion in \cite{Chauvin1991} formally relied on the translational invariance of Brownian motion. 

We first define the space of Ulam--Harris labels, or \emph{individuals},
\[
 U=\{\emptyset\} \cup \bigcup_{n\ge 1} {\N^*}^n,
\]
where we use the notation $\N^* = \{1,2,3,\ldots\}$ and $\N = \{0\}\cup \N^*$. This space is endowed with the ordering relations $\preceq$ and $\prec$ defined by
\[
 u \preceq v \iff \exists w\in U: v = uw\quad\tand\quad u\prec v \iff u\preceq v \tand u \ne v.
\]
A \emph{tree} is by definition a subset $\mathfrak t\subset U$, such that $\emptyset \in \mathfrak t$, $v\in \mathfrak t$ if $v\prec u$ and $u\in \mathfrak t$ and for every $u$ there is a number $k_u\in\N$, such that for all $j\in\N^*$, we have $uj\in \mathfrak t$ if and only if $j\le k_u$. Thus, $k_u$ is the number of children of the individual $u$. We denote the space of trees by $\mathscr T$ and endow it with the sigma-field $\mathscr A$ generated by the subsets $\mathscr T_u = \{\mathfrak t\in\mathscr T: u\in \mathfrak t\}$.

For a tree $\mathfrak t\in \mathscr T$ and $u\in \mathfrak t$, we define the \emph{subtree rooted at $u$} by
\[
 \mathfrak t^{(u)} = \{v\in U: uv \in \mathfrak t\}.
\]

Given a measurable space $\mathscr M$, a \emph{marked tree} (with space of marks $\mathscr M$) is a pair
\[
 \mathfrak t^{\mathscr M} = (\mathfrak t,(\eta_u;u\in\mathfrak t)),
\]
where $\mathfrak t\in \mathscr T$ and $\eta_u\in \mathscr M$ for all $u\in\mathfrak t$. The space of marked trees is denoted by $\mathscr T^\mathscr M$, and is endowed with the sigma-field $\mathscr A^{\mathscr M} = \pi^{-1}(\mathscr A)$, where $\pi: \mathscr T^\mathscr M\to \mathscr T$ is the canonical projection. Accordingly, we also define $\mathscr T^{\mathscr M}_u = \pi^{-1}(\mathscr T_u)$. The definition of a subtree extends as well to marked trees: For $u\in\mathfrak t$, we define
\[
 (\mathfrak t^{\mathscr M})^{(u)} = (\mathfrak t^{(u)},(\eta_{uv};v\in\mathfrak t^{(u)})).
\]

For our purposes, the space of marks $\mathscr M$ is always going to be a function space, namely, for a Polish space $\mathscr E$ and a cemetary symbol $\Delta \notin \mathscr E$, we define the Skorokhod space $D(\mathscr E)$ of functions $\Xi:[0,\infty)\to\mathscr E\cup \{\Delta\}$ which are right-continuous with left limits, with $\Xi(0)\ne \Delta$ and for which $\Xi(t) = \Delta$ implies $\Xi(s) = \Delta$ for all $s\ge t$. Then we define $\zeta(\Xi) = \inf\{t\ge 0: \Xi(t) = \Delta\}$. For an individual $u\in U$, its mark is denoted by $\Xi_u$ and we define $\zeta_u = \zeta(\Xi_u)$. The branching Markov process will then be defined on the space (we suppress the superscript $D(\mathscr E)$)
\[
 \Omega = \{\omega = (\mathfrak t,(\Xi_u;u\in\mathfrak t))\in\mathscr T^{D(\mathscr E)}: \forall u\in U\ \forall 1\le i\le k_u: \zeta_u < \infty \Rightarrow \Xi_u(\zeta_u-) = \Xi_{ui}(0)\},
\]
endowed with the sigma-field $\F = \Omega \cap \mathscr A^{D(\mathscr E)}$ generated by the sets $\Omega_u = \Omega \cap \mathscr T^{\mathscr M}_u$. We define for $u\in U$ the random variables
\[
 b_u = \sum_{v\prec u} \zeta_v,\quad d_u = b_u + \zeta_u = \sum_{v\preceq u} \zeta_v,
\]
which are the \emph{birth} and \emph{death} times of the individual $u$, respectively. We then define the set of individuals alive at time $t$ by 
\[
 \mathscr N(t) = \{u\in \mathfrak t: b_u\le t < d_u\}.
\]
The \emph{position} of $u$ at time $t$ is defined for $u\in \mathfrak t$ by
\[
 X_u(t) = \begin{cases}
           \Xi_v(t-b_v), &\tif v\in \mathscr N(t)\tand v\preceq u\\
           \Delta, &\tif d_u \le t.
          \end{cases}
\]

Now suppose we are given a defective strong Markov process $\overline{X} = (\overline{X}_t)_{t\ge 0}$ on $\mathscr E$, with paths in $D(\mathscr E)$. The law of $\overline{X}$ started in $x\in \mathscr E$ will be denoted by $\overline{P}^x$. For simplicity, we will assume that for every $x\in \mathscr E$, we have $\zeta(\overline{X}) < \infty$, $\overline{P}^x$-almost surely. Furthermore, let $((q(x,k))_{k\in \N})_{x\in\mathscr E}$ be a family of probability measures on $\N$, measurable with respect to $x$. Then we define the branching Markov process with particle motion $\overline{X}$ and reproduction law $q$ as the (unique) family of probability measures $(\P^x)_{x\in\mathscr E}$ on $\Omega$ which satisfies
\begin{equation}
 \label{eq:bmp_renewal}
\P^x(\dd \omega) = \overline{P}^x(\dd X_\emptyset) q(X_\emptyset(\zeta_\emptyset-),k_\emptyset) \prod_{i=1}^{k_\emptyset} \P^{X_\emptyset(\zeta_\emptyset-)}(\dd \omega^{(i)}).
\end{equation}
Note that by looking at the space-time process $(X_t,t)_{t\ge 0}$, we can (and will) extend this definition to the time-inhomogeneous case.

\subsection{Stopping lines}
\label{sec:stopping_lines}

The analogon to stopping times for branching Markov processes are \emph{(optional) stopping lines}, for which several definitions exist. For branching Brownian motion, they have first been defined by Chauvin \cite{Chauvin1991}. The definition we are giving below is equivalent to the definition there, although there are formal differences. Note that Jagers \cite{Jagers1989} has given a more general definition of stopping lines for discrete-time branching processes, and our definition of stopping lines is partly inspired by the exposition there. Note also that Biggins and Kyprianou \cite{Biggins2004} build up on Jagers' definition of stopping lines and define the subclasses of \emph{simple} and \emph{very simple} stopping lines (again for discrete-time processes). Chauvin's definition (and therefore ours as well) then corresponds to the class of very simple stopping lines.

We first define a \emph{(random) line} to be a set $\ell = \ell(\omega) \subset U\times [0,\infty)$, such that
\begin{enumerate}[nolistsep]
 \item $u\in \mathscr N(t)$ for all $(u,t)\in \ell$, and
 \item $(u,t)\in\ell$ implies $(v,s)\notin \ell$ for all $v\preceq u$ and $s< t$.
\end{enumerate}
Note that a line is at most a countable set. For a pair $(u,t)\in U\times [0,\infty)$ and a line $\ell$, we write $\ell \preceq (u,t)$ if there exists $(v,s)\in \ell$, such that $v\preceq u$ and $s\le t$. For a subset $A\subset U\times [0,\infty)$, we write $\ell \preceq A$ if $\ell \preceq (u,t)$ for all $(u,t)\in A$. If $\ell_1$ and $\ell_2$ are two lines, we define the line $\ell_1 \wedge \ell_2$ to be the maximal line (with respect to $\preceq$), which is smaller than both lines.

We now define for each $u\in U$ two filtrations on $\Omega_u$ by
\begin{align*}
 \F_u(t) &= (\Omega_u \cap \sigma(\Xi_u(s);0\le s\le t-b_u)) \vee \bigvee_{v\prec u}(\Omega_v \cap \sigma(\Xi_v))\\ \F^{\mathrm{pre}}_u(t) &= (\Omega_u \cap \sigma(\Xi_u(s);0\le s\le t-b_u))\vee \bigvee_{v\nsucceq u} (\Omega_v \cap \sigma(\Xi_v)).
\end{align*}
Informally, $\F_u(t)$ contains the information on the path from $u$ to the root between the times $0$ and $t$, and $\F^{\mathrm{pre}}_u(t)$ contains this information and the one concerning the descendants of $u$ after the time $t$. In particular, we have $\F_u(t) \subset \F^{\mathrm{pre}}_u(t)$ The filtration $\F_u(t)$ is denoted by $\mathscr A_u(t)$ in Chauvin's paper \cite{Chauvin1991}, and $\F^{\mathrm{pre}}_u(t)$ corresponds to the \emph{pre-$(u,t)$-sigma-algebra} as defined by Jagers \cite{Jagers1989}.

We can now define a \emph{stopping line} $\mathscr L$ to be a random line with the additional property
\begin{itemize}
 \item[3.] $\forall (u,t)\in U\times [0,\infty):\{\omega\in \Omega_u:\mathscr L\preceq(u,t)\} \in \F_u(t).$
\end{itemize}
The sigma-algebra $\F_{\mathscr L}$ of the past of $\mathscr L$ is defined to be the set of events $E \in \F$, such that for all $(u,t)\in U\times [0,\infty)$,
\[
 E \cap \{\omega \in \Omega_u : \mathscr L\preceq(u,t)\} \in \F^{\mathrm{pre}}_u(t).
\]
For example, for any $t\ge0$, the set $\mathscr N(t)\times\{t\}$ is a stopping line. If $T = T(\overline{X})$ is a stopping time for the strong Markov process $\overline{X}$, then
\[
 \mathscr L_T = \{(u,t)\in U\times [0,\infty): u\in \mathscr N(t)\tand t = T(X_u)\}\}
\]
is a stopping line as well.

The first important property of stopping lines is the strong branching property. In order to state it, we define for $t\ge0$, $u\in \mathscr N(t)$,
\[
 \omega^{(u,t)} = (\mathfrak t^{(u)},(\Xi'_{uv};v\in\mathfrak t^{(u)})),
\]
with $\Xi'_u(\cdot) = \Xi_u(\cdot+t-b_u)$ and $\Xi'_{uv} = \Xi_{uv}$ for $v\in \mathfrak t^{(u)}\backslash \{\emptyset\}$. The \emph{strong branching property} (\cite{Chauvin1991}, \cite{Jagers1989}) then states that for every stopping line $\mathscr L$, conditioned on $\F_\mathscr L$, the subtrees $\omega^{(u,t)}$, for $(u,t)\in \mathscr L$, are independent with respective distributions $\P^{X_u(t)}$.

\subsection{Many-to-few lemmas and spines}
\label{sec:spine}

Another important tool in the theory of branching processes is the so-called Many-to-one lemma, and its recently published extension, the Many-to-few lemma \cite{Harris2011} along with the \emph{spine decomposition} technique which comes along with it and has its origins in \cite{LPP1995}. Here we state stopping line versions of these lemmas, which to the knowledge of the author have not yet been stated in this generality in the literature, although they belong to the common folklore. We will therefore only sketch how they can be derived from the existing literature.

We assume for simplicity that the strong Markov process $\overline{X}$ admits a representation as a conservative strong Markov process $X$ with paths in $D(\mathscr E)$, which is killed at a rate $R(x)$, where $R:\mathscr E\to [0,\infty)$ is measurable. The law of $X$ started at $x$ is denoted by $P^x$ and the time of killing by $\zeta$. Given a stopping time $T$ for $X$, we can then define a stopping time $\overline T$ for $\overline X$ by setting $\overline T = T$, if $T < \zeta$ and $\overline T = \infty$ otherwise. For simplicity, we write $\mathscr L_T$ for $\mathscr L_{\overline T}$. Finally, for every $x\in\mathscr E$, define $m(x) = \sum_{k\ge 0} (k-1) q(x,k)$, $m_1(x) = \sum_{k\ge 0} k q(x,k)$ and $m_2(x) = \sum_{k\ge 0} k(k-1) q(x,k)$.

We are now going to present the spine decomposition technique, following \cite{Hardy2006}. They assume that $q(x,0) \equiv 0$, but this restriction is actually not necessary, as noted in \cite{Harris2011}. Given a tree $\mathfrak t$, a \emph{spine} of $\mathfrak t$ is formally an element of the \emph{boundary} of $\mathfrak t$, i.e.\ it is a line of descent $\xi = (\xi_0 = \emptyset,\xi_1,\xi_2,\ldots)$ from the tree, which is finite if and only if the last element is a leaf of the tree. We augment our space $\Omega$ to the space $\Omega^*$ by
\[
 \Omega^* = \{(\omega,\xi): \omega\in\Omega,\ \xi \text{ is a spine of the tree underlying $\omega$}\}
\]
We are going to denote by $\xi_t$ the individual $u\in U$ that satisfies $u \in \mathscr N(t)$ and $u\in \xi$, if it exists\footnote{If $R(x)$ is bounded from above, which will always be the case in this paper, this individual exists with probability one.}, and $\xi_t = \emptyset$ otherwise. Instead of writing $X_{\xi_t}(t)$, we are going to write for short $X_\xi(t)$. We also note that the definition of stopping lines can be extended to $\Omega^*$ by projection.

Now, for every $x\in \mathscr E$, one can define a probability measure $\P_x^*$ on $\Omega^*$ in the following way:
\begin{itemize}[nolistsep, label=$-$]
\item Initially, $X_\xi(0) = x$.
\item The individuals on the spine move according to the strong Markov process $X$ and die at the rate $m_1(y)R(y)$, when at the point $y\in\mathscr E$.
\item When an individual on the spine dies at the point $y\in\mathscr E$, it leaves $k$ offspring at the point where it has died, with probability $(m_1(x))^{-1} kq(x,\cdot)$ (this is also called the size-biased distribution of $q(x,\cdot)$\footnote{The size-biased distribution of the Dirac-mass at $0$ is again the Dirac-mass at $0$}).
\item Amongst those offspring, the next individual on the spine is chosen uniformly. This individual repeats the behaviour of its parent (started at the point $y$).
\item The other offspring initiate independent branching Markov processes according to the law $\P^y$, independently of the spine.
\end{itemize}
This decomposition first appeared in \cite{Chauvin1988}. We now have

\begin{lemma}[Many-to-one]
\label{lem:many_to_one}
 Let $\mathscr L$ be a stopping line, such that $\P_x^*$-almost surely, there exists $t\ge 0$, such that $(\xi_t,t)\in \mathscr L$. Denote this time by $T$. Let $Y$ be a random variable of the form
\[
 Y = \sum_{(u,t)\in \mathscr L} Y_u \Ind_{(u\in \xi)},
\]
where $Y_u$ an $\F_\mathscr L$-measurable random variable for every $u\in U$. Then
\begin{equation}
 \E_x^*\Big[Ye^{\int_0^T R(X_\xi(t))m(X_\xi(t))\,\dd t}\Big] = \E^x\Big[\sum_{(u,t)\in \mathscr L} Y_u\Big].
\end{equation}
\end{lemma}

Proofs of this result can be found for fixed time in \cite{Kyprianou2004}, \cite{Hardy2006} or \cite{Harris2011}. The proofs in \cite{Hardy2006} and \cite{Harris2011} can be extended to stopping lines once the martingales that appear in the proof are still uniformly integrable when stopped at the stopping line $\mathscr L$. Adapting the arguments of \cite{Kyprianou2000} or \cite{Biggins2004} to the continuous-time setting, one sees that this is true by the hypothesis we have placed on $\mathscr L$. This hypothesis is also referred to as the stopping line $\mathscr L$ being \emph{dissecting}.

Often, we will use a simpler version of the Many-to-one lemma, which is the following
\begin{lemma}[Simple Many-to-one]
\label{lem:many_to_one_simple}
 Let $T = T(X)$ be a stopping time for the strong Markov process $X$ which satisfies $P^x(T<\infty) = 1$ for every $x\in\mathscr E$. Let $f:\mathscr E\to [0,\infty)$ be measurable. Then we have
\[
 \E^x\Big[\sum_{(u,t)\in \mathscr L_T} f(X_u(t))\Big] = E^x\Big[e^{\int_0^T R(X_t) m(X_t)\,\dd t} f(X_T)\Big]
\]
\end{lemma}

The next lemma tells us about second moments of sums of the previous type. To state it, we define for a stopping time $T$ for $X$, the density kernel of the branching Markov process before $\mathscr L_T$, by
\begin{equation}
\label{eq:def_density}
 p_T(x,\dd y,t) = \E^x\Big[\sum_{u\in \mathscr N(t)} \Ind_{(X_u(t)\in \dd y,\ t < T(X_u))}\Big].
\end{equation} 
\begin{lemma}
\label{lem:many_to_two}
 Let $H$ be the hitting time functional of a closed set $F\subset \mathscr E$ on $D(\mathscr E)$ which satisfies $P^x(H<\infty) = 1$ for every $x\in\mathscr E$. Let $f:\mathscr E\to [0,\infty)$ be measurable. Then we have
\begin{multline}
\label{eq:many_to_two}
 \E^x\Big[\Big(\sum_{(u,t)\in \mathscr L_H} f(X_u(t))\Big)^2\Big] = \E^x\Big[\sum_{(u,t)\in \mathscr L_H} (f(X_u(t)))^2\Big]\\
 + \int_0^\infty \int_{\mathscr E} p_H(x,\dd y,t) R(y)m_2(y) \Big(\E^y\Big[\sum_{(u,t)\in \mathscr L_H} f(X_u(t))\Big]\Big)^2\,\dd t
\end{multline}
\end{lemma}
This lemma can be proven using the Many-to-few lemma from \cite{Harris2011} (which is valid for stopping lines as well by the same argument as the one above) or with Lemma \ref{lem:many_to_one}, by noting that
\[
 \Big(\sum_{(u,t)\in \mathscr L_H} f(X_u(t))\Big)^2 = \sum_{(u,t)\in \mathscr L_H} (f(X_u(t)))^2 + \sum_{(u,t)\in \mathscr L_H} \Big(f(X_u(t)) \sum_{(v,s)\in \mathscr L_H,\ v\ne u} f(X_v(s))\Big).
\]
For an intuitive explanation of the terms appearing in \eqref{eq:many_to_two}, see the proof of Proposition 18 in \cite{Berestycki2010}.

Taking for $X$ the space-time process $(Y_t,t)_{t\ge 0}$ of a possibly non-homogeneous strong Markov process $(Y_t)_{t\ge 0}$ with paths in $D(\mathscr E)$ and the closed set $F = \mathscr E\times \{t\}$, for some $t\ge 0$, we obtain the following useful corollary, which appeared already in \cite{Sawyer1976} and \cite{Watanabe1967} in the homogeneous case.
\begin{lemma}
\label{lem:many_to_two_fixedtime}
Let $f:\mathscr E\times \R_+\to [0,\infty)$ be measurable and let $t\ge 0$. Then we have
\begin{multline}
\label{eq:many_to_two_fixedtime}
 \E^{(x,0)}\Big[\Big(\sum_{u\in \mathscr N(t)} f(Y_u(t),t)\Big)^2\Big] = \E^{(x,0)}\Big[\sum_{u\in \mathscr N(t)} (f(Y_u(t),t))^2\Big]\\
 + \int_0^t \int_{\mathscr E} p(x,\dd y,s) R(y,t)m_2(y,t) \Big(\E^{(y,s)}\Big[\sum_{u\in \mathscr N(t)} f(Y_u(t),t)\Big]\Big)^2\,\dd s
\end{multline}
\end{lemma}

\subsection{Doob transforms}
\label{sec:doob}
As in the previous subsection, we assume for simplicity that the strong Markov process $\overline{X}$ admits a representation as a conservative strong Markov process $X$ with paths in $D(\mathscr E)$, which is killed at a rate $R(x)$, where $R:\mathscr E\to [0,\infty)$ is measurable. Let $H$ be the hitting time functional of a closed set $F\subset \mathscr E$ on $D(\mathscr E)$. Furthermore, let $h:F\to [0,1]$ be a measurable function. We can then extend the function $h(x)$ to $\mathscr E$ by setting
\[
 h(x) = \E^x\Big[\prod_{(u,t)\in \mathscr L_H} h(X_u(t))\Big],
\]
%
We are going to assume that $h(x) > 0$ for all $x\in \mathscr E\backslash F$. Then for all such $x$ we can define a law $\P_h^x$ on $\Omega$ by
\[
 \P_h^x(\dd \omega) = (h(x))^{-1} \prod_{(u,t)\in \mathscr L_H} h(X_u(t))\times \P^x(\dd \omega),
\]
where the multiplication is in the sense of a Radon--Nikodym derivative. Now define
\[
 Q(x) = \sum_{k\ge 0} q(x)h(x)^{k-1},\quad\tand\quad q_h(x,k) = \frac{q(x)h(x)^{k-1}}{Q(x)}.
\]
By \eqref{eq:bmp_renewal}, we now have (dropping the symbol $\emptyset$ for better reading and setting $H = H(X_\emptyset)$)
\begin{multline*}
 h(x) \P_h^x(\dd \omega) = \overline{P}^x(\dd X)
\Big(
    \Ind_{(H < \zeta)}
	h(X(H))q(X(H),k)
	\prod_{i=1}^{k} \P^{X(H)}(\dd \omega^{(i)})\\
    +\Ind_{(\zeta \le H)}
	h(X(\zeta-))^{k}q(X(\zeta-),k)
	\prod_{i=1}^{k} \P^{X(\zeta-)}(\dd \omega^{(i)})
	\prod_{(u,t)\in \mathscr L_H(\omega^{(i)})}h(X_u(t))
\Big).
\end{multline*}
If we denote by $X^H$ the process $X$ stopped at $H$, and the law of $X^H$ under $\overline{P}^x$ by $(\overline{P}^x)^H$, then the last equation and the strong Markov property give
\begin{equation}
\label{eq:bmp_conditioned_renewal}
 \begin{split}
h(x) \P_h^x(\dd \omega) = (\overline{P}^x)^H(\dd X^H)
  \Big(
    \Ind_{(H < \zeta)} h(X(H))
  + \Ind_{(\zeta \le H)} h(X(\zeta-))Q(X(\zeta-))
  \Big)\\
\times
  \Big(
    \Ind_{(H < \zeta)} \P^{X(H)}(\dd \omega^{(\emptyset,H)})
    +\Ind_{(\zeta \le H)}
	q_h(X(\zeta-),k)
	\prod_{i=1}^{k} \P_h^{X(\zeta-)}(\dd \omega^{(i)})
  \Big).
 \end{split}
\end{equation}
In particular, integrating over $k$, $\omega^{(i)}$, $i=1,2,\ldots$, and $X_\emptyset(t)$ for $t\in [H,\zeta)$, we get that
\[
 h(x) = (\overline{E}^x)^H\Big(\Ind_{(H < \zeta)} h(X(H)) + \Ind_{(\zeta \le H)} h(X(\zeta-))Q(X(\zeta-))\Big).
\]
We can therefore define a law $\overline{P}_h^x$ on the paths in $D(\mathscr E)$ stopped at $H$ by
\[
 \overline{P}_h^x(\dd X) = (h(x))^{-1} \Big(
    \Ind_{(H(X) < \zeta)} h(X(H))
  + \Ind_{(\zeta \le H(X))} h(X(\zeta-))Q(X(\zeta-))
  \Big)\times (\overline{P}^x)^H(\dd X),
\]
where the multiplication is again in the sense of a Radon--Nikodym derivative. Then \eqref{eq:bmp_conditioned_renewal} yields the following decomposition of the law $\P_h^x$:
\begin{itemize}[nolistsep,label=$-$]
 \item As long as a particle has not hit the set $F$ yet, it moves according to the law $\overline{P}_h^x$, and, when it gets killed at the point $y$, spawns $k$ offspring according to the law $q_h(y,\cdot)$, which initiate independent branching Markov processes according to the law $\P_h^y$.
 \item When a particle hits the set $F$ at the point $y$, it continues as a branching Markov process according to the law $\P^y$.
\end{itemize}
If $R(x) \equiv R$, one gets a simpler characterization of the law $\overline{P}_h^x$: In this case, $h(x)$ is a harmonic function for the law of the stopped process $X^H$ under $P^x$, whence we can define the Doob transform
\[
 P_h^x(\dd X) = (h(x))^{-1} \Big(\Ind_{(H = \infty)} + \Ind_{(H < \infty)} h(X(H))\Big)P^x(\dd X^H).
\]
Then the law $\overline{P}_h^x$ is obtained from the law $P_h^x$ by killing the process at the time-dependent rate $R Q(x)\Ind_{(t < H)}$.

\section{Branching Brownian motion with absorption at a critical line}
\label{sec:critical_line}
From this section on, $q(k)$ will denote a law on $\{0,1,2,\ldots\}$. We define $m = \sum_k (k-1)q(k)$ and $m_2 = \sum_k k(k-1)q(k)$ and suppose that $m>0$ and $m_2 < \infty$. We let $\sigma > 0$ and define $c_0 = \sigma\sqrt{2m}$. In this section we are studying the branching Markov process where, starting with a single particle at the origin, particles move according to Brownian motion with variance $\sigma^2>0$ and drift $-c_0$ and branch at rate 1 into $k$ particles according to the reproduction law $q(k)$. At the point $-x$, we add an \emph{absorbing barrier} to the process, i.e.\ particles hitting this barrier are instantly killed. Formally, we are considering the process up to the stopping line $\mathscr L_{H_{-y}}$, where $H_{-y}$ is the hitting time functional of the point $-y$. It is well-known since Kesten \cite{Kesten1978} that this process gets extinct almost surely. As a consequence, the number of particles absorbed at the barrier, i.e.\ the random variable
\[
Z_y = \#\mathscr L_{H_{-y}},
\]
is almost surely finite. By the strong branching property and the translational invariance of Brownian motion, one sees that the process $(Z_y)_{y\ge 0}$ is a continuous-time Galton--Watson process, a fact which was first noticed by Neveu \cite{Neveu1988} (see \cite{Athreya1972}, Chapter III or \cite{Harris1963}, Chapter V for an introduction to continuous-time Galton--Watson processes). Neveu also stated that the infinitesimal generating function $u(s)$ of this process has the representation $u = \psi'\circ \psi^{-1}$, where $\psi$ is a so-called traveling wave of the FKPP (Fisher--Kolmogorov--Petrovskii--Piskounov) equation: Write $f(s) = \sum_k s^k q(k)$. Then $\psi$ is a solution of the equation
\begin{equation}
 \label{eq:traveling_wave}
 \frac {\sigma^2} 2 \psi'' - c_0\psi' = \psi - f\circ \psi,
\end{equation}
with $\psi(-\infty) = 1$ and $\psi(+\infty)$ is the extinction probability of the process, i.e.\ the smaller root of $f(s) = s$. For a proof of these results, see \cite{Maillard2010}, Section 3.

In the same paper \cite{Neveu1988}, Neveu introduced his \emph{multiplicative martingales}, which he used to derive the Seneta-Heyde norming for the martingale $e^{-c_0 y}Z_y$. He proved that in the case of binary branching, one has
\begin{equation}
 \label{eq:def_Wy}
 W_y := c_0 y e^{-c_0 y} Z_y \to W,
\end{equation}
as $y\to\infty$, where $W > 0$ almost surely. His proof relied on a known asymptotic for the traveling wave $\psi$, namely that
\begin{equation}
 \label{eq:tw_asymptotics}
 1-\psi(-x) \sim K x e^{-c_0 x},\quad\tas x \to \infty,
\end{equation}
for some constant $K> 0$.
It was recently shown \cite{Yang2011} that this asymptotic is true if and only if $E[L\log^2 L] <\infty$ and the proof of \eqref{eq:def_Wy} works in this case as well. We also still have in this case, for every $x\in \R$,
\begin{equation}
 \label{eq:W_laplace}
 E[e^{-e^{c_0 x}W}] = \psi(x),
\end{equation}
a fact which was already proven by Neveu \cite{Neveu1988} for dyadic branching.

In \cite{Berestycki2010}, further properties of the limit $W$ have been established under the hypothesis of dyadic branching, namely
\begin{equation}
\label{eq:W_tail}
 \P(W > x) \sim \frac{1}{x},\quad \tas x\to\infty,
\end{equation}
and
\begin{equation}
 \label{eq:W_expec}
 \E[W\Ind_{(W\le x)}] - \log x \to \const{eq:W_expec},\quad \tas x\to\infty,
\end{equation}
for some constant $\const{eq:W_expec}\in\R$. Equation \eqref{eq:W_tail} has been proven in Propositions 27 and 40 of \cite{Berestycki2010}, and \eqref{eq:W_expec} appears in the proof of Proposition 39 of the same paper. Their arguments were very ingenious but indirect and although they could be extended to general reproduction laws with finite variance, we will reprove them here directly under (probably) minimal assumptions, based on methods of \cite{Maillard2010}. The main result in this section is
\begin{proposition}
 \label{prop:W}
If $E[L\log^2L] < \infty$, then \eqref{eq:W_tail} holds. If $E[L\log^3L]<\infty$, then \eqref{eq:W_expec} holds.
\end{proposition}
See also \cite{Buraczewski2009} for a proof of \eqref{eq:W_tail} in the case of branching random walk.
Before proving this result in the next subsection, we state a lemma which is immediate from \eqref{eq:def_Wy} and the fact that $Z_y$ is almost surely finite (see also Corollary 25 in \cite{Berestycki2010}):
\begin{lemma}
\label{lem:y_zeta}
 Suppose $E[L\log^2 L]<\infty$. For any $\eta > 0$, there exist $y$ and $\zeta$, such that
\[
 y \ge \eta^{-1}\quad\tand\quad P(|W_y - W| > \eta) + P(\mathscr L_{H_{-y}} \nsubseteq U\times [0,\zeta]) \le \eta.
\]
\end{lemma}

\subsection{Proof of Proposition \ref{prop:W}}

Rescaling space by $c_0^{-1}$, we can always come down to the case $c_0 = 1$, which we will assume to hold for the rest of the proof. Define $\chi(\lambda) = E[e^{-\lambda W}]$. Our first result is:

\begin{lemma}
\label{lem:chi}
 Suppose that $E[L\log^2 L] < \infty$. Then, 
\begin{equation}
\label{eq:chi_asymptotic}
 \chi''(\lambda) \sim \frac{1}{\lambda},
\end{equation} 
as $\lambda \to 0$. If furthermore $E[L\log^3 L] < \infty$, then
\begin{equation}
\label{eq:chi_asymptotic_2}
 \chi''(\lambda) = \frac{1}{\lambda} - r(\lambda),
\end{equation} 
where $r(\lambda) \ge 0$ and $\int_0^1 r(\lambda)\,\dd\lambda < \infty$.
\end{lemma}

\begin{proof}
Define $\phi(x) = 1-\psi(-x)$, such that $u(s) = \phi'(\phi^{-1}(s))$. By \eqref{eq:traveling_wave}, we have
\begin{equation}
\label{eq:phi} 
\frac 1 2 \phi''(x) + \phi'(x) = f(1-\phi(x)) - (1-\phi(x)).
\end{equation}
Then by \eqref{eq:tw_asymptotics}, we have
\begin{equation}
 \label{eq:phi_asymptotic}
 \phi(x) \sim K x e^{-x},\quad\tas x\to \infty.
\end{equation}
Setting $g(s) = 2[f(1-s) - 1 + f'(1)s] \ge 0$ and $\rho = \phi + \phi'$, we get from \eqref{eq:phi},
\begin{equation}
 \label{eq:rho}
 \rho'(x) = -\rho(x) + g(\phi(x)).
\end{equation}
As in the proof of Theorem 1.1 in \cite{Maillard2010}, we will study the function $\rho$ through the integral equation corresponding to \eqref{eq:rho}, namely
\begin{equation}
 \label{eq:rho_integral}
 \rho(x) = e^{-x}\Big(\rho(0) + \int_0^x e^yg(\phi(y))\,\dd y\Big) = e^{-x}\Big(\rho(0) + \int_{\phi(x)}^{\phi(0)} \frac{e^{\phi^{-1}(s)}g(s)}{-u(s)}\ \dd s\Big).
\end{equation}
Now, by Theorem B of \cite{Bingham1974} (see also Theorem 8.1.8 in \cite{Bingham1987}) we have for every $d\ge 0$,
\begin{equation}
 \label{eq:bingham}
 \int_0^1 \frac{\log^d \frac 1 s}{s^2}g(s)\,\dd s < \infty\quad\Longleftrightarrow\quad\int_0^1 \frac{\log^d \frac 1 s}{s}g'(s)\,\dd s < \infty\quad\Longleftrightarrow\quad E[L\log^{1+d}L]<\infty.
\end{equation}
Furthermore, by Proposition 3.2 in \cite{Maillard2010}, we have $-u(s) \sim s$, as $s\to 0$, and by \eqref{eq:phi_asymptotic}, we have $e^{\phi^{-1}(s)} \sim (\log 1/s)/s$, as $s\to 0$. Under the hypothesis $E[L\log^2 L]<\infty$, we therefore have
\[
 \int_0^{\phi(0)} \frac{e^{\phi^{-1}(s)}g(s)}{-u(s)}\ \dd s < \infty,
\]
whence, by \eqref{eq:rho_integral},
\begin{equation}
\label{eq:rho_asymptotic}
\rho(x) \sim K e^{-x},\quad\tas x\to\infty,
\end{equation}
where the constant $K$ is actually the same as the one in \eqref{eq:phi_asymptotic}, see the proof of Theorem 1.1 in \cite{Maillard2010}. Now, from \eqref{eq:W_laplace}, we get $\chi(\lambda) = 1-\phi(-\log \lambda)$, whence, by \eqref{eq:rho} and \eqref{eq:rho_integral},
\begin{equation}
 \label{eq:chi_rho}
\begin{split}
\chi''(\lambda) = -\frac 1 {\lambda^2}\rho'(-\log \lambda) &= \frac K {\lambda} + \frac 1 {\lambda^2}\Big(-\lambda\int_{-\log\lambda}^\infty e^yg(\phi(y))\,\dd y - g(\phi(-\log\lambda))\Big)\\
& = \frac K {\lambda} + \frac 1 \lambda\Big(\int_{-\log\lambda}^\infty e^y\phi'(y)g'(\phi(y))\,\dd y\Big),
\end{split}
\end{equation}
where the last equation follows from integration by parts. This proves \eqref{eq:chi_asymptotic}, with the constant $K$ instead of $1$, since the last integral vanishes as $\lambda \to 0$. Now, setting
\[
 r(\lambda) = - \frac 1 \lambda\Big(\int_{-\log\lambda}^\infty e^y\phi'(y)g'(\phi(y))\,\dd y\Big),
\]
we first remark that $r(\lambda) \ge 0$, since the integrand is negative for $y\in\R$. By the Fubini--Tonelli theorem, we then have
\begin{align*}
 \int_0^1r(\lambda)\dd \lambda & = \int_0^1 -\frac 1 \lambda \int_{-\log \lambda}^\infty e^y\phi'(y)g'(\phi(y))\,\dd y\,\dd \lambda\\
& = \int_0^\infty - ye^y\phi'(y)g'(\phi(y))\,\dd y\\
& = \int_0^{\phi(0)}e^{\phi^{-1}(y)} \phi^{-1}(y) g'(y)\,\dd y,
\end{align*}
which is finite if and only if $E[L\log^3 L]<\infty$, by \eqref{eq:bingham} and the fact that $e^{\phi^{-1}(y)} \phi^{-1}(y) \sim (\log^2 1/s)/s.$ This proves \eqref{eq:chi_asymptotic_2}, again with the constant $K$ instead of $1$.

The previous arguments worked for every traveling wave $\psi$. In order to show that that the constant $K$ is equal to 1 in our case, we use Neveu's multiplicative martingale (see also \cite{Liu2000}, Theorem 2.5). We know \cite{Neveu1988,Chauvin1991}, $((1-\phi(x+y))^{Z_y})_{y\ge 0}$ is a martingale for every $x\in\R$, bounded by $1$. By \eqref{eq:def_Wy} and \eqref{eq:phi_asymptotic}, we get by dominated convergence, for every $x\in\R$,
\[
 \chi(Ke^x) = \lim_{y\to \infty} E[e^{-Kye^{x-y}Z_y}] = \lim_{y\to\infty}E[(1-\phi(y-x))^{Z_y}] = 1-\phi(-x) = \chi(e^x).
\]
This yields $K=1$.
\end{proof}

\begin{remark}
 Choosing arbitrary initial points $x_0,x_1\in\R$ instead of $0$ in \eqref{eq:rho_integral}, one sees that
\[
 e^{x_0}\rho(x_0) + \int_{x_0}^\infty e^yg(\phi(y))\,\dd y = e^{x_1}\rho(x_1) + \int_{x_1}^\infty e^yg(\phi(y))\,\dd y.
\]
In particular, since $\rho$ is bounded, letting $x_0\to -\infty$ and $x_1\to +\infty$ yields
\[
 \int_{-\infty}^\infty e^yg(\phi(y))\,\dd y = 1.
\]
One could hope (see the proof of Proposition \ref{prop:W} below) that this helps in determining the constant \const{eq:W_expec}, but apparently this does not seem to be the case.
\end{remark}

\begin{proof}[Proof of Proposition \ref{prop:W}]
 We define the function
\[
 V_n(x) = \int_0^y x^n P(W\in\dd x) = E[W^n\Ind_{(W\le x)}],
\]
such that with $\chi^{(n)}$ denoting the $n$-th derivative of $\chi$, we have for $\lambda > 0$,
\[
 \chi^{(n)}(\lambda) = (-1)^n \int_0^\infty e^{-\lambda x} \dd V_n(x).
\]
If $E[L\log^2L]<\infty$, Proposition \ref{prop:W} and Karamata's Tauberian theorem (\cite{Feller1971}, Theorem XIII.5.2 or \cite{Bingham1987}, Theorem 1.7.1) now yields
\begin{equation}
 \label{eq:V2}
 V_2(x) \sim x,\quad\tas x\to\infty.
\end{equation}
By an integration by parts argument (see also \cite{Feller1971}, Theorem VIII.9.2 or \cite{Bingham1987}, Theorem 8.1.2), we get \eqref{eq:W_tail}. Now suppose that $E[L\log^3L] < \infty$. By Lemma \ref{lem:chi}, we have $\chi'(\lambda) - \log \lambda \to c\in \R$, as $\lambda \to 0$. By Theorem 3.9.1 from \cite{Bingham1987} (with $\ell(x) \equiv 1$), this yields
\[
 V_1(x) - \log x \to \gamma - c,\quad\tas x\to \infty,
\]
where $\gamma$ is the Euler--Mascheroni constant. This is exactly \eqref{eq:W_expec}.


\end{proof}

\section{Branching Brownian motion in an interval}
\label{sec:interval}

In this section we study branching Brownian motion killed upon exiting an interval. Most ideas in this section (except for Section \ref{sec:weakly_conditioned}) stem from Sections 2 and 3 of \cite{Berestycki2010} and for completeness, we will reprove some of their results with streamlined proofs. However, we will also extend their results to the case of Brownian motion with variable drift.

\subsection{Notation}
\label{sec:interval_notation}
During the rest of the paper, the symbol $C$ stands for a positive constant, which may only depend on the reproduction law $q$, except in Section \ref{sec:moving_barrier}, where it may also depend on some other constants which will be specified. Its value may change from line to line. If a subscript is present, then this subscript is the number of the equation where this constant appears for the first time (example: $\Const{eq:Rt_nu_expec}$). In this case, this constant is fixed after its value has been chosen in the corresponding equation. If $X$ is any mathematical expression, then the symbol $O(X)$ stands for a term whose absolute value is bounded by $C|X|$.

Recall the definition of $q(k)$, $m$ and $m_2$ from Section \ref{sec:critical_line} and the hypotheses on $m$ and $m_2$. From this section on, we further define $c_0 = \sqrt{2m}$. Furthermore, in this section, we let $a \ge \pi/c_0$ and set
\begin{equation}
 \label{eq:mu}
\mu = \sqrt{2m - \frac{\pi^2}{a^2}}.
\end{equation}
From \eqref{eq:mu}, one easily gets the basic estimate
\begin{equation}
 \label{eq:c0_mu}
 0\le c_0 - \mu \le \frac{\pi^2}{2\mu a^2}.
\end{equation}

We then denote by $\P^x$ the law of the branching Markov process where, starting with a single particle at the point $x\in\R$, particles move according to Brownian motion with variance $1$ and drift $-\mu$ and branch at rate 1 into $k$ particles according to the reproduction law $q(k)$. Expectation with respect to $\P^x$ is denoted by $\E^x$. On the space of continuous functions from $\R_+$ to $\R$, we define $H_0$ and $H_a$ to be the hitting time functionals of $0$ and $a$. We further set $H = H_0\wedge H_a$. Then note that the density kernel of the branching Brownian motion below $\mathscr L_H$, as defined in \eqref{eq:def_density}, has a density with respect to Lebesgue measure given for $t>0$ and $x,y \in (0,a)$ by
\begin{equation}
 \label{eq:density_bbm}
\p_t(x,y) = e^{\mu(x-y) + \frac{\pi^2}{2 a^2} t} p_t^a(x,y),
\end{equation}
where $p_t^a$ was defined in \eqref{eq:def_p}.

Now, let $f\in\mathscr{C}^2(\R_{\ge 0},\R_{\ge 0})$ be non-decreasing, with $f(0)=0$. Such a function will be called a \emph{barrier function}. We set
\begin{equation}
\label{eq:def_f_norm}
 ||f|| = \max\{||f||_\infty, ||f'||_\infty,||f'||_\infty^2,||f''||_\infty\}.
\end{equation}
Now define
\begin{equation}
 \label{eq:mu_t}
\mu_t = \mu + \frac{\dd}{\dd t} f(t/a^2) = \mu + \frac{1}{a^2} f'(t/a^2),
\end{equation}
such that $\mu_0 = \mu$ and $\mu_t \ge \mu$ for all $t\ge0$. We denote by $\P^x_f$ the law of the branching Brownian motion described above, but with infinitesimal drift $-\mu_t$. Expectation with respect to $\P^x_f$ is denoted by $\E^x_f$ and the density of the process is denoted by $\p^f_t(x,y)$.

The above definitions can be extended to arbitrary initial configurations of particles distributed according to a counting measure $\nu$ on $(0,a)$. In this case the superscript $x$ is replaced by $\nu$ or simply omitted if $\nu$ is known from the context.

\subsection{The processes \texorpdfstring{$Z_t$ and $Y_t$}{Z\_t and Y\_t}}

Recall from Section \ref{sec:preliminaries} that the set of particles alive at time $t$ is denoted by $\mathscr{N}(t)$. We define
\[
 \mathscr N_{0,a}(t) = \{u\in \mathscr N(t): H(X_u) > t\},
\]
where $H$ was defined in the previous subsection. Now set $w(x) = ae^{\mu (x-a)}\sin(\pi x/a)$ and define
\[
 Z_t = \sum_{u\in \mathscr N_{0,a}(t)} w(X_u(t))\quad\tand\quad Y_t = \sum_{u\in \mathscr N_{0,a}(t)} e^{\mu (X_u(t)-a)}.
\]
Then $Z_t$ is a martingale under $\P^x$, since $e^{mt} w(B_t)$ is a martingale for a Brownian motion with drift $-\mu$ killed at $0$ and $a$, which is easily seen by It\={o}'s formula, for example. Furthermore, it is easy to see as well that $Z_t$ is a supermartingale under $\P^x_f$.

The following lemma relates the density of BBM with variable drift to BBM with fixed drift:
\begin{lemma}
 \label{lem:mu_t_density}
\[
 \p^f_t(x,y) = \p_t(x,y) e^{-c_0 f(t/a^2) + \mathrm{Err}},
\]
where $|\mathrm{Err}| \le ||f||\Big(\frac{1}{a} + \frac{t}{a^3} + \frac{\pi^2}{2\mu a^2}\Big)$.
\end{lemma}
\begin{proof}
 By the Many-to-one lemma and Girsanov's theorem, we have
\begin{equation}
\label{eq:005}
\begin{split}
 \p^f_t(x,y) &= e^{mt}W_{-\mu_t}^x\left(B_t\in \dd y,\, H > t\right)\\
&= \exp\left(mt - \int_0^t\frac{\mu_s^2-\mu^2}{2} \dd s\right)W_{-\mu}^x\left(\exp\Big(-\int_0^t \mu_s-\mu\,\dd B_s\Big),\,B_t\in\dd y, H > t\right).
\end{split}
\end{equation}
By integration by parts, we have
\begin{equation}
 \label{eq:ipp_mu_B}
 \int_0^t\mu_s\,\dd B_s = \mu_t B_t - \mu B_0 - \int_0^tB_s\,\dd \mu_s.
\end{equation}
Since $B_t \in (0,a)$ for all $t\ge 0$, we have
\begin{equation}
 \label{eq:010}
 \Big|\int_0^t B_s\,\dd \mu_s\Big| \le \int_0^t \Big|B_s \frac {f''(s/a^2)} {a^4}\Big|\,\dd s \le ||f''||_\infty \frac t {a^3}.
\end{equation}
Furthermore,
\[
 \frac{\mu_t^2}{2} = \frac{\mu^2}{2} + \frac{\mu}{a^2}f'(t/a^2)+\frac{f'(t/a^2)^2}{2a^4},
\]
such that
\begin{equation}
 \label{eq:020}
 \Big|\int_0^t \frac{\mu_s^2}{2}\,\dd s - \frac{\mu^2}{2}t - \mu f(t/a^2)\Big| \le \frac{||f'||_\infty^2 t}{2a^4}.
\end{equation}
Finally,
\begin{equation}
 \label{eq:030}
 \Big|\mu_tB_t - \mu B_t\Big| = \Big|\frac{1}{a^2}f'(t/a^2) B_t\Big| \le \frac{||f'||_\infty}{a}.
\end{equation}
Equations \eqref{eq:density_bbm}, \eqref{eq:005}, \eqref{eq:ipp_mu_B}, \eqref{eq:010}, \eqref{eq:020} and \eqref{eq:030} now give
\[
  \p^f_t(x,y) = \p_t(x,y) e^{-\mu f(t/a^2) + \mathrm{Err}},
\]
and the lemma now follows from \eqref{eq:c0_mu}.
\end{proof}

\begin{proposition}
 \label{prop:quantities}
Under any initial configuration of particles, for every $t\ge 0$, we have
\begin{equation}
 \label{eq:Zt_expectation}
\E_f[Z_t] = Z_0e^{-c_0 f(t/a^2) + \mathrm{Err}},
\end{equation}
and if in addition $\mu \ge c_0/2$, then
\begin{equation}
 \label{eq:Zt_variance} 
\Var_f(Z_t) \le C e^{3\mathrm{Err}} \Big(\frac{t}{a^3} Z_0 + Y_0 \Big).
\end{equation}
Furthermore, we have for every $t\ge 0$ (without hypothesis on $\mu$),
\begin{equation}
 \label{eq:Yt_weak}
\E_f[Y_t] \le Ce^{\mathrm{Err}}Y_0.
\end{equation}
and for $t\ge a^2$,
\begin{equation}
  \label{eq:Yt}
\E_f[Y_t] \le Ce^{\mathrm{Err}} \frac{Z_0}{a}.
\end{equation}
Here, $\mathrm{Err}$ has the same meaning as in Lemma \ref{lem:mu_t_density}.
\end{proposition}

\begin{proof} 
Equation \eqref{eq:Zt_expectation} follows from Lemma \ref{lem:mu_t_density} and the fact that $Z_t$ is a martingale under $\P^x$. In order to show \eqref{eq:Yt_weak} and \eqref{eq:Yt}, it suffices by Lemma \ref{lem:mu_t_density} to consider the case without variable drift. We first suppose that $t\ge a^2$. By \eqref{eq:density_bbm} and \eqref{eq:p_estimate}, we get
\[
 \E^x[Y_t] \le e^{\mu (x-a)}\int_0^a e^{\frac{\pi^2}{2a^2}t} p_t(x,y)\,\dd y
\le C e^{\mu (x-a)} \sin(\pi x/a)\int_0^a \frac 2 a \sin(\pi y/a)\,\dd y.
\]
The last integral is independent of $a$. Summing over $x$ yields \eqref{eq:Yt} as well as \eqref{eq:Yt_weak} in the case $t\ge a^2$. Now, if $t< a^2$, by the Many-to-one lemma and Girsanov's theorem, we have
\[
 \E^x[Y_t] = e^t W^x_{-\mu}\Big[e^{\mu (X_t-a)},\ H_0\wedge H_a < t\Big] = e^{\pi^2t/(2a^2)} W^x[H_0\wedge H_a < t] e^{\mu (x-a)}.
\]
Summing over $x$ yields \eqref{eq:Yt_weak}.

Throughout the proof of \eqref{eq:Zt_variance}, we use a constant $C$, which depends only on the reproduction law $q(k)$ and which may change from line to line. By Lemma \ref{lem:many_to_two_fixedtime},
\begin{equation}
 \label{eq:bbm_2ndmoment}
 \E^x[Z_t^2] = \E_f^x\Big[\sum_{u\in \mathscr N_{0,a}(t)}w(X_u(t))^2\Big] + 2m_2\int_0^a \int_0^t \p_s^f(x,y) (\E_{f(\cdot + s)}^y[Z_{t-s}])^2\,\dd s\,\dd y.
\end{equation}
By Lemma \ref{lem:mu_t_density} and the fact that $Z_t$ is a martingale with respect to the law $\P^x$, \eqref{eq:bbm_2ndmoment} yields
\begin{equation}
 \label{eq:050}
 \E^x[Z_t^2] \le C e^{3 \mathrm{Err}} \left(\E^x\Big[\sum_{u\in \mathscr N_{0,a}(t)}w(X_u(t))^2\Big] + \int_0^a \int_0^t \p_s(x,y) w(y)^2\,\dd s\,\dd y\right).
\end{equation}
Now we have for $x\in (0,a)$,
\[
 w(x)^2 = (a\sin(\pi x/a) e^{-\mu(a-x)})^2 \le \pi^2 (a-x)^2e^{-2\mu(a-x)} \le C e^{\mu(x-a)},
\]
because $\mu \ge c_0/2$ by hypothesis. This yields
\begin{equation}
 \label{eq:S1}
 S_1 := \E^x\Big[\sum_{u\in \mathscr N_{0,a}(t)}w(X_u(t))^2\Big] \le C \E^x[Y_t] \le C e^{\mu(x-a)},
\end{equation}
by \eqref{eq:Yt_weak}. Now, by \eqref{eq:density_bbm} and \eqref{eq:Ja}, we have
\[
 S_2 := \int_0^a \int_0^t \p_s(x,y) w(y)^2\,\dd s\,\dd y = ae^{\mu (x-a)}\int_0^aae^{\mu (y-a)}\sin^2(\pi y/a) J^a(x,y,t)\,\dd y.
\]
Lemma \ref{lem:I_estimates} now gives
\begin{equation}
\label{eq:S2}
\begin{split}
 S_2 &\le C ae^{\mu (x-a)}\int_0^a e^{-\mu y}\sin^2(\pi y/a)\Big(t\sin(\pi x/a)\sin(\pi y/a) + ay\Big)\,\dd y\\
&\le C ae^{\mu (x-a)} \Big(\sin(\pi x/a)\frac t{a^3} + \frac 1 a\Big)\int_0^\infty e^{-\mu y} y^3\,\dd y,
\end{split}
\end{equation}
the last line following again from the change of variables $y\mapsto a-y$ and the inequality $\sin x\le x$. Using again the fact that $\mu \ge c_0/2$, equations \eqref{eq:050}, \eqref{eq:S1} and \eqref{eq:S2} now imply
\begin{equation}
 \label{eq:Zt_2ndmoment}
\E^x_f[Z_t^2] \le C e^{3 \mathrm{Err}} \Big(\frac{t}{a^3} w(x) + e^{\mu(x-a)} \Big).
\end{equation}
If we write the positions of the initial particles as $x_1,\ldots,x_n$, then by the independence of their contributions to $Z_t$ and by \eqref{eq:Zt_2ndmoment},
\begin{equation}
 \Var_f(Z_t) = \sum_i \Var^{x_i}_{f_i}(Z_t) \le \sum_i \E^{x_i}_{f_i}[Z_t^2] \le e^{3\mathrm{Err}} \sum_i \E^{x_i}[Z_t^2] \le C e^{3 \mathrm{Err}} \left(\frac{t}{a^3} Z_0 + Y_0\right),
\end{equation}
by \eqref{eq:Zt_2ndmoment}. This proves \eqref{eq:Zt_variance}.
\end{proof}

\subsection{The particles hitting the right border}
In this section we recall some formulas from \cite{Berestycki2010} about the number of particles hitting the right border of the interval. We reprove these formulae here for completeness and because Lemma \ref{lem:I_estimates} makes their proofs straightforward. For most formulae we will assume that $f\equiv 0$, i.e.\ that we are working under the measure $\P$. Only Lemma \ref{lem:R_f} contains an upper bound on the expected number of particles for general $f$, which will be useful in Section \ref{sec:moving_barrier}.

For a measurable subset $S\subset \R$, define $R_S$ to be the number of particles killed at the right border during the (time) interval $S$, i.e.
\[
 R_S = \#\{(u,t):u\in \mathscr N(t)\tand H_0(X_u)>H_a(X_u) = t\in S\}.
\]
The following lemma gives exact formulae of the expectation and the second moment of $R_S$.

\begin{lemma}
 \label{lem:R_moments}
For every $x\in (0,a)$, we have
 \begin{align}
  \label{eq:RS_x_expec}
  \E^x[R_S] &= e^{\mu(x-a)} I^a(x,S),\\
  \label{eq:RS_x_2ndmoment}
  \E^x[R_S^2] &= \E^x[R_S] + m_2 e^{\mu(x-a)} \int_0^a \dd y\, e^{\mu(y-a)} \int_0^\infty \dd t\,e^{\frac{\pi^2}{2a^2}t}p_t^a(x,y) I^a(y,S-t)^2
 \end{align}
\end{lemma}
We will first prove a more general result, which will be needed in Section \ref{sec:fugitive}.
\begin{lemma}
\label{lem:R_many_to_one}
 For every $x\in (0,a)$ and any measurable function $f:\R_+\to \R_+$, we have
 \[
  \E^x\Big[\sum_{(u,t)\in \mathscr L_H} f(t) \Ind_{(X_u(t) = a)}\Big] = e^{\mu(x-a)}\int_0^t f(s)I^a(x,\dd s).
 \]
\end{lemma}

\begin{proof}
Recall that $H_0$ and $H_a$ denote the hitting time functionals of $0$ and $a$ and $H = H_0\wedge H_a$. Then note that $W^x_{-\mu}(H<\infty) = 1$ for all $x\in [0,a]$. We then have
\begin{align*}
 \E^x\Big[\sum_{(u,t)\in \mathscr L_H} f(t) \Ind_{(X_u(t) = a)}\Big] &= W^x_{-\mu}\Big[e^{mH_a} f(H_a)\Ind_{(H_0 > H_a\le t)}\Big] && \text{by Lemma \ref{lem:many_to_one_simple}}\\
&= e^{\mu(x-a)} W^x_0\Big[e^{\frac{\pi^2}{2a^2}H_a}f(H_a)\Ind_{(H_0 > H_a\le t)}] && \text{by Girsanov's transform}\\
&= e^{\mu(x-a)} \int_0^t f(s)I^a(x,\dd s) && \text{by \eqref{eq:Ia}.}
\end{align*}
\end{proof}
\begin{proof}[Proof of Lemma \ref{lem:R_moments}]
Equation \eqref{eq:RS_x_expec} follows from Lemma \ref{lem:R_moments} and \eqref{eq:IJ_scaling} by taking $f=\Ind_S$. Equation \eqref{eq:RS_x_2ndmoment} follows from Lemma \ref{lem:many_to_two} and \eqref{eq:RS_x_expec}.
\end{proof}

\begin{lemma}
\label{lem:Rt}
 For any initial configuration $\nu$, and any $0\le s\le t$, we have
\begin{equation}
 \label{eq:Rt_nu_expec}
 |\E R_{[s,t]} - \frac{\pi (t-s)}{a^3} Z_0| \le \Const{eq:Rt_nu_expec} \Big(Y_0 \wedge E_{s/a^2} (1\wedge (t-s)/a^3) Z_0\Big),
\end{equation}
where $E_s$ is defined in \eqref{eq:def_E}. Furthermore, if $\mu \ge c_0/2$ and $0\le t\le a^3$, then for each $x\in (0,a)$,
\begin{equation}
 \label{eq:Rt_nu_variance}
 \E^x R_t^2\le \Const{eq:Rt_nu_variance} \Big(\frac{t}{a^3} w(x) + e^{\mu(x-a)}\Big),
\end{equation}
\end{lemma}

\begin{proof}
We have $\E R_t = \int \nu(\dd x) E^x R_t$, such that \eqref{eq:Rt_nu_expec} follows from \eqref{eq:RS_x_expec} and Lemma \ref{lem:I_estimates}. For the second moment, we have by \eqref{eq:RS_x_2ndmoment},
\[
\begin{split}
 \E^x [R_t^2 - R_t] &= m_2e^{\mu(x-a)} \int_0^a \dd y\, e^{\mu(y-a)} \int_0^t \dd s\,e^{\frac{\pi^2}{2a^2}t}p_s^a(x,y)I^a(y,t-s)^2\\
 &\le m_2 e^{\mu(x-a)} \int_0^a \dd y\, e^{\mu(y-a)} I^a(y,t)^2 J^a(x,y,t)\\
 &\le C e^{\mu(x-a)} \int_0^a \dd y\, e^{\mu(y-a)} (t/a^2 \sin(\pi y/a) + 1)^2\\
&\quad\quad\quad\quad\quad\times (a t/a^2 \sin(\pi x/a)\sin(\pi y/a) + a^{-1}(x\wedge y)(a-(x\vee y)))
\end{split}
\]
Performing the change of variables $y\mapsto a-y$ in the integral and making use of the inequalities $a^{-1}(x\wedge y)(a-(x\vee y)) \le a-y$ and $\sin x \le x$, we get
\[
\begin{split}
 \E^x [R_t^2 - R_t] &\le O(1) e^{\mu(x-a)} (\sin(\pi x/a) t/a^2 + 1) \int_0^\infty \dd y\,e^{-\mu y} (y + y^2 t/a^3 + y^3 t^2/a^6)\\
 &\le C e^{\mu(x-a)} (\sin(\pi x/a) t/a^2 + 1) (1+ t^2/a^6),
\end{split}
\]
for some constant $C$, which does not depend on $\mu$ by the hypothesis $\mu \ge c_0/2$. Using the hypothesis $t\le a^3$ and \eqref{eq:Rt_nu_expec} yields \eqref{eq:Rt_nu_variance}.
\end{proof}

\begin{lemma}
 \label{lem:R_f}
Let $f$ be a function as in Section \ref{sec:interval_notation}. Then for every $x\in (0,a)$, we have
\[
 \E^x_f[R_S] \le \E^x[R_S].
\]
\end{lemma}
\begin{proof}
 As in the proof of Lemma \ref{lem:R_moments}, we have
\[
 \E^x_f[R_S] = W^x_{-\mu_t}\Big[e^{mH_a} \Ind_{(H_0 > H_a\in S)}\Big] = W^x_0\Big[e^{-\int_0^{H_a} \mu_t\,\dd B_t - \int_0^{H_a} \mu_t^2/2\,\dd t + mH_a} \Ind_{(H_0 > H_a\in S)}\Big],
\]
by Girsanov's theorem. Now, we have by \eqref{eq:ipp_mu_B}, on the event $\{H_0 > H_a\}$,
\[
 \int_0^{H_a} \mu_t\,\dd B_t = \mu(a-x) + a (\mu_{H_a} - \mu) - \int_0^{H_a} B_t\,\dd \mu_t \ge \mu(a-x),
\]
since $B_t \in [0,a]$ for $t\in [0,H_a]$. This gives
\[
 \E^x_f[R_S] \le e^{\mu(x-a)}W^x_0\Big[e^{(m-\mu^2/2) H_a} \Ind_{(H_0 > H_a\in S)}\Big] = \E^x[R_S],
\]
by the proof of Lemma \ref{lem:R_moments}.
\end{proof}

We finish this section with a lemma which links BBM with absorption at a critical line to our BBM with selection model.

\begin{lemma}
\label{lem:critical_line}
 Let $\zeta \ge 1$, $y\ge 1$, $\mu \ge c_0/2$ and $f$ be a barrier function (defined in Section \ref{sec:interval_notation}). Suppose that $\sqrt a \ge y + \zeta$ and $||f||\le \sqrt a$. Let $(x_i,t_i)_{i=1}^N$ be a collection of space-time points with
\[
 x_i = a-y+(c_0-\mu)t_i-f(s/a^2),\quad i=1,\ldots,N,
\]
and $t_i \le \zeta$ for all $i$. Define $Z = \sum_i w(x_i)$, $Y = \sum_i e^{\mu (x_i-a)}$ and $W_y = c_0ye^{-c_0 y} N$. Then,
\[
 Z = \frac{\pi}{c_0} W_y\Big(1+O\Big(\frac 1 a \Big)\Big)\quad\tand\quad  Y = \frac 1 {c_0 y} W_y \Big(1+O\Big(\frac{1}{a}\Big)\Big).
\]
In particular, for large $a$, we have
\[
 Y \le Z/y.
\]

\end{lemma}
\begin{proof}
By \eqref{eq:c0_mu} and the hypotheses $\mu \ge c_0/2$ and $\zeta \ge 1$, we have for all $i$,
\[
 x_i = a-y + O\Big(\frac{\zeta(1+||f||)}{a^2}\Big).
\]
Hence, by \eqref{eq:c0_mu} and the hypotheses $\mu \ge c_0/2$, $||f||\le \sqrt a$ and $\sqrt a \ge y + \zeta$, 
\begin{equation}
 \label{eq:250}
 e^{\mu x_i} = e^{\mu a-c_0y}\Big(1+O\Big(\frac{1}{a}\Big)\Big).
\end{equation}
Furthermore, since $x-x^2/3\le \sin x \le x$ for $x\ge 0$, and by the hypotheses $y\ge 1$, and $a \ge y + \zeta$ 
\begin{equation}
 \label{eq:270}
 \sin(\pi x_i/a) =  \sin(\pi (a-x_i)/a) =\frac \pi a y\Big(1 + O\Big(\frac y a\Big)\Big).
\end{equation}
The lemma now follows by summing over \eqref{eq:250} and \eqref{eq:270}.
\end{proof}

\subsection{Penalizing the particles hitting the right border}
\label{sec:weakly_conditioned}
In this section, let $(U_u)_{u\in U}$ be iid random variables, uniformly distributed on $(0,1)$, independent of the branching Brownian motion. Furthermore, let $p(t):\R_+\in (0,1]$ be measurable and such that $p(t) = 0$ for large enough $t$. Recall that $H = H_0 \wedge H_a$. We define the event
\[
 E = \{\nexists (u,t)\in \mathscr L_H: X_u(t) = a \tand U_u \le p(t)\}.
\]
Our goal in this section is to describe the law $\Ptilde^x = \P^x(\cdot|E)$. We first note that
\begin{equation}
\label{eq:180}
 \P^x(\dd \omega, E) = \P^x(\dd \omega) \prod_{(u,t)\in L_t}\Big(\Ind_{(X_u(t) \ne a)} + p(t)\Ind_{(X_u(t) = a)}\Big).
\end{equation}
In order to apply the results from Section \ref{sec:doob}, we define
\begin{align}
\label{eq:def_gamma}
 h(x,t) &= \P^{(x,t)}(E)\\
\label{eq:def_beta}
 Q(x,t)      &= \sum_{k=0}^\infty q(k)h(x,t)^{k-1}\\
\label{eq:def_btilde}
 \widetilde q(x,t,k) &= q(k) h(x,t)^{k-1} / Q(x,t)
\end{align}
By the results from Section \ref{sec:doob}, under the law $\Ptilde^x$, the BBM stopped at $\mathscr L_H$ is the branching Markov process where
\begin{itemize}[nolistsep,label=$-$]
 \item particles move according to the Doob transform of Brownian motion with drift $-\mu$, stopped at $0$ and $a$, by the space-time harmonic function $h(x,t)$, and
 \item a particle located at the point $x$ at time $t$ branches at rate $Q(x,t)\Ind_{x\in (0,a)}$, throwing $k$ offspring with probability $\widetilde q(x,t,k)$ and
 \item a particle located at $0$ or $a$ does not branch.
\end{itemize}

We have the following useful Many-to-one lemma for the conditioned process stopped at the stopping line $\mathscr L_t = \mathscr L_H \wedge t$: Define the function
\begin{equation}
 \label{eq:def_e}
 e(x,t) = \sum_{k\ge 0} k (1-h(x,t)^{k-1}) q(k) \le m_2(1-h(x,t)).
\end{equation}
\begin{lemma}
\label{lem:many_to_one_conditioned}
 For any measurable function $f:[0,a]\to \R_+$, we have
\begin{equation}
\label{eq:many_to_one_conditioned}
 \Etilde^x\Big[\sum_{u\in \mathscr L_t} f(X_u(t))\Big] = \frac{e^{\mu x}}{h(x,0)} W^x\Big[f(X_{H\wedge t})e^{-\mu X_{H\wedge t}}h(X_{H\wedge t},H\wedge t) e^{-\frac{\pi^2}{2a^2}(H\wedge t) -\int_0^{H\wedge t} e(X_s,s) \,\dd s}\Big].
\end{equation}
In particular, if we denote by $\widetilde{\p}(x,y,t)$ the density of the $\Ptilde^x$-BBM, then
\begin{equation}
 \label{eq:ptilde_estimate}
\widetilde{\p}(x,y,t) \le \frac{h(y,t)}{h(x,0)} \p_t(x,y).
\end{equation}
\end{lemma}

\begin{proof}
By Lemma \ref{lem:many_to_one_simple} and the description of the law $\Ptilde^x$ given above, we have
\[
 \Etilde^x\Big[\sum_{u\in \mathscr L_t} f(X_u(t))\Big] = W^x_{-\mu}\Big[f(X_{H\wedge t})h(X_{H\wedge t},H\wedge t) e^{\int_0^{H\wedge t} \widetilde m(X_s,s)Q(X_s,s) \,\dd s}\Big],
\]
where $\widetilde m(x,t) = \sum_k(k-1)\widetilde q(x,t,k)$. Applying Girsanov's transform yields \eqref{eq:many_to_one_conditioned}. Equation \eqref{eq:ptilde_estimate} follows from \eqref{eq:many_to_one_conditioned} applied to the Dirac Delta-function $f = \delta_y$, $y\in (0,a)$, together with \eqref{eq:density_bbm}.
\end{proof}

The previous lemma immediately gives an upper bound for the quantities we are interested in:
\begin{corollary}
For any $x\in (0,a)$ and $t\ge 0$, we have
 \begin{align}
  \label{eq:conditioned_Z_upperbound}
  \Etilde^x[Z_t] &\le (h(x,0))^{-1} \E^x[Z_t],\\
  \label{eq:conditioned_Z2}
  \Etilde^x[Z^2_t] &\le (h(x,0))^{-1} \E^x[Z^2_t],\\
  \label{eq:conditioned_Y}
  \Etilde^x[Y_t] &\le (h(x,0))^{-1} \E^x[Y_t].
 \end{align}
\end{corollary}
\begin{proof}
Equations \eqref{eq:conditioned_Z_upperbound} and \eqref{eq:conditioned_Y} immediately follow from \eqref{eq:ptilde_estimate} and the fact that $\gamma(y,t)\le 1$ for all $y$ and $t$. In order to prove \eqref{eq:conditioned_Z2}, we note that by Lemma \ref{lem:many_to_two_fixedtime} and the description of the conditioned process,
\[
 \begin{split}
  \Etilde^x[Z^2_t] &= \Etilde^x\Big[\sum_{u\in N(t)} w(X_u(t))^2\Big] + \int_0^t\int_0^a\widetilde{\p}(x,y,t) \widetilde{m_2}(y,t)Q(y,t)\left(\Etilde^{(x,t)}[Z]\right)^2\,\dd y\,\dd t,
 \end{split}
\]
where $\widetilde{m_2}(x,t) = \sum_{k\ge 0} k(k-1) \widetilde{q}(x,t,k).$ Equation \eqref{eq:conditioned_Z2} now follows from \eqref{eq:bbm_2ndmoment} together with \eqref{eq:ptilde_estimate}, \eqref{eq:def_btilde} and  \eqref{eq:conditioned_Z_upperbound}.
\end{proof}

The following lemma gives a good lower bound on $\Etilde^x[Z_t]$. We define
\[
 \widebar{p}_t = \sup_{s\in[0,t]} p(s).
\] 

\begin{lemma}
\label{lem:conditioned_Z_lowerbound}
Suppose $\mu \ge c_0/2$ and $a^2\le t\le a^3$ and $p(s) = 0$ for all $s\ge t$. We have
\begin{equation}
\label{eq:conditioned_Z_lowerbound}
 \Etilde^x[Z_t] \ge w(x)\Big(1-\widebar{p}_t \Const{eq:conditioned_Z_lowerbound}(t/a^3 + (1\wedge (a-x)^{-1})\Big).
\end{equation}
\end{lemma}

In order to prove it, we will need the following estimate on $h(x,0)$, which will be sharpened in Lemma \ref{lem:T0}.

\begin{lemma}
 \label{lem:h_estimate}
\[
1-h(x,0) \le \widebar{p}_t (\pi w(x)t/a^3 + Ce^{\mu(x-a)}) \le \widebar{p}_tC e^{\mu(x-a)} \Big((a-x)t/a^3 + 1\Big) .
\]
\end{lemma}
\begin{proof}
By Markov's inequality, we have
\[
 \begin{split}
  1-h(x,0) &= \P^x(\#\{(u,s)\in \mathscr L_t: X_u(s)=a,\ U_u \le p(t)\} \ge 1) \\
    &\le \E^x(\#\{(u,s)\in \mathscr L_t: X_u(s)=a,\ U_u \le p(t)\})\\
    &\le \widebar{p}_t \E^x(R_t),
 \end{split}
\]
The lemma now follows from Lemma \ref{lem:Rt} and the inequality $\sin x\le x$, $x\in[0,\pi]$.
\end{proof}



\begin{proof}[Proof of Lemma \ref{lem:conditioned_Z_lowerbound}]
Since $p(s) = 0$ for all $s\ge t$ by hypothesis, we have $h(y,t) = 1$ for all $y\in(0,a)$. Lemma \ref{lem:many_to_one_conditioned} and the second property of the Brownian taboo process (see Section \ref{sec:taboo}) now imply
\begin{equation}
 \label{eq:210}
 \Etilde^x[Z_t] \ge w(x)W_{\mathrm{taboo}}^x\Big[e^{-\int_0^t e(X_s,s) \,\dd s}\Big] \ge w(x)\Big(1-W_{\mathrm{taboo}}^x\Big[\int_0^t e(X_s,s) \,\dd s\Big]\Big),
\end{equation}
by the inequality $e^{-x} \ge 1-x$ for $x\ge 0$. By \eqref{eq:def_e}, Lemma \ref{lem:h_estimate} and the hypotheses, we have for every $y\in(0,a)$ and $s\ge 0$,
\[
 e(a-y,s) \le e(a-y,0) \le \widebar{p}_t C e^{-\mu y}(yt/a^3+1) \le \widebar{p}_t C e^{-(c_0/3)y}.
\]
By Lemma \ref{lem:k_integral} and the fact that the law of the Brownian taboo process is preserved under the map $y\mapsto a-y$, this gives
\begin{equation}
\label{eq:220}
 W_{\mathrm{taboo}}^x\Big[\int_0^t e(X_s,s) \,\dd s\Big] \le \widebar{p}_t C \Big(t/a^3 + \operatorname{err}(a-x)\Big).
\end{equation}
The lemma now follows from \eqref{eq:210} and \eqref{eq:220}.
\end{proof}

Finally, we study the law of $R_t$ under the new probability.
\begin{lemma}
 \label{lem:conditioned_Rt}
We have for every $x\in [0,a]$,
\begin{equation}
\label{eq:conditioned_Rt}
 \E^x[R_t] - \widebar{p}_t\E^x[R_t^2] \le \Etilde^x[R_t] \le (h(x,0))^{-1}\E^x[R_t],
\end{equation}
and if there is a $p\in[0,1]$, such that $p(s) \equiv p$ for $s\le t$, then we even have
\begin{equation}
\label{eq:conditioned_Rt_constant}
 \Etilde^x[R_t] \le \E[R_t].
\end{equation}
\end{lemma}

\begin{proof}
Let $\mathscr R_t$ be the stopping line
\[
 \mathscr R_t = \{(u,s)\in \mathscr L_{H_a}: s\le t\}.
\]
We have by definition of the law $\Ptilde$,
\begin{equation}
 \label{eq:230}
 \Etilde^x[R_t] = \frac{\E^x\Big[R_t \prod_{(u,s)\in\mathscr R_t} (1-p(X_u(s)))\Big]}{\E^x\Big[\prod_{(u,s)\in\mathscr R_t} (1-p(X_u(s)))\Big]}.
\end{equation}
Now the denominator is $h(x,0)$ by \eqref{eq:def_gamma}, which yields the right-hand side of \eqref{eq:conditioned_Rt}. The left-hand side follows by noticing that
\[
 \E^x\Big[R_t \prod_{(u,s)\in\mathscr R_t} (1-p(X_u(s)))\Big] \ge \E^x[R_t(1-\widebar{p}_t)^{R_t}] \ge \E^x[R_t] - \widebar{p}_t\E^x[R_t^2].
\]
For \eqref{eq:conditioned_Rt_constant}, we note that if $p(s) \equiv p$ for $s\le t$, then by \eqref{eq:230},
\[
 \Etilde^x[R_t] = \frac{\E^x[R_t(1-p)^{R_t}]}{\E^x[(1-p)^{R_t}]}.
\]
Since $(1-p)^k$ is decreasing in $k$, this yields \eqref{eq:conditioned_Rt_constant}.
\end{proof}

\section{The system before a breakout}
\label{sec:before_breakout}

In this section, we are studying branching Brownian motion with drift $-\mu$ and absorption at $0$ until a \emph{breakout} occurs, an event which will be defined in Section \ref{sec:before_breakout_definitions} and which corresponds to a particle going far to the right and spawning a big number of descendants. In \eqref{eq:decomposition}, we decompose the system into a particle conditioned to break out at a specific time $T$ (this particle will be called the \emph{fugitive}) and the remaining particles, which are conditioned not to break out before time $T$. These two parts will be studied seperately, the former in Section \ref{sec:fugitive} and the latter in Section \ref{sec:before_breakout_particles}. Before that, in Section \ref{sec:time_breakout}, we study the law of the time of the first breakout, showing that it is approximately exponentially distributed. First of all, however, we start with the necessary definitions:

\subsection{Definitions}
\label{sec:before_breakout_definitions}
We will introduce several parameters which will be used during the rest of the paper. The two most important parameters are $a$ and $A$, which are both large positive constants. The meaning of $a$ is as in the previous sections: It is the right border of an interval in which the particles are staying most of the time, and a breakout will be defined below as the event that a particle hits $a$ and then spawns many descendants. The parameter $A$ has a more subtle meaning and controls the number of particles of the system and with it the intensity at which particles hit the point $a$. In Section \ref{sec:moving_barrier}, we will indeed choose the initial conditions such that $Z_0 \approx \kappa e^A$, where $\kappa$ is a fixed constant.

In \cite{Berestycki2010}, the parameter $a$ was called $L_A$ (which we changed for typographical reasons) and $a$ and $A$ were related by the equation
\begin{equation*}
 a = \frac 1 {c_0}\Big(\log N + 3\log \log N - A\Big),
\end{equation*}
where $N$ was a parameter representing the approximate number of individuals in the system. The parameter $A$ then represented a shift of the right barrier. Although this choice of parameters may be more intuitive then ours, we found it technically more convenient to drop the parameter $N$ altogether, and work only with $a$ and $A$ instead.

As in \cite{Berestycki2010}, when we study the system when $a$ and $A$ are large, we will first let $a$ go to infinity, then $A$. Thus, the statement ``For large $A$ and $a$ we have\ldots'' means: ``There exists $A_0$ and a function $a_0(A)$, such that for $A \ge A_0$ and $a \ge a_0(A)$ we have\ldots''. Likewise, the statement ``As $A$ and $a$ go to infinity\ldots'' means ``For all $A$ there exists $a_0(A)$ such that as $A$ goes to infinity and $a\ge a_0(A)$\ldots''. We further introduce the notation $o_A(1)$, which stands for a deterministic term independent of the initial conditions of the process and which goes to $0$ as $A$ and $a$ go to infinity. Furthermore, $o(1)$ will denote a term which goes to $0$ as $a$ goes to infinity (with $A$ fixed).

The remaining parameters we introduce are all going to depend on $A$, but not on $a$. First of all, there is the small parameter $\ep$, which controls the intensity of the breakouts. Indeed, when $Z_0 \approx \kappa e^A$, the mean time one has to wait for a breakout will be approximately proportional in $\ep$. Morally, one could choose $\ep$ such that $e^{-A/2} \ll \ep \ll A^{-1}$, but for technical reasons we will require that
\begin{align}
 \label{eq:ep_upper}
 \ep &\le \Const{eq:ep_upper} A^{-17},\quad \tand\\
 \label{eq:ep_lower}
 \ep &\ge \Const{eq:ep_lower} e^{-A/6}.
\end{align}
Another protagonist is $\eta$, which we will choose as small as we need and which will be used to bound the probability of very improbable events, as well as the contribution of the variable $Y$. It will be enough to require that
\begin{equation}
 \label{eq:eta}
 \eta \le e^{-2A},
\end{equation}
which, by \eqref{eq:ep_lower}, implies
\begin{equation}
 \label{eq:eta_ep}
\eta \le C \ep^{12}.
\end{equation}
The last parameters are $y$ and $\zeta$, which are defined as in Lemma \ref{lem:y_zeta}, with $\eta$ there being the $\eta$ defined above. Note that the parameters $\eta$, $y$ and $\zeta$ appeared already in \cite{Berestycki2010} and had the same meaning there.

We can now proceed to the definition of the process. Recall the definition of $\mu$ in \eqref{eq:mu}. We will always suppose that $a$ is large enough, such that
\begin{equation}
 \label{eq:mu_c0}
 \mu \ge c_0/2.
\end{equation}
As in Section \ref{sec:interval_notation}, we denote by $\P^x$ and $\E^x$ the law and expectation of branching Brownian motion with drift $-\mu$, starting from a particle at the point $x\in\R$, and extend this defintion to general initial distributions of particles according to a counting measure $\nu$. Recall from Section \ref{sec:preliminaries_definition} that $\mathscr N(t)$ denotes the set of individuals alive at time $t$. We want to absorb the particles at 0 and do this formally by setting
\[
 \mathscr N_0(t) = \{u\in \mathscr N(t): H_0(X_u) > t\},
\]
where $H_0$ is again the hitting time functional of 0.

Instead of absorbing particles at $a$, we are now going to classify them into \emph{tiers} in the following way: Particles that have never hit the point $a$ form the particles of tier 0. As soon as a particle hits $a$ it advances to tier 1. Say this happens at a time $\tau_0$. In order to advance to tier 2, a particle has to come back to the critical line $a-y+(c_0-\mu)(t-\tau_0)$ and then back to $a$ again. Here, $y$ is a large positive constant to be defined later.

Formally, let $u\in U$, $t\ge 0$. We define two sequences of random times $(\tau_n(u))_{n\ge-1}$ and $(\sigma_n(u))_{n\ge0}$ by $\tau_{-1}(u) = 0$, $\sigma_0(u) = 0$ and for $n\ge 0$:
\begin{equation}
 \label{eq:def_tau_sigma}
\begin{split}
 \tau_n(u) &= \inf\{s\ge \sigma_n(u):X_u(s) = a\},\\
 \sigma_{n+1}(u) &= \inf\{s\ge \tau_n(u):X_u(s)=a-y+(c_0-\mu)(s-\tau_n(u))\},
\end{split}
\end{equation}
where we set $\inf \emptyset = \infty$. We now define for $t\ge 0$ the stopping lines
\begin{align}
 \label{eq:def_Rline}
\mathscr{R}^{(l)}_t &= \{(u,s)\in U\times[0,t]: s=\tau_l(u)\tand u\in\mathscr{N}_0(s)\},\ l \ge -1,\tand\\
 \label{eq:def_Lline}
\mathscr{S}^{(l)}_t &= \{(u,s)\in U\times[0,t]: s=\sigma_l(u)\tand u\in\mathscr{N}_0(s)\},\ l\ge 0.
\end{align}
That means, $\mathscr{R}^{(l)}_t$ contains the particles of tier $l$ at the moment at which they touch the right barrier and $\mathscr{S}^{(l)}_t$ contains the particles of tier $l$ at the moment at which they come back to the critical line. Note that the sets $\mathscr{R}^{(l)}_t$ and $\mathscr{S}^{(l)}_t$ are increasing in $t$ and $\mathscr{R}^{(l-1)}_t \preceq \mathscr{S}^{(l)}_t \preceq \mathscr{R}^{(l)}_t$ for every $l\ge0$. We also set
\[
R^{(l)}_t = \#\mathscr{R}^{(l)}_t.
\]
In order to extend the definitions of the variables $Z_t$ and $Y_t$ to the current setup, we could simply replace $\mathscr N_{0,a}(t)$ by $\mathscr N_0(t)$ in their definition (see Section \ref{sec:interval_notation}). However, it will be more useful to take special care of the individuals $u$ for which $\tau_l(u) \le t < \sigma_{l+1}(u)$ for some $l\ge 0$, since these are in some kind of ``intermediary'' state which is difficult to analyse. We therefore define the stopping lines
\begin{multline}
 \label{eq:def_Nline_l}
 \mathscr{N}^{(l)}_t = \{(u,s)\in U\times \R_+: u\in\mathscr{N}_0(s),\ \tau_{l-1}(u) \le s < \tau_{l}(u),\\
\text{ and either }t<\sigma_l(u) = s\tor \sigma_l(u) \le t = s\},\ l\ge 0,
\end{multline}
and
\begin{equation}
\label{eq:def_Nline}
 \mathscr{N}_t = \bigcup_{l\ge 0}  \mathscr{N}^{(l)}_t,
\end{equation}
such that $\mathscr{N}^{(l)}_t$ contains the particles of tier $l$ that have already come back to the critical line at time $t$, as well as the descendants of those that haven't, at the moment at which they hit the critical line. We then define for $l\ge 0$ (recall that $w(x) = a\sin(\pi x/a)e^{\mu(x-a)}$),
\[
 Z_t^{(l)} = \sum_{u\in\mathscr{N}^{(l)}_t} w(X_u(t)),\quad Y_t^{(l)} = \sum_{u\in\mathscr{N}^{(l)}_t} e^{\mu(X_u(t)-a)},
\]
For any symbol $S$ and $0\le k\le l$, we define,
\[
 S^{(k;l)} = \sum_{i=k}^l S^{(i)},\quad S^{(l+)} = S^{(l;\infty)},\quad S = S^{(0+)}.
\]

For a particle $(u,s) \in \mathscr{R}^{(l)}_t$, we define the stopping line
\begin{equation*}
  \mathscr S^{(u,s)} = \{(v,r)\in U\times \R_+: v\in \mathscr N_0(r),\ (u,t)\preceq (v,r) \tand r = \sigma_{l+1}(v)\}.
\end{equation*}
This stopping line yields a collection $(X_v(r),r-s)_{(v,r)\in \mathscr S^{(u,s)}}$ of space-time points, and we denote by $Z^{(u,s)}$, $Y^{(u,s)}$ and $W_y^{(u,s)}$ the corresponding quantities from Lemma \ref{lem:critical_line}. Of course, we have chosen the stopping line in such a way that the variable $W_y^{(u,s)}$ follows the same law as the variable $W_y$ defined in \eqref{eq:def_Wy}. We also define $\tau_{\mathrm{max}}^{(u,s)} = \max_{(v,r)\in \mathscr S^{(u,s)}} (r-s)$. We then define the event
\begin{equation}
 \label{eq:def_B}
B^{(u,s)} = \{Z^{(u,s)} > \ep e^A\}\cup \{\tau_{\mathrm{max}}^{(u,s)} > \zeta\},
\end{equation}
which is called the event of a \emph{breakout}, since $e^{-A} Z^{(u,s)}$ measures the number of descendants of the particle $(u,s)$ (the inclusion of the ``bad'' event $\{\tau_{\mathrm{max}}^{(u,s)} > \zeta\}$ is for technical reasons). The particle $u$ is then also called the \emph{fugitive}. We set
\begin{equation}
 \label{eq:def_pB}
p_B = \P^a(B^{(\emptyset,0)}),
\end{equation}
and define the law of BBM started at $a$ with the first particle conditioned not to break out:
\[
 \Q^a(\cdot) = \P^a(\cdot\,|\, B^c) = \frac{\P^a(\cdot,\ B^c)}{1-p_B},
\]
where we set $B = B^{(\emptyset,0)}$. We further set $Z = Z^{(\emptyset,0)}$ and $W_y = W_y^{(\emptyset,0)}$ and note that by Lemmas \ref{lem:critical_line} and \ref{lem:y_zeta}, we have for large $a$,
\begin{equation}
 \label{eq:Z_W}
\P^a(|Z-\frac{\pi}{c_0}W| > 2\eta) + \P^a(\tau_{\mathrm{max}}^{(u,s)} \le \zeta) < \eta,
\end{equation}
where $W$ is defined as in \eqref{eq:def_Wy}. Hence, by \eqref{eq:eta} and \eqref{eq:W_tail}, we get
\begin{equation}
 \label{eq:pB}
p_B = \Big(\frac{\pi}{c_0} + o_A(1) + o(1)\Big)\frac 1 {\ep e^A},
\end{equation}
which goes to 0 as $A$ and $a$ go to infinity, by \eqref{eq:ep_lower}. Furthermore, \eqref{eq:Z_W} yields for large $A$ and $a$,
\begin{equation}
\begin{split}
\label{eq:QaZ}
 \Q^a[Z] &= (\E[\frac \pi {c_0}W\Ind_{(\frac \pi {c_0}W\le \ep e^A(1+o(1)) + O(\eta))}] + O(\eta \ep e^A))(1+O(p_B)) \\
&= \frac{\pi}{c_0}(A+\log\ep+\const{eq:QaZ}+ o_A(1) + o(1)),
\end{split}
\end{equation}
by \eqref{eq:W_expec}, \eqref{eq:ep_upper}, \eqref{eq:eta} and \eqref{eq:pB}. In particular, we have for $A\ge 1$ and large $a$,
\begin{equation}
 \label{eq:QaZC}
\Q^a[Z] \le CA.
\end{equation}
Moreover, by \eqref{eq:W_tail}, \eqref{eq:Z_W} and \eqref{eq:eta}, we have for $A\ge 1$ and large $a$,
\begin{equation}
 \label{eq:QaZsquared}
\Q^a[Z^2] \le C\ep e^A.
\end{equation}

We now define for every $l\in\Znn$ the time of the first breakout of a particle of tier $l$,
\begin{equation}
 \label{eq:def_Tl}
T^{(l)}(\omega) = \inf\{t\ge0: \omega \in \bigcup_{(u,s)\in\mathscr{R}^{(l)}_t} B^{(u,s)}\},
\end{equation}
and set
\begin{equation}
\label{eq:def_T}
 T = \min_l T^{(l)}.
\end{equation}
We denote by $\mathscr U$ the fugitive of the breakout that happened at time $T$.

Now fix $t > 0$. We want to describe the system conditioned on $T=t$. For this, suppose that at time $0$ the particles are distributed according to a counting measure $\nu = \sum_{i=1}^n \delta_{x_i}$. Define $p_i = \P^\nu(i \preceq \mathscr U\,|\,T = t)$, which yields a law $(p_i)_{i=1}^n$ on the initial particles, depending on $\nu$ and $t$. Since the variable $T$, the time of the first breakout, is the minimum of the variables $T_i$, $i=1,\ldots,n$, the times of the first breakout of the BBM descending from the particle $i$, we can decompose the process into
\begin{equation}
\label{eq:decomposition}
 \P^\nu\Big(\prod_{i=1}^n \dd \omega^{(i)}\,\Big|\,T=t\Big) = \sum_{i=1}^n p_i \Big(\P^{x_i}(\dd \omega^{(i)}\,|\,T=t) \times \prod_{j\ne i} \P^{x_j}(\dd \omega^{(j)}\,|\,T>t)\Big).
\end{equation}
That is, we first choose according to the law $(p_i)_{i=1}^n$ the initial particle that is going to cause the breakout. This particle spawns a BBM conditioned to break out at time $t$. The remaining particles spawn independent BBM conditioned not to break out before time $t$.


\subsection{The time of the first breakout}
\label{sec:time_breakout}

We want to prove that the random variable $T$ defined above is approximately exponentially distributed with parameter $p_B\pi Z_0/a^3$, which is the statement of the following proposition:

\begin{proposition}
\label{prop:T}
 Let $0\le t\le a^3/(3\Const{eq:375}A)$ and suppose that $Y_0\le e^{-1}$. Define $\theta = p_B\pi Z_0$. Then, for $A$ and $a$ large enough, we have
\begin{equation}
\label{eq:T}
\P(T> t) = \exp\Big(-\theta t/a^3\Big(1 + O(At/a^3 + p_B)\Big)+O(p_B Y_0)\Big).
\end{equation}
\end{proposition}

The proof proceeds by a sequence of lemmas. Lemma \ref{lem:T0} gives a estimate on $\P(T^{(0)} > t)$. This is used in Lemma \ref{lem:T_a}, in order to obtain an estimate on $\P^a(T > t)$, using a recursive argument. Finally, Proposition \ref{prop:T} is proven by combining Lemmas \ref{lem:T0} and \ref{lem:T_a}.

\begin{lemma}
 \label{lem:T0}
Let $0\le t\le a^3$. Define $\theta$ as in Proposition \ref{prop:T}. Suppose that $p_B \le 1/2$ and $Y_0\le e^{-1}$. Then,
\begin{equation}
\label{eq:T0}
\P(T^{(0)}> t) = \exp\Big(-\theta t/a^3\Big(1 + O(p_B)\Big)+O(p_BY_0)\Big).
\end{equation}
\end{lemma}

Before proving Lemma \ref{lem:T0}, we establish a weaker estimate on $\P(T^{(0)}> t))$.

\begin{lemma}
\label{lem:T0_small}
 Let $0\le t\le a^3$. Define $\theta$ as in Proposition \ref{prop:T}. Suppose that $p_B \le 1/2$ and $Y_0\le 1$. Then
\begin{equation}
\label{eq:T0_small}
 \P(T^{(0)}> t) = \exp\Big(-\theta t/a^3\big(1 + O(p_B(1 + Z_0 t/a^3))\big)+O(p_B Y_0))\Big).
\end{equation}
\end{lemma}
\begin{proof}
We have for $t\ge0$,
\begin{equation}
 \label{eq:308}
  \P(T^{(0)} > t) = \E\Big[\prod_{(u,s)\in\mathscr{R}^{(0)}_t} \Ind_{B^{(u,s)}}\Big] = \E\Big[(1-p_B)^{R^{(0)}_t}\Big],
\end{equation}
since by the strong branching property, the random variables $Z^{(u,s)}$ are independent conditioned on $\mathscr{R}^{(0)}_t$. By Lemma \ref{lem:Rt}, and the assumption $t\le a^3$, we have
\begin{equation}
 \label{eq:310}
|\E[R^{(0)}_t] - \pi Z_0 t/a^3| \le \Const{eq:Rt_nu_expec} Y_0,\quad\tand
\end{equation}
and
\begin{equation}
 \label{eq:315}
\E[(R^{(0)}_t)^2] \le 2 (\pi Z_0 t/a^3)^2 + C (\pi Z_0 t/a^3+Y_0),
\end{equation}
where the last line follows from the inequality $(x+y)^2 \le 2(x^2+y^2)$ and the assumption $Y_0\le 1$. By Jensen's inequality and \eqref{eq:310}, we have
\begin{equation}
 \label{eq:317}
 \E\Big[(1-p_B)^{R^{(0)}_t}\Big] \ge \E\Big[e^{\log(1-p_B)\E[R^{(0)}_t]}\Big] \ge \exp\Big(-\theta t/a^3 + O\left(p_B(\theta t/a^3+Y_0)\right)\Big),
\end{equation}
since $|\log(1-x)| \le x+x^2$ for $x\le 1/2$. This gives the lower bound in \eqref{eq:T0_small}. For the upper bound, equations \eqref{eq:310} and \eqref{eq:315} together with the inequality $(1-p)^n \le 1-np+n(n-1)p^2/2$ give
\begin{equation}
\label{eq:325}
 \E\Big[(1-p_B)^{R^{(0)}_t}\Big] \le 1- \theta t/a^3 + p_BC(\theta t/a^3 + Y_0) + (\theta t/a^3)^2,
\end{equation}
The lemma now follows from \eqref{eq:308}, \eqref{eq:317} and \eqref{eq:325} together with the inequality $1-x \le e^{-x}$.
\end{proof}

\begin{proof}[Proof of Lemma \ref{lem:T0}]
 Let $x_1,\ldots,x_n$ be the positions of the initial particles. Since the initial particles spawn independent branching Brownian motions, we have
\begin{equation}
\label{eq:340}
 \P(T^{(0)} > t) = \prod_i \P^{x_i}(T^{(0)}>t).
\end{equation}
Define $z_i = w(x_i)$ and $y_i = e^{\mu(x_i-a)}$. Then trivially $y_i \le Y_0$ for all $i$, and therefore, since $Y_0 \le e$ by assumption,
\[
 \mu(a-x_i) = |\log y_i| \ge |\log Y_0| \ge 1,\quad\text{ for all $i$}.
\]
As a consequence, by the inequality $\sin x\le x$ for $x\ge0$, we have
\begin{equation}
 \label{eq:345}
z_i = a y_i \sin(\pi x_i/a) \le \pi \mu^{-1} e^{-\mu(a-x_i)} \mu(a-x_i) \le \pi \mu^{-1}Y_0|\log Y_0|,
\end{equation}
since the function $x\mapsto xe^{-x}$ is decreasing for $x\ge 1$. By Lemma \ref{lem:T0_small}, \eqref{eq:mu_c0}, \eqref{eq:340} and \eqref{eq:345} and the hypothesis $t\le a^3$, we now have,
\[
 \P(T^{(0)} > t) = \prod_i\exp\Big(-p_B\pi z_i t/a^3\Big(1 + O(p_B(1+Y_0|\log Y_0|))\Big) + O(p_B y_i)\Big).
\]
Since $Y_0|\log Y_0| \le 1$ by hypothesis, this proves the lemma.
\end{proof}

In the following lemma, note that according to the definition of the tiers, a particle starting at $a$ starts immediately in tier 1.
\begin{lemma}
 \label{lem:T_a}
Let $0\le t\le a^3/(3\Const{eq:375}A)$. Then, for large $A$ and $a$,
\begin{equation}
 \label{eq:T_a}
\Q^a(T > t) \ge \exp\Big(-p_B\pi \tfrac t {a^3} \Q^a[Z]\left(1+O(A\tfrac t {a^3} + p_B)\right) + O(\eta)\Big)
\end{equation}
\end{lemma}
\begin{proof}
We have
\begin{equation}
 \label{eq:360}
\begin{split}
 \Q^a(T > t) &= \Q^a\Big[\prod_{(u,s)\in\mathscr{S}^{(1)}_t}\P^{X_u(s)}(T > t-s)\Big] \ge \Q^a[\P^\nu(T > t)],
\end{split}
\end{equation}
where $\nu = \sum_{(u,s)\in\mathscr{S}^{(1)}_t} \delta_{X_u(s)}$. Since $T>t$ implies $T^{(0)}>t$, we have
\begin{equation}
 \label{eq:362}
\P^\nu(T > t) = \P^\nu(T > t|T^{(0)}>t) \P^\nu(T^{(0)} > t).
\end{equation}
Let $Z$ and $Y$ be as in Lemma \ref{lem:critical_line}. By the definition of $\Q^a$ and Lemma \ref{lem:critical_line}, we have $\Q^a$-almost surely $Y \le \eta \ep e^A$, such that for large $A$, $Y\le e^{-1}$ and $p_BY \le \eta C$ by \eqref{eq:eta} and \eqref{eq:pB}. By Lemma \ref{lem:T0}, we now have for large $A$,
 \begin{equation}
\label{eq:363}
\begin{split}
\P^\nu(T^{(0)} > t) &\ge \exp\Big(-\theta Z\left(1+O(p_B)\right) + O(\eta)\Big).                                                                                                                                                                                                                                                                                                                       \end{split}
\end{equation} 
As for the first factor in \eqref{eq:362}, we have, with the notation from Section \ref{sec:weakly_conditioned}, with $p(s) \equiv p_B$,
\begin{align*}
 \P^\nu(T > t|T^{(0)}>t) &= \widetilde{\P}^\nu(T > t) = \widetilde{\P}^\nu\Big[\prod_{(u,s)\in\mathscr{R}^{(0)}_t}\Q^a(T>t-s)\Big] \ge \widetilde{\P}^\nu\Big[\Q^a(T>t)^{R^{(0)}_t}\Big].
\end{align*}
By Jensen's inequality and \eqref{eq:conditioned_Rt_constant}, this implies
\begin{equation}
\label{eq:364}
 \P^\nu(T > t|T^{(0)}>t) \ge \Q^a(T>t)^{\widetilde{\E}^\nu[R^{(0)}_t]} \ge \Q^a(T>t)^{\E^\nu[R^{(0)}_t]}
\end{equation}
Now, by \eqref{eq:Rt_nu_expec}, we have, by Lemma \ref{lem:critical_line} and $y^{-1} \le \eta$,
\begin{equation}
\label{eq:365}
 \E^\nu[R^{(0)}_t] \le Z \Big(\pi \frac t {a^3} + \eta\Big).
\end{equation} 
Equations \eqref{eq:362}, \eqref{eq:363}, \eqref{eq:364} and \eqref{eq:365}, together with Jensen's inequality, now imply
\begin{multline}
\label{eq:370}
 \Q^a[\P^\nu(T > t)] \ge \Q^a(T>t)^{\Q^a[Z] (\pi t/a^3 + \eta)}\times\exp\Big(-\theta \Q^a[Z]\left(1+O(p_B)\right) + O(\eta)\Big).
\end{multline}
Now, by \eqref{eq:QaZC}, \eqref{eq:360} and \eqref{eq:370}, we have,
\begin{equation}
\label{eq:375}
 \Q^a(T>t)^{1-\delta} \ge \exp\Big(-\theta \Q^a[Z] \left(1+O(p_B)\right) + O(\eta)\Big),
\end{equation} 
with $\delta = \Const{eq:375}(At/a^3 + \eta(1 + A))$. By \eqref{eq:eta} and the hypothesis on $t$, we have $\delta \le 1/2$ for large $A$, whence $(1-\delta)^{-1} \le 1+2\delta.$ Raising both sides in \eqref{eq:375} to the power $(1-\delta)^{-1}$ yields the lemma.
\end{proof}

\begin{proof}[Proof of Proposition \ref{prop:T}]
We have the trivial upper bound $\P(T > t) \le \P(T^{(0)} > t)$, and Lemma \ref{lem:T0} now implies the upper bound in \eqref{eq:T}. For the lower bound, we note that as in the proof of Lemma \ref{lem:T_a}, we have by Jensen's inequality and \eqref{eq:conditioned_Rt_constant},
\begin{equation}
\label{eq:380}
 \P(T > t) = \P(T > t\,|\,T^{(0)} > t)\P(T^{(0)} > t) \ge \Q^a(T > t)^{\E[R_t^{(0)}]}\P(T^{(0)} > t).
\end{equation}
By Lemma \ref{lem:T_a}, and since $At/a^3 = O(1)$ by hypothesis, we have
\[
 \Q^a(T > t) \ge \exp\Big(O(p_B At/a^3 + \eta)\Big), 
\]
and by \eqref{eq:Rt_nu_expec}, we have
\begin{equation*}
 \E[R_t^{(0)}] = \pi t/a^3 Z_0 + O(Y_0),
\end{equation*}
such that, since $At/a^3 = O(1)$ by hypothesis,
\begin{equation}
\label{eq:382}
 \Q^a(T > t)^{\E[R_t^{(0)}]} \ge \exp\Big(O\left(\theta t/a^3(At/a^3 + \eta/p_B) + \eta + Y_0(p_B +\eta)\right)\Big)
\end{equation}
The lower bound in \eqref{eq:T} now follows from \eqref{eq:380}, \eqref{eq:382} and \eqref{eq:T0}, together with the fact that $\eta \le Cp_B^2$ by \eqref{eq:eta} and \eqref{eq:pB}.
\end{proof}

\begin{lemma}
\label{lem:moments_T}
Define $\theta$ as in Proposition \ref{prop:T}. Suppose that $Y_0\le e^{-1}$ and let $\alpha \ge 0$ and $n\in\N$. Then, for large $A$,
\begin{equation}
\label{eq:T_moments}
 \E[(T/a^3 + \alpha)^n\Ind_{(T\le a^3)}] \le \E[(T^{(0)}/a^3 + \alpha)^n\Ind_{(T^{(0)}\le a^3)}] \le C\sum_{k=0}^n\frac{n!\alpha^k}{k!(\theta/2)^{n-k}}
\end{equation}
Furthermore, if $0\ge \beta = o_A(A^{-1})$, then for large $A$ and $a$,
\begin{equation}
 \label{eq:T_expec}
 \E[(T/a^3)\Ind_{(T/a^3) \le \beta}] = \theta^{-1}(1+O(A\beta + p_B)) + O((\beta+\theta^{-1})e^{-O(\beta\theta)})
\end{equation}
\end{lemma}
\begin{proof}
We first note that we have, for $n\ge 0$ and $\gamma > 0$,
\begin{equation}
 \label{eq:398}
\int_0^\infty (t+\alpha)^n e^{-\gamma t}\,\dd t = \sum_{k=0}^n \frac{n!\alpha^k}{k!\gamma^{n+1-k}}.
\end{equation}
Now, we have
\[
\begin{split}
\E[(T^{(0)}/a^3 + \alpha)^n\Ind_{(T^{(0)}\le a^3)}] &= \int_0^1 (t+\alpha)^n\,\P(T^{(0)}/a^3 \in \dd t) \\
&\le n\int_0^1 (t+\alpha)^{n-1}\,\P(T^{(0)}>ta^3)\,\dd t + \alpha^n.
\end{split}
\]
The second inequality of \eqref{eq:T_moments} now follows from Lemma \ref{lem:T0} and \eqref{eq:398}, since $p_B \to 0$ as $A$ goes to infinity. The first inequality follows in the same way, using the trivial fact that $\P(T>ta^3) \le \P(T^{(0)}>ta^3)$. For the second part, we note that
\[
 \E[(T/a^3)\Ind_{(T/a^3) \le \beta}] = \int_0^{\beta} \P(T > ta^3)\,\dd t - \beta \P(T > \beta a^3),
\]
and by Proposition \ref{prop:T} and the hypothesis on $Y_0$, we have for $t\le \beta$ and large $A$ and $a$,
\[
 \P(T > ta^3) = (1+O(p_B))\exp(-\theta t(1+O(A\beta + p_B))).
\]
Equation \eqref{eq:T_expec} now follows from the last two equations.
\end{proof}

We now show how we can couple the variable $T$ with an exponentially distributed variable:
\begin{lemma}
\label{lem:coupling_T}
 Suppose there exists a universal constant $\kappa$, such that $e^{-A}Z_0 = \kappa + O(\ep^{3/2})$ and that $Y_0 \le \eta Z_0$. Then there exists a coupling $(T,V)$, such that $T$ is $\sigma(V)$-measurable and the random variable $V$ is exponentially distributed with parameter $p_Be^A\pi\kappa$ and such that for large $A$ and large $a$, we have $\P(B_{\mathrm{coupl}}) \le C\ep^2$, where $B_{\mathrm{coupl}}$ is the event
\[
 B_{\mathrm{coupl}} = \{|T/a^3 - V| > \ep^{3/2}\} \cup \{|T/(a^3V) - 1| > \sqrt{\ep}) \le \ep^2\}.
\]
\end{lemma}

\begin{proof}
For brevity, we define $\gamma := p_Be^A\pi\kappa$. Let $F$ be the tail distribution function of $T$, i.e. $F(t) := \P(T \ge t)$. It is clear that $T$ has no atoms except $\infty$. We can therefore define a random variable $U$ which is uniformly distributed on $(0,1)$ by setting
\[
  U = F(T)\Ind_{(T<\infty)} + U' F(\infty)\Ind_{(T=\infty)},
\]
where $U'$ is a uniformly distributed random variable on $(0,1)$, independent of $T$. Now we define $V = -\gamma^{-1} \log U$. Then $V$ is exponentially distributed with parameter $\gamma$ and $T = F^{-1}(e^{-\gamma V})$, where $F^{-1}$ denotes the generalized right-continuous inverse of $F$. Hence, $T$ is $\sigma(V)$-measurable. On $\{T<\infty\}$, we have by Proposition \ref{prop:T}, for $a$ large enough,
\begin{equation}
\label{eq:390}
\begin{split}
 V &= -\gamma^{-1} (p_Be^A\pi e^{-A} Z_0 T/a^3(1+ O(AT/a^3 + p_B))) + O(p_BY_0)\\
&= T/a^3(1 + O(\ep^{3/2} + AT/a^3 + p_B)) + O(\kappa p_B e^A\eta),
\end{split}
\end{equation}
by the hypotheses on $Z_0$ and $Y_0$. Hence, by \eqref{eq:ep_lower}, \eqref{eq:eta_ep} and \eqref{eq:pB}, we have for $a$ large enough,
\[
 |T/a^3 - V| = O(\ep^{3/2}T/a^3 + A(T/a^3)^2) + O(\ep^{3/2}).
\]
But now we have by Lemma \ref{lem:T0}, for large $A$ and $a$,
\begin{equation}
\label{eq:392}
 \P(T/a^3 > \ep^{3/4}/\sqrt{A}) \le \P(T^{(0)}/a^3 > \ep^{3/4}/\sqrt{A}) \le Ce^{-O(\ep^{-1/4}/\sqrt{A})} \le \ep^2/2,
\end{equation}
by \eqref{eq:ep_lower}. Furthermore, we get from \eqref{eq:390},
\begin{equation}
 \label{eq:394}
V/(T/a^3) = 1 + O(\ep^{3/2} + AT/a^3) + (O(\ep\eta) + o(1))/(T/a^3),
\end{equation}
and by \eqref{eq:eta_ep}, we have by Proposition \ref{prop:T},
\begin{equation}
 \label{eq:396}
\P(T/a^3 \le \sqrt{\ep}\eta) = O(\eta/\sqrt{\ep}) \le \ep^2/2,
\end{equation}
for large $A$ and $a$. Equations \eqref{eq:390}, \eqref{eq:392}, \eqref{eq:394} and \eqref{eq:396} now prove the lemma.
\end{proof}

\subsection{The particles that do not participate in the breakout}
\label{sec:before_breakout_particles}

In this section, we fix $t\le a^3/(3\Const{eq:375}A)$. We are going to study the system conditioned not to break out until time $t$, the law and expectation of which are denoted as in Section \ref{sec:weakly_conditioned} by $\Phat$ and $\Ehat$, respectively, hence
\[
 \Phat(\cdot) = \P(\cdot\,|\,T > t).
\]
Under the law $\Phat$, the process stopped at $\mathscr L_H \wedge t$ then follows the law $\Ptilde$ from Section \ref{sec:weakly_conditioned}, with
\begin{equation}
 \label{eq:def_ps}
p(s) = p_B\Ind_{(s\le t)} + (1-p_B)\Q^a(T \le t-s),
\end{equation}
such that by Lemma \ref{lem:T_a}, \eqref{eq:eta} and \eqref{eq:QaZC}, for large $A$ and $a$,
\begin{equation}
 \label{eq:pbar_estimate}
\widebar{p}_t \le C p_B.
\end{equation}
As in the proof of Lemma \ref{lem:conditioned_Rt}, one can then show that $\Qhat^a[Z] = (1+O(p_B)) \Q^a[Z]$ and $\Qhat^a[Z^2] \le (1+O(p_B)) \Q^a[Z^2]$, such that by \eqref{eq:QaZ}, \eqref{eq:QaZC} and \eqref{eq:QaZsquared}, we have for large $A$ and $a$,
\begin{equation}
 \label{eq:QhatZ}
 \Qhat^a[Z] = \frac{\pi}{c_0}(A+\log\ep+\const{eq:QaZ}+ o_A(1)),\quad \Qhat^a[Z] \le CA,\quad\tand\quad \Qhat^a[Z^2] \le C\ep e^A
\end{equation}

We define two filtrations $(\G_l)_{l\ge 0}$ and $(\H_l)_{l\ge 0}$ by
\[
 \G_l = \F_{\mathscr{S}^{(l)}_t \wedge t},\quad \H_l = \F_{\mathscr{R}^{(l)}_t \wedge t},
\]
such that  and $\G_l \subset \H_l \subset \G_{l+1}$ for every $l$. Now define
\[
 Z_\emptyset^{(l)} = \sum_{(u,s)\in\mathscr{S}^{(l)}_t} Z^{(u,s)},\quad Y_\emptyset^{(l)} = \sum_{(u,s)\in\mathscr{S}^{(l)}_t} Y^{(u,s)},
\]
and recall from Section \ref{sec:before_breakout_definitions} the definition of $Z_\emptyset^{(k;l)}$, $Z_\emptyset^{(l+)}$ and $Z_\emptyset$ and the corresponding quantities for $Y$.
\begin{lemma}
 \label{lem:Zl_exp_upperbound}
 Suppose $Y_0 \le \eta Z_0$. We have for all $l\ge 1$, and large $A$ and $a$,
\begin{equation}
 \label{eq:Zl_exp_upperbound}
 \Ehat[Z_\emptyset^{(l)}\,|\,\G_{l-1}] \le (\pi+\Const{eq:Zl_exp_upperbound}p_B)\Qhat^a[Z](\tfrac t {a^3} + \Const{eq:Zl_exp_upperbound}\eta) Z_\emptyset^{(l-1)}.
\end{equation}
In particular,
\begin{equation}
 \label{eq:Zl_exp_upperbound_2}
 \Ehat[Z_\emptyset^{(l)}] \le \Big((\pi+\Const{eq:Zl_exp_upperbound}p_B)\Qhat^a[Z](\tfrac t {a^3} + \Const{eq:Zl_exp_upperbound}\eta)\Big)^l Z_0.
\end{equation}
In the case $l=1$, we also have for large $A$ and $a$,
\begin{equation}
 \label{eq:Z1_exp_lowerbound}
 \Ehat[Z_\emptyset^{(1)}] \ge (\pi-\Const{eq:Z1_exp_lowerbound} p_B)\Qhat^a[Z](\tfrac t {a^3} - \Const{eq:Z1_exp_lowerbound}\eta) Z_0.
\end{equation}
\end{lemma}

\begin{proof}
We have
\begin{equation}
\label{eq:Zl_Hl}
 \Ehat[Z_\emptyset^{(l)}\,|\,\H_{l-1}] = \Ehat\Big[\sum_{(u,s)\in\mathscr R^{(l-1)}_t} Z^{(u,s)}\,\Big|\,\H_{l-1}\Big] = \Qhat^a[Z] R_t^{(l-1)},
\end{equation}
since conditioned on $\H_{\mathscr R_t^{(l-1)}\wedge t}$, the random variables $Z^{(u,s)}$ are iid under $\Phat$ of the same law as $Z$ under $\Qhat^a$ and independent of $\H_{\mathscr R_t^{(l-1)}\wedge t}$, by the strong branching property. Now, we have
\begin{equation}
\label{eq:400}
 \Ehat[R_t^{(l-1)}\,|\,\G_{l-1}] = \sum_{(u,s)\in\mathscr S^{(l-1)}_t} \E^{X_u(s)}[R_{t-s}^{(0)}\,|\,T > t-s].
\end{equation}
By \eqref{eq:conditioned_Rt} and \eqref{eq:Rt_nu_expec}, the right-hand side of \eqref{eq:400} is less than or equal to
\[
 \Big(\pi \frac t {a^3} Z_\emptyset^{(l-1)} + \Const{eq:Rt_nu_expec} Y_\emptyset^{(l-1)}\Big)\max_{(u,s)\in \mathscr S^{(l-1)}_t} h(X_u(s),s)^{-1}
\]
and we have $Y_\emptyset^{(l-1)} \le Z_\emptyset^{(l-1)}/y \le Z_\emptyset^{(l-1)} \eta$, $\Phat$-almost surely, by the definition of the event $\Gamma$ in Lemma \ref{lem:critical_line} for $l\ge 2$ and by hypothesis for $l=1$. Furthermore, by Lemma \ref{lem:h_estimate}, \eqref{eq:mu_c0} and \eqref{eq:pbar_estimate}, we have $(h(x,0))^{-1} \le 1+Cp_B$, as soon as $p_B$ is small enough. Combining these inequalities with \eqref{eq:Zl_Hl} and \eqref{eq:400} gives \eqref{eq:Zl_exp_upperbound}. Equation \eqref{eq:Zl_exp_upperbound_2} follows easily from \eqref{eq:Zl_exp_upperbound}. Now, in the case $l=1$, we have $\G_0 = \F_0$ by definition. Let $\nu=\sum_{i=1}^n$ denote the initial configuration. By \eqref{eq:conditioned_Rt} and \eqref{eq:Rt_nu_expec} and \eqref{eq:Rt_nu_variance}, we have
\[
\begin{split}
 \Ehat[R_t^{(0)}] = \sum_{i=1}^n \Ehat^{x_i}[R_t^{(0)}] &\ge \sum_{i=1}^n \E^{x_i}[R_t^{(0)}] - \widebar{p}_t \E^{x_i}[(R_t^{(0)})^2] \\
&\ge \pi \frac t {a^3} Z_0 - \Const{eq:Rt_nu_expec} Y_0 - \widebar{p}_t \Const{eq:Rt_nu_variance}(\frac t {a^3} Z_0 +Y_0)
\end{split}
\]
This yields \eqref{eq:Z1_exp_lowerbound}, since $Y_0 \le \eta Z_0$ by hypothesis.
\end{proof}

In applications of Lemma \ref{lem:Zl_exp_upperbound}, we will often sum the right-hand side of \eqref{eq:Zl_exp_upperbound_2} over all $l\ge 0$. We therefore define $\tconst{eq:def_tcrit}$ to be the solution of
\begin{equation}
 \label{eq:def_tcrit}
 (\pi+\Const{eq:Zl_exp_upperbound}p_B)\Qhat^a[Z](\tconst{eq:def_tcrit}/{a^3} + \Const{eq:Zl_exp_upperbound}\eta) = 1/2.
\end{equation}

We now turn to the variance of $Z^{(l)}_\emptyset$ in the cases that are of interest to us, namely, for $l=1,2$.

\begin{lemma}
 \label{lem:Zl_variance}
 We have for $l\ge 1$ and large $A$ and $a$,
\[
 \Varhat(Z_\emptyset^{(l)}\,|\,\G_{l-1}) \le C\ep e^AZ^{(l-1)}_\emptyset(t/a^3+\eta).
\]
In particular, we have for $l=1$,
\[
  \Varhat(Z_\emptyset^{(1)}) \le C\ep e^AZ_0(t/a^3+\eta).
\]
\end{lemma}
\begin{proof}
 We have 
\begin{equation}
\label{eq:430} 
\begin{split}
\Varhat(Z_\emptyset^{(l)}\,|\,\G_{l-1}) &= \Ehat[\Varhat(Z_\emptyset^{(l)}\,|\,\H_l)\,|\,\G_{l-1}] + \Varhat(\Ehat[Z_\emptyset^{(l)}\,|\,\H_l]\,|\,\G_{l-1})\\
&= \Var_{\Qhat^a}(Z)\Ehat[R^{({l-1})}_t\,|\,\G_{l-1}] + \Qhat^a[Z]\Varhat(R^{({l-1})}_t\,|\,\G_{l-1})
\end{split}
\end{equation}
By Lemma \ref{lem:Rt}, the assumption $t\le a^3$ and the fact that $Y_0 \le Z_0/y$ on $G_0$ (in the case $l=1$), we have
\[
 \Ehat[R^{(l-1)}_t\,|\,\G_{l-1}] \le C(t/a^3+\eta)Z^{(l-1)}_\emptyset,\quad\tand\quad \Varhat(R^{({l-1})}_t\,|\,\G_{l-1}) \le C(t/a^3+\eta)Z^{(l-1)}_\emptyset.
\]
The lemma now follows from these equations, together with \eqref{eq:ep_lower} and \eqref{eq:QhatZ}.
\end{proof}

\begin{lemma}
 \label{lem:Z2_variance}
 Suppose that $Z_0 \le Ce^A$. Then we have for large $A$ and $a$,
\[
 \Varhat(Z_\emptyset^{(2)}) \le Ce^{A}Z_0\big(\ep A (t/a^3+\eta)^2 + \ep A^2 (t/a^3+\eta)^3 + A^4(t/a^3+\eta)^4\big).
\]
\end{lemma}
\begin{proof}
 We have by repeated application of Lemmas \ref{lem:Zl_exp_upperbound} and \ref{lem:Zl_variance},
\begin{align*}
 \Varhat(Z_\emptyset^{(2)}) &= \Varhat(\Ehat[Z_\emptyset^{(2)}\,|\,\G_1]) + \Ehat[\Varhat(Z_\emptyset^{(2)}\,|\,\G_1)] \\
&\le \Ehat[(CA(t/a^3+\eta)Z_\emptyset^{(1)})^2] + \Ehat[C\ep e^A Z_\emptyset^{(1)}(t/a^3+\eta)]\\
&\le CA^2(t/a^3+\eta)^2(\Varhat(Z_\emptyset^{(1)}) + \Ehat[Z_\emptyset^{(1)}]^2) + C\ep e^A (t/a^3+\eta)\Ehat[Z_\emptyset^{(1)}]\\
&\le C(A^2\ep e^A (t/a^3+\eta)^3 Z_0 + A^4(t/a^3+\eta)^4Z_0^2 + \ep e^A A (t/a^3+\eta)^2 Z_0).
\end{align*}
The hypothesis on $Z_0$ now proves the lemma.
\end{proof}

\begin{lemma}
 \label{lem:Zl_t}
We have for all $l\ge 0$ and for large $A$ and $a$,
\[
 \Phat(|Z_t^{(l)} - Z_\emptyset^{(l)}| > K\,|\,\G_l) \le CK^{-2} Z_\emptyset^{(l)} + \Ind_{(Z_\emptyset^{(l)} > CK/p_B)}.
\]
In particular, suppose that $t\le \tconst{eq:def_tcrit}$, then
\[
 \Phat(|Z_t^{(0;2)} - Z_\emptyset^{(0;2)}| > K) \le C K^{-2} Z_0 + Ce^{-A}Z_0/(K\ep).
\]

\end{lemma}
\begin{proof}
We decompose 
\begin{multline}
\label{eq:460}
 \Phat(|Z_t^{(l)} - Z_\emptyset^{(l)}| > 2K\,|\,\G_l) \le \Phat(|Z_t^{(l)} - \Ehat[Z_t^{(l)}\,|\,\G_l]| > K\,|\,\G_l) + \Ind_{(|\Ehat[Z_t^{(l)}\,|\,\G_l] - Z_\emptyset^{(l)}| > K)}.
\end{multline} 
Now, we have by the conditional Chebychev inequality:
\begin{equation}
 \label{eq:462}
\Phat(|Z_t^{(l)} - \Ehat[Z_t^{(l)}\,|\,\G_l]| > K\,|\,\G_l) \le \frac{\Varhat(Z_t^{(l)}\,|\,\G_l)}{K^2}.
\end{equation} 
By \eqref{eq:conditioned_Z2} and \eqref{eq:Zt_2ndmoment},
\begin{equation}
 \label{eq:463}
 \Varhat(Z_t^{(l)}\,|\,\G_l) \le \sum_{(u,s)\in \mathscr S^{(l)}_t} \Ehat^{(X_u(s),s)}[Z^{(0)}_t] \le C Z_\emptyset^{(l)}.
\end{equation}
As for the second term in \eqref{eq:460}, we have by \eqref{eq:Zt_expectation}, \eqref{eq:conditioned_Z_upperbound}, \eqref{eq:pbar_estimate} and Lemmas \ref{lem:conditioned_Z_lowerbound} and \ref{lem:h_estimate},
\begin{equation}
 \label{eq:465}
 |\Ehat[Z_t^{(l)}\,|\,\G_l] - Z_\emptyset^{(l)}| \le Cp_B Z_\emptyset^{(l)},
\end{equation}
since $x e^{-\mu x} \le C$ by \eqref{eq:mu_c0}. Equations \eqref{eq:460}, \eqref{eq:463} and \eqref{eq:465} now finish the proof of the first inequality. The second inequality follows readily by taking expectations and using Lemma \ref{lem:Zl_exp_upperbound} and Markov's inequality.
\end{proof}

\begin{lemma}
\label{lem:Z3_and_Y}
Suppose that $t\le \tconst{eq:def_tcrit}$. Then for large $A$ and $a$,
\begin{equation}
\label{eq:Z3_estimate}
 \Phat(Z^{(3+)}_t > K) \le C K^{-1} A^3(t/a^3+\eta)^3 Z_0,\quad \tand
\end{equation} 
\begin{equation}
 \label{eq:Y_estimate}
 \Phat(Y_t > K) \le CK^{-1}\eta Z_0.
\end{equation} 
\end{lemma}
\begin{proof}
First note that we have $h(x,0) \ge 1/2$ for large $A$. We now have by \eqref{eq:conditioned_Z_upperbound} and \eqref{eq:Zt_expectation},
\[
 \Ehat[Z^{(l)}_t] \le 2 \Ehat[Z_\emptyset^{(l)}] \le C\eta Z_0 (\tfrac t {a^3} + \Const{eq:Zl_exp_upperbound}\eta)^l.
\]
Using the hypothesis, summing over $l\ge 3$ and applying Markov's inequality yields \eqref{eq:Z3_estimate}. For \eqref{eq:Y_estimate}, we note that by \eqref{eq:conditioned_Y} and \eqref{eq:Yt_weak} and $h(x,0) \ge 1/2$
\[
 \Ehat[Y^{(l)}_t] \le C\Ehat[Y_\emptyset^{(l)}] \le C K^{-1}\eta \Ehat[Z_\emptyset^{(l)}] \le C\eta Z_0 2^{-l},
\]
by Lemma \ref{lem:Zl_exp_upperbound} and the hypothesis. Summing over $l\ge0$ and using Markov's inequality finishes the proof.
\end{proof}

\subsection{The fugitive and its family}
\label{sec:fugitive}
We now describe the BBM starting from a single particle and conditioned to break out at a fixed time $t$. We could describe this system by similar methods as those employed in Section \ref{sec:doob}, but since we are only interested in first moment estimates, it is faster to use the Many-to-one lemma instead, which is the method of the proof of the following lemma:

\begin{lemma}
 \label{lem:many_to_one_fugitive}
Let $f:[0,a]\to \R_+$ be measurable, $a^2 \le \tau \le a^3$ and $t \ge \tau$. Let $p:\R_+\to[0,1]$ be measurable with $p(s) = 0$ for $s\ge t$ and $\Const{eq:610}\widebar{p}_t \le 1/4$. Denote by $\Ptilde$ the law associated to $p(s)$ as in Section \ref{sec:weakly_conditioned}. Then,
\begin{multline}
\label{eq:many_to_one_fugitive}
 \Etilde^x\Big[\sum_{(u,s)\in \mathscr L_t}f(X_u(s))\Ind_{(\mathscr U \ne u)}\,\Big|\,\mathscr R_t^{(0)} \cap U\times\{\tau\} \ne 0\Big]\\
 \le \Const{eq:many_to_one_fugitive} W^{x,\tau,a}_{\mathrm{taboo}}\left[\int_0^\tau \E^{X_s}\Big[\sum_{(v,r)\in \mathscr L_{t-s}} f(X_v(r-s))\Big]\,\dd s\right].
\end{multline}
\end{lemma}

\begin{proof}
The left-hand side in \eqref{eq:many_to_one_fugitive} equals
\begin{equation}
 \label{eq:600}
 LHS = \frac{\Etilde^x\Big[\sum_{(u,s)\in \mathscr L_t}\Ind_{(H_a(X_u) \in \dd \tau)} \sum_{(v,r)\in \mathscr L_t} f(X_v(r))\Ind_{(v\ne u)}\Big]}{\Etilde^x_t\Big[\sum_{(u,s)\in \mathscr L_t}\Ind_{(H_a(X_u) \in \dd \tau)}\Big]}
\end{equation}
By Lemma \ref{lem:many_to_one}, the numerator of the right-hand side of \eqref{eq:600} equals
\[
 NUM = \Etilde^{*}_x\Big[\sum_{(v,r)\in \mathscr L_t} f(X_v(r))\Ind_{(v\ne \xi_r)}e^{\int_0^\tau \widetilde m(\xi_s,s)Q(\xi_s,s)\,\dd s} \Ind_{(H_0(\xi) > H_a(\xi)\in\dd \tau)}\Big].
\]
According to the description of the conditioned process in Section \ref{sec:weakly_conditioned} and the description of the spine in Section \ref{sec:spine}, the particles on the spine spawn on average $Q(x,s)\sum_k k(k-1)\widetilde q(x,s,k)\,\dd s$ particles during an interval $[s,s+\dd s]$, which is less than or equal to $m_2 h(x,s)\,\dd s$. Conditioning on the trajectory of the spine and using \eqref{eq:ptilde_estimate} now yields
\[
 NUM \le m_2 e^{m\tau}W^x_{-\mu}\left[\frac{h(a,\tau)}{h(x,0)}\int_0^\tau \E^{X_s}\Big[\sum_{(v,r)\in \mathscr L_{t-s}} f(X_v(r-s))\Big]\,\dd s\,\Ind_{(H_0 > H_a\in\dd \tau)}\right]
\]
Applying Girsanov's theorem to this expression and Lemma \ref{lem:many_to_one_conditioned} to the denumerator in \eqref{eq:600}, we get
\[
 LHS \le \frac{m_2 W^{x,\tau,a}_{\mathrm{killed}}\Big[\int_0^\tau \E^{X_s}\big[\sum_{(v,r)\in \mathscr L_{t-s}} f(X_v(r-s))\big]\,\dd s\Big]}{W^{x,\tau,a}_{\mathrm{killed}}[e^{-\int_0^\tau e(X_s,s)\,\dd s}]},
\]
where $W^{x,\tau,a}_{\mathrm{killed}}$ is the law of a bridge from $x$ to $a$ of length $\tau$ of a Brownian motion killed at $0$ and $a$. But since the taboo process is obtained from the killed BM by a space-time Doob transform, this is the same as $W^{x,\tau,a}_{\mathrm{taboo}}$. As in the proof of Lemma \ref{lem:conditioned_Z_lowerbound}, we have, by the hypotheses on $\tau$,
\begin{equation}
\label{eq:610}
 W^{x,\tau,a}_{\mathrm{taboo}}[e^{-\int_0^\tau e(X_s,s)\,\dd s}] \ge 1-\Const{eq:610}\widebar{p}_t.
\end{equation}
This implies the lemma, by \eqref{eq:pbar_estimate} and the hypothesis on $\widebar{p}_t$.
\end{proof}

We now set up the important definitions. Recall that $\mathscr U$ denotes the fugitive. Define
\[
 \mathscr{\widebar N}(t) = \{u\in \mathscr N_0(t):(u,t)\wedge (\mathscr U,T) \in U\times \bigcup_{l\ge 0}[\sigma_l(\mathscr U),\tau_l(\mathscr U))\},
\]
\[
  \mathscr{\widebar R}(t) = \{(u,s)\in \bigcup_{l\ge 0}\mathscr R^{(l)}(t):(u,s)\wedge (\mathscr U,T) \in U\times \bigcup_{l\ge 0}[\sigma_l(\mathscr U),\tau_l(\mathscr U)))\}
\]
and
\[
 \mathscr{\widecheck N}(t) = \{u\in \mathscr N_0(t):(u,t)\wedge (\mathscr U,T) \in U\times \bigcup_{l\ge 1}[\tau_{l-1}(\mathscr U),\sigma_l(\mathscr U))\}.
\]
We then define
\[
 \widebar Z_t = \sum_{u\in \mathscr{\widebar N}(t)}w(X_u(t)),\quad \widebar Y_t = \sum_{u\in \mathscr{\widebar N}(t)}e^{\mu(X_u(t)-a)},\quad \widebar R_t = \#\mathscr{\widebar R}(t),
\]
and
\[
 \widecheck Z_t = \sum_{u\in \mathscr{\widecheck N}(t)}w(X_u(t)),\quad \widecheck Y_t = \sum_{u\in \mathscr{\widecheck N}(t)}e^{\mu(X_u(t)-a)}.
\]
Note that on the event $T=T^{(0)}$, we have $\mathscr{\widecheck N}(T) = \emptyset$ by definition. For the other particles, we have:
\begin{lemma}
\label{lem:cousins_T0}
Suppose that $t\le \tconst{eq:def_tcrit}$ and $\Const{eq:610}\widebar{p}_t\le 1/4$. Then,
\[ 
\E^x[\widebar Z_t\,|\,T = T^{(0)} = t] \le CA,\quad\E^x[\widebar Y_t\,|\,T = T^{(0)} = t] \le C,\quad \E^x[\widebar R_t\,|\,T = T^{(0)} = t] \le C
\]
\end{lemma}
\begin{proof}
 We have for every $s\ge 0$, $\E^x[Z^{(0)}_s] = w(x) \le \pi(a-x)e^{-\mu(a-x)}$ by \eqref{eq:Zt_expectation}. Furthermore, $\E^x[Y^{(0)}_s] \le Cw(x)/a \le Ce^{-\mu(a-x)}$ by \eqref{eq:Yt_weak}. Finally, we have by Lemma \ref{lem:Rt}, $\E^x[R^{(0)}_s] \le C((a-x)+1)e^{-\mu(a-x)}$ for $s\le a^3$. By Lemmas \ref{lem:many_to_one_fugitive} and \ref{lem:k_integral}, we now have
\[
 \E^x[Z^{(0)}_t\,|\,T = T^{(0)} = t] \le C,\quad \E^x[Y^{(0)}_t\,|\,T = T^{(0)} = t]\le C,\quad \E^x[R^{(0)}_t\,|\,T = T^{(0)} = t] \le C.
\]
From the estimate on $R^{(0)}_t$, it follows that $\E^x[Z^{(1)}_\emptyset\,|\,T = T^{(0)} = t] \le C\Qhat^a[Z] \le CA$, by \eqref{eq:QhatZ}. Hence, by Lemma \ref{lem:Zl_exp_upperbound} and the hypothesis, we have $\E^x[Z^{(1+)}_\emptyset\,|\,T = T^{(0)} = t] \le CA$. By  \eqref{eq:conditioned_Z_upperbound} and \eqref{eq:Zt_expectation}, we now have
\[
 \E^x[Z^{(1+)}_t\,|\,T = T^{(0)} = t] \le CA,
\]
and by \eqref{eq:conditioned_Y} and \eqref{eq:Yt_weak}, we have
\[
 \E^x[Y^{(1+)}_t\,|\,T = T^{(0)} = t] \le C\eta A.
\]
Since $\widebar Z_t = Z^{(0)}_t + Z^{(1+)}_t$, $\widebar Y_t = Y^{(0)}_t + Y^{(1+)}_t$ and $\eta \le A^{-1}$, this implies the lemma.
\end{proof}

On the event $T=T^{(1)}$, the situation is more complex, as shown by the following lemma.
\begin{lemma}
\label{lem:cousins_T1}
Suppose that $t\le \tconst{eq:def_tcrit}$ and $\Const{eq:610}\widebar{p}_t\le 1/4$. Then,
\[ 
\E^x[\widebar Z_t\,|\,T = T^{(1)} = t] \le CA,\quad \E^x[\widebar Y_t\,|\,T = T^{(1)} = t] \le C,\quad\E^x[\widebar R_t\,|\,T = T^{(1)} = t] \le C.
\]
Moreover, on the event $T = T^{(1+)}$, we have $\widecheck Z^{(1)}_\emptyset \le \ep e^A$, and
\begin{align*}
 \E^x[\widecheck Z^{(2+)}_t\,|\,T = T^{(1)} = t] &\le C\ep e^A A(t/a^3+\eta),\\
  \E^x[\widecheck Y_t\,|\,T = T^{(1)} = t] &\le C \ep \eta e^A,\\
 \P^x[|\widecheck Z^{(1)}_t - \widecheck Z^{(1)}_\emptyset| > K\,|\,T = T^{(1)} = t] &\le C \ep e^A K^{-2}(t/a^3+\eta) + C(t/a^3 + \eta)/K.
\end{align*}
\end{lemma}
\begin{proof}
On the event $T = T^{(1)}$, conditioning on $\tau_1(\mathscr U)$, we get two independent pieces of the process, one starting at $x$ conditioned to hit $a$ at $\tau_1(\mathscr U)$, the second starting at $a$, following the law of $\Q^a$ conditioned on $T = T^{(1)} = t - \tau_1(\mathscr U)=: t'$. Stopping this process at the line $\mathscr S^{(1)}_{t'}$, one of those particles then spawns BBM conditioned on $T = T^{(0)} = t'$ and the others spawn BBM conditioned on $T > t'$. Now, $\widebar Z_t$ and $\widebar Y_t$ are the sums of the respective variables corresponding to the two pieces and the inequalities on their expectations now follows from Lemma \ref{lem:cousins_T0}.

On the event $T = T^{(l)}$, for $l\ge 1$, we can generalize the above decomposition and conditioning on $\tau_1(\mathscr U),\dots,\tau_{l-1}(\mathscr U)$ we get $l$ independent pieces of the process. On this event, we note that $\widecheck Z^{(1)}_\emptyset \le \ep e^A$, since no breakout occurred before the time $t=\tau_l(\mathscr U)$. This immediately gives the estimates on the first and second moment of $\widecheck Z^{(1)}_\emptyset$.

For the proof of the remaining inequalities, we note that we have by Lemma \ref{lem:Zl_exp_upperbound},
\[
\E^x[\widecheck Z^{(2+)}_\emptyset] \le \sum_{l=1}^\infty C(\Qhat^a[Z](\tfrac t {a^3} + \Const{eq:Zl_exp_upperbound}\eta))^l \widecheck Z^{(1)}_\emptyset \le C\ep e^A A(\tfrac t {a^3} + \eta),
\]
by the hypothesis on $t$. The last three equations now follow from these results as in Lemmas \ref{lem:Zl_t} and \ref{lem:Z3_and_Y}.
\end{proof}

Define $T^{(l;m)} = \min_{l\le i\le m} T^{(i)}$ and $T^{(l+)} = \min_{i\ge l} T^{(i)}$.
\begin{lemma}
\label{lem:T1_T2}
Suppose that $C_1 e^A \le Z_0 \le C_2e^A$ and $Y_0 \le \eta Z_0$. Then for large $A$,
 \[
  \P(T^{(1+)} < T^{(0)}) \le C\ep A,
 \]
and
\[
 \P(T^{(2+)} < T^{(0;1)}) \le C(\ep A)^2.
\]
\end{lemma}
\begin{proof}
Let $t_0:= \tconst{eq:def_tcrit} \wedge a^3/(3\Const{eq:375}A)$. By \eqref{eq:ep_lower}, \eqref{eq:pB} and \eqref{eq:QhatZ}, we have for large $A$,
\begin{equation}
 \label{eq:777}
\P(T^{(0)} > t_0) \le \exp(-C A/\ep) \le \ep.
\end{equation}
Now, for the rest of the proof, let $t\le t_0$. We have by the decomposition \eqref{eq:decomposition} of the process conditioned on $T=t$,
\[
\begin{split}
 \P^\nu(T^{(1+)} > t\,|\,T^{(0)} = t) &= \sum_{i=1}^n p_i \P^{\nu-\delta_{x_i}}(T^{(1+)} > t\,|\,T^{(0)} > t)\,\P^{x_i}(T^{(1+)} > t\,|\,T^{(0)} = t)\\
&\ge \sum_{i=1}^n p_i \Q^a(T > t)^{\E^{\nu-\delta_{x_i}}[R^{(0)}_t] + \E^{x_i}[R^{(0)}_t\,|\,T^{(0)} = t]},
\end{split}
\]
by Jensen's inequality. By Lemma \ref{lem:Rt} and the hypothesis on $Y_0$, we have $\E^{\nu-\delta_{x_i}}[R^{(0)}_t]\le \E^\nu[R^{(0)}_t]\le C(t/a^3 + \eta)Z_0$. By Lemma \ref{lem:cousins_T0}, we have for large $A$, $\E^{x_i}[R^{(0)}_t\,|\,T^{(0)} = t]\le C$. In total, we get by Lemma \ref{lem:T_a}, for $t\le t_0$,
\begin{equation}
\label{eq:620}
 \P^\nu(T^{(1+)} > t\,|\,T^{(0)} = t) \ge \exp\Big(-C\big(Ap_B(t/a^3 + \eta)^2Z_0 + t/a^3(\eta Z_0 +Ap_B) + \eta + \eta^2 Z_0\big)\Big).
\end{equation}
By \eqref{eq:777}, \eqref{eq:620}, and the inequality $1-e^{-x}\le x$, we get
\begin{multline}
  \P^\nu(T^{(1+)} < T^{(0)}) \le C(\E[(T/a^3 + \eta)^2\Ind_{(T\le t_0a^3)}]Ap_BZ_0\\
 + \E[T/a^3\Ind_{(T\le t_0a^3)}](\eta Z_0 +Ap_B) + \eta + \eta^2 Z_0 + \ep).
\end{multline}
The result now follows by Lemma \ref{lem:moments_T}, together with the hypothesis on $Z_0$, \eqref{eq:ep_lower}, \eqref{eq:eta} and \eqref{eq:pB}. For the second part of the lemma, we first note that by \eqref{eq:decomposition} and the union bound,
\[
\begin{split}
 \P^\nu(T^{(2+)} < t\,|\,T^{(1)} > T^{(0)} = t) &= \sum_{i=1}^n p_i\Big( \P^{\nu-\delta_{x_i}}(T^{(2+)} < t\,|\,T^{(0;1)} > t)\\
&+\P^{x_i}(T^{(2+)} < t\,|\,T^{(1)} > T^{(0)} = t)\Big).
\end{split}
\]
Now we have
\[
\begin{split}
 \P^{\nu-\delta_{x_i}}(T^{(2+)} < t\,|\,T^{(0;1)} > t) &\le \P^\nu(T^{(2+)} < t\,|\,T^{(0;1)} > t)\\
 &\le \widebar{p}_t \E^\nu[R^{(2)}_t\,|\,T^{(0;1)} > t],
\end{split}
\]
by Markov's inequality. As in the proof of Lemma \ref{lem:Zl_exp_upperbound}, we can show that for $l\ge 1$,
\begin{align}
 \label{eq:631}
 \Q^a[R^{(l+1)}_t\,|\,T^{(1;l-1)} > t\,|\,\mathscr \H_1] &\le C A^l(t/a^3+\eta)^l R^{(0)}_t,
\end{align}
since we have, as in $\eqref{eq:QhatZ}$, for every $l\ge 0$, $\Q^a[Z\,|\,T^{(1;l)} > t] = (1+O(p_B)) \Q^a[Z] \le CA$, by \eqref{eq:QaZ}. With \eqref{eq:Rt_nu_expec} and \eqref{eq:pbar_estimate}, this gives
\begin{equation}
 \label{eq:630}
\P^{\nu-\delta_{x_i}}(T^{(2+)} < t\,|\,T^{(0;1)} > t) \le \Const{eq:630} p_B A^2(t/a^3+\eta)^3 Z_0.
\end{equation}
Moreover, we have
\begin{align*}
 \P^x(T^{(2+)} < t\,|\,T^{(1)} > T^{(0)} = t) &\le \widebar{p}_t\E^x[R^{(2)}_t\,|\,T^{(1)} > T^{(0)} = t]\\
&\le p_B C\Q^a[Z]^2(t/a^3+\eta)^2\E^x[R^{(0)}_t\,|\,T^{(0)} = t] && \text{by \eqref{eq:631}, \eqref{eq:pbar_estimate}}\\
&\le p_B CA^2(t/a^3+\eta)^2 && \text{by Lemma \ref{lem:cousins_T0}.}
\end{align*}
In total, this gives
\begin{equation}
\label{eq:660}
\P^\nu(T^{(2+)} < t\,|\,T^{(1)} > T^{(0)} = t) \le Cp_B (A^2 (t/a^3+\eta)^3Z_0 + A^2(t/a^3+\eta)^2)
\end{equation}
Moreover, we have
 \begin{multline}
\label{eq:661}
\P^\nu(T^{(2+)} < t\,|\,T^{(0)} > T^{(1)} = t) \le \sum_{i=1}^n p_i\Big(\P^{\nu-\delta_{x_i}}(T^{(2+)} < t\,|\,T^{(0;1)} > t)\\
+\P^{x_i}(T^{(2+)} < t\,|\,T^{(0)} > T^{(1)} = t)\Big).
\end{multline}
The first term in \eqref{eq:661} has been bounded in \eqref{eq:630}. For the second term, we note that we have
\begin{multline*}
 \P^x(T^{(2+)} < t\,|\, T^{(0)} > T^{(1)} = t) \le \widebar{p}_tC \Big(\Q^a[R^{(2)}_t\,|\,T^{(1)} > t]\E^x[\widebar R^{(0)}_t\,|\,T^{(0)} > T^{(1)} = t]\\
+ \Q^a[R^{(1)}_t] \E^x[\widebar R^{(1)}_t\,|\,T^{(0)} > T^{(1)} = t]\Big),
\end{multline*}
and by Lemma \ref{lem:cousins_T1}, together with \eqref{eq:ep_lower}, \eqref{eq:pB}, \eqref{eq:pbar_estimate} and the hypotheses on $Z_0$ and $Y_0$, we get
\begin{equation}
\label{eq:663}
 \P^x(T^{(2+)} < t\,|\, T^{(0)} > T^{(1)} = t) \le  Cp_B\Big(A^2(t/a^3+\eta)^2 + \ep e^A A(t/a^3+\eta)^2\Big) \le C A (t/a^3+\eta)^2.
\end{equation}
Equations \eqref{eq:630}, \eqref{eq:661} and \eqref{eq:663} now yield
\begin{equation}
\label{eq:666}
\P^\nu(T^{(2+)} < t\,|\,T^{(0)} > T^{(1)} = t) \le Cp_BA^2 (t/a^3+\eta)^3Z_0 + C A (t/a^3+\eta)^2,
\end{equation}
and \eqref{eq:660} and \eqref{eq:666} then yield
\begin{equation}
\label{eq:668}
 \P^\nu(T^{(2+)} < t\,|\,T^{(0;1)} = t) \le Cp_BA^2 (t/a^3+\eta)^3Z_0 + C A (t/a^3+\eta)^2.
\end{equation}
The second part of the lemma now follows from \eqref{eq:668}, by integrating over $t$ from $0$ to $t_0$ and using Lemma \ref{lem:moments_T} and \eqref{eq:777}.
\end{proof}

\section{The system with the moving barrier}
\label{sec:moving_barrier}

We will now define properly the BBM with the moving barrier. We will still use all the definitions from Section \ref{sec:before_breakout_definitions}, with one notational change: Recall that by \eqref{eq:decomposition}, we can decompose the process into two parts: the first part consisting of the particles spawned by the ancestor of the fugitive, and the second part consisting of the remaining particles. As in Section \ref{sec:fugitive}, the quantities which refer to the particles of the first part will be denoted by a bar (e.g.\ $\widebar Z$) or check (e.g. $\widecheck Z$). The quantities of the second part will be denoted with a hat in this section (e.g.\ $\widehat Z$), in reference to the law $\Phat$ from Section \ref{sec:before_breakout_particles}. Furthermore assume from now on that there is a constant $\kappa$, such that for each $A$ and $a$ large enough the initial distribution satisfies $|e^{-A}Z_0-\kappa|\le \ep^{3/2}$ and $Y_0\le \eta Z_0$. The constant $\kappa$ will be regarded as universal, in the sense that the terms denoted by $O()$, $o_A()$ and $o(1)$ may depend on $\kappa$.

Suppose further that we are given a family $(f_x)_{x\ge 0}$ of non-decreasing functions $f_x \in \mathscr C^2(\R,\R_+)$, such that for each $x\ge 0$, $f_x(t) = 0$ for $t\le 0$, $f_x(+\infty) = x$ and for each $\delta > 0$ there exist $M_x = M_x(\delta)$, $M_t = M_t(\delta)$, such that
\begin{itemize}
 \item $M_x(\delta) \to \infty$ as $\delta \to 0$,
 \item $||f_x|| \le \delta^{-1}$ for all $x\in [0,M_x]$, and
 \item $f_x(t) \ge x-\delta$ for all $t\ge M_t$,
\end{itemize}
where $||f||$ is defined in \eqref{eq:def_f_norm}. It is easy to construct such a family: Take any non-decreasing function $f\in \mathscr C^2(\R_+,\R_+)$, such that $f(t)=0$ for all $t\le 0$ and $f(t) = 1$ for all $t\ge 1$ and define $f_x = xf$ for $x\ge 0$. Then $||f_x|| \le ||f|| (x\vee x^2)$, whence this family satisfies the above conditions with $M_x(\delta) = (||f||\delta)^{-1/2} \wedge (||f||\delta)^{-1}$ and $M_t(\delta) \equiv 1$.

Now suppose we are given a BBM with constant drift $-\mu$ starting from the initial configuration $\nu_0$. We are going to define for each $n\in\N$ define a stopping time $T_n$ and a barrier process $(X^{(n)}_t)_{t\in [T_{n-1},T_n]}$ as follows:
\begin{enumerate}
 \item We set $T_0 = 0$ and $X^{(1)}_0 = 0$.
 \item Denote by $T$ the time of the first breakout of the BBM absorbed at 0 and by $\mathscr U$ the fugitive, as in Section \ref{sec:before_breakout_definitions}. We set $X^{(1)}_t = 0$ for $t\in [0,T]$.
\item
Define
\begin{equation}
\label{eq:Delta_def}
\Delta = \frac{1}{c_0} \log\Big(\frac{Z_T}{\kappa e^A}\vee 1\Big),
\end{equation}
where $Z_T$ is defined in Section \ref{sec:before_breakout_definitions}.
 \item Define $T_1' = T+\tau^{(\mathscr U,T)}_{\mathrm{max}}$ and $T_1 = (T + a^{5/2})\vee T_1'$. Note that $T_1'$ and therefore also $T_1$ is a stopping time for the BBM. Now define
\[
X^{(1)}_t = f_{\Delta}((t-T_1')/a^2), t\in[T,T_1].
\]
We then give the particles an additional drift $-(\dd/\dd t)X^{(1)}_t$ for $t\in [T_1',T_1]$, in the meaning of Section \ref{sec:interval_notation}.
 \item We have now defined $T_1$ and $X^{(1)}$. We further define $\nu_1$ to be the measure formed by the particles at time $T_1$, which have never hit $0$. To define $T_2$ and $X^{(2)}$, we repeat the above steps with the process formed by the BBM started from those particles, with the definitions changed such that the barrier process starts at $X^{(2)}_{T_1} = X^{(1)}_{T_1}$, time starts at $T_1$ etc.
\item We now construct the barrier process $X^{(\infty)}_t$ from the pieces by $X^{(\infty)}_t = X^{(n)}_t$, if $t\in [T_{n-1},T_n]$.
\end{enumerate}
\begin{remark}
 The random line formed by the particles at time $T_1$ which have never crossed the barrier $X^{(1)}$ is not a stopping line in the sense that we have defined it, but in Jagers' sense (see Section \ref{sec:stopping_lines}), such that the strong branching property applies here as well. It is even a simple stopping line in the terminology of Biggins and Kyprianou \cite{Biggins2004}.
\end{remark}

Recall the definition of the phrase ``As $A$ and $a$ go to infinity'' from Section \ref{sec:before_breakout_definitions}. Our main theorem is the following:

\begin{theorem}
 \label{th:barrier}
As $A$ and $a$ go to infinity, the process $(X_t)_{t\ge 0} = (X^{(\infty)}_{ta^3c_0^2/\pi^2} - A t)_{t\ge 0}$ converges in the sense of finite-dimensional distributions to the L\'evy process $(L_t)_{t\ge0}$ with $L_0 = 0$ and cumulant $K_\kappa(\lambda)$ given by
\begin{equation}
\label{eq:cumulant_levy}
 K_\kappa(\lambda) = \log E[e^{i\lambda L_1}] = i\lambda (\log \kappa + c) + c_0\int_0^\infty e^{i\lambda x} - 1 - i\lambda x\Ind_{(x\le 1)}\,\Lambda(\dd x),
\end{equation}
where $\Lambda(\dd x)$ is the image of the measure $x^{-2}\dd x$ by the map $x\mapsto c_0^{-1}\log(1+x)$ and $c\in\R$ is a constant depending only on the reproduction law $q(k)$.
\end{theorem}

A stronger convergence than convergence in the sense of finite-dimensional distributions is convergence in law with respect to Skorokhod's topology (see \cite{Ethier1986}, Chapter 3). Obviously, the convergence in Theorem \ref{th:barrier} does not hold in this stronger sense, because the barrier is continuous but the L\'evy process is not, and the set of continuous functions is closed in Skorokhod's topology. However, if we create artificial jumps, we can rectify this:

\begin{theorem}
 \label{th:barrier2}
Define $J_t = X^{(\infty)}_{T_n}$, if $t\in [T_n,T_{n+1})$, for $n\in\N$. Then as $A$ and $a$ go to infinity, the process $(X'_t)_{t\ge 0} = (J_{ta^3c_0^2/\pi^2}-At)_{t\ge 0}$ converges in law with respect to Skorokhod's topology to the L\'evy process defined in the statement of Theorem \ref{th:barrier}.
\end{theorem}

Define the sequence $(G_n)_{n\ge -1}$ of ``good events'' by $G_{-1} = \Omega$ and $G_n$ to be the intersection of $G_{n-1}$ with the following events:
\begin{itemize}
 \item $\nu_n$ has support in $(0,a)$,
 \item $\mathscr N_{T_n} \subset U\times \{T_n\}$ and $T_n > T_n'$ (for $n>0$).
 \item $|e^{-A}Z_{T_n} -\kappa| \le \ep^{3/2}$ and $Y_{T_n} \le \eta Z_{T_n}$. 
\end{itemize}
The core of the proof of Theorems \ref{th:barrier} and \ref{th:barrier2} will be the following proposition:
\begin{proposition}
\label{prop:piece}
Fix $\lambda \in \R$. Suppose that $\P(G_0) = 1$. Define $\gamma_0 = \pi/(c_0^2p_Be^A)$.
Then there exists $\delta > 0$, such that for $n\le \ep^{-1-\delta/2}$ and large $A$ and $a$, we have $\P(G_n) \ge 1-n\ep^{1+\delta}$ and
\begin{equation}
\label{eq:piece_fourier}
 \log \E[e^{i\lambda X_{T_n}}] = n \kappa^{-1}\gamma_0(K_\kappa(\lambda) + i\lambda A + o_A(1) + O(\ep^\delta)),
\end{equation} 
with $K_\kappa(\lambda)$ defined as in Theorem \ref{th:barrier} and where $o_A(1)$ and $O(\ep^\delta)$ may depend on $\lambda$.
\end{proposition}
\begin{remark}
 The process $Z_t$ approximately measures the population size, in the sense that the number of particles at the time $t + \delta a^3$, where $\delta$ is a small constant, is approximately $2\pi c_0^{-2} e^{c_0a}a^{-3}\times Z_t$ \cite{Berestycki2010}. That is why we stated in the introduction that the number of particles stays almost constant in our model. We could indeed show for example that the process of the number of particles converges to a constant in the sense of finite-dimensional distributions. Since we plan to show much stronger results in a second work, we will not prove this here, however.
\end{remark}

\subsection{Proof of Proposition \ref{prop:piece}}

In this subsection, we are under the hypotheses of Proposition \ref{prop:piece}, i.e.\ we suppose that $\nu_0$ has support in $(0,a)$, $|e^{-A}Z_0 -\kappa| \le \ep^{3/2}$ and $Y_0 \le \eta Z_0$.

\paragraph{The particles on the stopping line $\mathscr N_T$.} In a first step, we will describe the state of the system at the stopping line $\mathscr N_T$, defined in \eqref{eq:def_Nline}. Recall that this stopping line consists of those particles, for which $\sigma_l \le T < \tau_l$ for some $l$, and of the descendants of those for which $\tau_l \le T < \sigma_{l+1}$ for some $l$, as soon as they hit the critical line. This latter case applies in particular to the fugitive $\mathscr U$, for which $T = \tau_l(\mathscr U)$ for some $l$. We will show in this paragraph that the following events happen with high probability and give first and second moment estimates of the quantities appearing in the definitions:
\begin{align*}
 G_{\mathrm{bulk}} &= \{|Z_T - Z^{(\mathscr U,T)} - (\widehat Z^{(0;2)}_\emptyset + \widecheck Z^{(1)}_\emptyset)| \le \ep^{3/2}/4|\} \cap \{Z_T - Z^{(\mathscr U,T)} \le e^A/\ep\} \cap \{Y_T \le \ep^{-3/2}\},\\
 G_{\mathrm{fug}} &= \{Z^{(\mathscr U,T)} \le e^A/\ep,\ \tau_{\max}^{(\mathscr U,T)} \le \zeta\}.
\end{align*}

First of all, since the process of the descendants of the fugitive after time $T$ follows the law $\P^a(\cdot\,|\,(\Gamma^{(\emptyset,0)})^c)$, we have for $A$ and $a$ large enough,
\begin{equation}
 \label{eq:745}
\P(G_{\mathrm{fug}}^c) \le p_B^{-1}(\P^a(Z^{(\emptyset,0)} > e^A/\ep) + \P^a(\tau_{\max}^{(\emptyset,0)} >  \zeta)) \le C\ep^2,
\end{equation}
by \eqref{eq:ep_lower} and \eqref{eq:eta}.

As for $G_{\mathrm{bulk}}$, let $A$ be large enough, such that $\sqrt \ep \le \tconst{eq:def_tcrit}/a^3$. Recall the decomposition of the BBM conditioned on $T=t$ given by \eqref{eq:decomposition} and denote by $x_0$ the position of the particle that is the ancestor of the fugitive. We have as in the proof of Lemma \ref{lem:T0}, for large $A$,
\[
w(x_0) \le \pi \mu^{-1} Y_0|\log Y_0| \le C \eta |\log \eta| e^A \le \ep^{3/2}e^A,
\]
by \eqref{eq:eta}, whence $|e^{-A}\widehat Z_0 -\kappa| \le 2\ep^{3/2}$ and $\widehat Y_0 \le \eta\widehat Z_0$. We then have by Lemma \ref{lem:Zl_exp_upperbound}, for $t\le \sqrt{\ep} a^3$,
\begin{equation}
  \label{eq:710}
 \E[e^{-A}\widehat Z^{(1;2)}_\emptyset\,|\,T=t] = (\pi + O(p_B))e^{-A}Z_0\Qhat^a[Z](\tfrac t {a^3} + O(\eta)) + O(A^2 (\tfrac t {a^3} + \eta)^2).
\end{equation}
Furthermore, by Lemmas \ref{lem:Zl_variance} and \ref{lem:Z2_variance} and the inequality $(x+y)^2 \le 2(x^2+y^2)$, we have
\begin{equation}
\label{eq:712}
 \Var(e^{-A}\widehat Z^{(1;2)}_\emptyset\,|\,T=t) \le C\big(\ep (\tfrac t {a^3} + O(\eta)) + \ep A (\tfrac t {a^3} + O(\eta))^2 + \ep A^2 (\tfrac t {a^3} + O(\eta))^3 + A^4(\tfrac t {a^3} + O(\eta))^4\big)
\end{equation}
Lemma \ref{lem:moments_T} and \eqref{eq:710} now give for large $A$,
\begin{equation}
 \label{eq:714}
\begin{split}
 \E[e^{-A}\widehat Z^{(1;2)}_\emptyset\Ind_{(T\le \sqrt{\ep}a^3)}] &= e^{-A}/p_B(1+O(A\sqrt{\ep}))\Qhat^a[Z] + O(A^2\ep^2)\\
&= c_0\gamma_0 (A+\log \ep+c+o_A(1)),
\end{split}
\end{equation}
by \eqref{eq:ep_upper}, \eqref{eq:eta_ep} and \eqref{eq:QhatZ}. Note that by \eqref{eq:pB}, $\gamma_0 = \ep(1+o_A(1)+o(1))$.
Similarly, \eqref{eq:710} and \eqref{eq:712} and Lemma \ref{lem:moments_T} give
\begin{equation}
 \label{eq:716}
 \E((e^{-A}\widehat Z^{(1;2)}_\emptyset)^2\Ind_{(T\le \sqrt{\ep}a^3)}] = O(A^2\ep^2).
\end{equation}
Likewise, if we define
\[
 \widehat G = \{|\widehat Z_T^{(0;2)}-\widehat Z_\emptyset^{(0;2)}| \le \ep^2 e^A,\ \widehat Z_T^{(3+)} \le \ep^{3/2}/10,\ \widehat Z_\emptyset^{(0;2)} < e^A/(4\ep),\ \widehat Y_T \le \ep^2 \},
\]
then we have by \eqref{eq:714}, Lemmas \ref{lem:moments_T}, \ref{lem:Zl_t} and \ref{lem:Z3_and_Y}, the hypotheses on $Z_0$ and $Y_0$ and the union bound,
\begin{equation}
 \label{eq:717}
 \P(\widehat G^c,\ T\le \sqrt{\ep}a^3) \le C (\ep^{-2}e^{-A} + A^3 \ep^{3/2} + e^{-A} + \eta e^A) \le C A^3 \ep^{3/2},
\end{equation}
by \eqref{eq:ep_lower} and \eqref{eq:eta}. As for the particles from the family of the fugitive, note first that we have by Lemma \ref{lem:T1_T2},
\begin{equation}
 \label{eq:720}
 \P(T^{(2+)} = T) \le \P(T^{(2+)} < T^{(0;1)}) \le C A^2 \ep^2.
\end{equation}
Furthermore, by Lemmas \ref{lem:cousins_T1} and \ref{lem:T1_T2} and the fact that $\widecheck Z^{(1)}_\emptyset = 0$ on the event $\{T=T^{(0)}\}$,
\begin{equation}
 \label{eq:721}
 \E[e^{-A}\widecheck Z^{(1)}_\emptyset,\ T\le \sqrt{\ep}a^3] \le C \sqrt{\ep} \P(T^{(1+)} = T) \le CA\ep^{3/2},
\end{equation}
and likewise
\begin{equation}
 \label{eq:722}
 \E[(e^{-A}\widecheck Z^{(1)}_\emptyset)^2,\ T\le \sqrt{\ep}a^3] \le C \sqrt{\ep} \P(T^{(1+)} = T) \le CA\ep^{5/2}.
\end{equation}
Likewise, if we define
\[
 \widecheck G = \{|\widecheck Z_T^{(1)}-\widecheck Z_\emptyset^{(1)}| \le \ep^2 e^A,\ \widecheck Z_T^{(2+)} \le \ep^{3/2}/10,\ \widecheck Z_\emptyset^{(1)} < e^A/(4\ep),\ \widecheck Y_T \le \ep^2 \},
\]
then we have by \eqref{eq:720}, \eqref{eq:721}, Lemmas \ref{lem:cousins_T1} and \ref{lem:T1_T2} and Markov's inequality,
\begin{equation}
 \label{eq:727}
\begin{split}
 \P(\widecheck G^c,\ T\le \ep^{3/4}a^3) &\le \P(\widecheck G^c,\ T\le \ep^{3/4}a^3\,|\,T = T^{(1)})\P(T = T^{(1)}) + \P(T = T^{(2+)})\\
 &\le C\ep A (\ep^{-3}e^{-A} + A \ep^{1/4} + e^{-A} + \eta e^A /\ep) + A^2 \ep^2\\
 &\le C A^2 \ep^{5/4},
\end{split}
\end{equation}
by \eqref{eq:ep_lower} and \eqref{eq:eta}. Finally, defining
\[
 \widebar G = \{\widebar Z_T \le \ep^{-3/2},\ \widebar Y_T \le \ep^{-3/2}/2\},
\]
we get by Lemmas \ref{lem:cousins_T0} and \ref{lem:cousins_T1} and Markov's inequality,
\begin{equation}
 \label{eq:737}
 \P(\widebar G^c) \le C A \ep^{3/2}.
\end{equation}
Altogether, since $Z_T = Z^{(\mathscr U,T)} + \widehat Z^{(0+)}_T + \widecheck Z^{(1+)}_T + \widebar Z_T$, we have $\widehat G \cap \widecheck G \cap \widebar G \subset G_{\mathrm{bulk}}$ for large $A$, and thus, by \eqref{eq:ep_upper}, \eqref{eq:717}, \eqref{eq:727} and \eqref{eq:737},
\begin{equation}
 \label{eq:740}
 \P(G_{\mathrm{bulk}}^c,\ T\le \ep^{3/4}a^3) \le C A^2 \ep^{5/4}.
\end{equation}

\paragraph{The particles touching the right barrier after the breakout.} Now, from the time $T+\zeta$ on, we are moving the barrier according to the function $f_\Delta$, which is equivalent to having the variable drift $-\mu_t = -\mu-f_\Delta(t/a^2)/a^2$. Note that on $G_{\mathrm{fug}}$, the variable $\Delta$ is $\F_{T+\zeta}$-measurable and that $T_1 = T+a^{5/2}$. Since $Z_T \le 2e^A/\ep$ on $G_{\mathrm{bulk}}\cap G_{\mathrm{fug}}$, we now have for large $a$, by the hypotheses on the functions $(f_x)$,
\begin{equation}
 \label{eq:Delta_f}
\ton G_{\mathrm{bulk}}\cap G_{\mathrm{fug}}: ||f_\Delta|| \le \sqrt a\quad\tand\quad \Delta - f_\Delta((a^{5/2}-\zeta)/a^2) = o(1).
\end{equation}

We now show that on the good events defined above, with high probability there is no particle hitting the right barrier between the times $T+a^2$ and $T+a^{5/2}$ and the descendants of the particles that hit the right barrier between $T$ and $T+a^2$ are negligible. For this, \emph{we start afresh the notation of the tiers from the stopping line $\mathscr N_T$ on}, indicating this change of notation by a prime ('), i.e. for all particles $u$, such that $\mathscr N_T \preceq (u,t)$ for some $t$, we set $\sigma'_0(u)$ to be the second coordinate of $\mathscr N_T \wedge (u,t)$ and define $\sigma'_n$ and  $\tau'_n$ by
\[
\begin{split}
 \tau_n'(u) &= \inf\{s\ge \sigma_n:X_u(s) = a\},\\
 \sigma_{n+1}'(u) &= \inf\{s\ge \tau_n:X_u(s)=a-y+(c_0-\mu)(s-\tau_n)-f_\Delta((s-(T+\zeta))/a^2)\}.
\end{split}
\]
The stopping lines $\mathscr R'^{(l)}_t$ etc. are then defined as in Section \ref{sec:before_breakout_definitions}, adding $f_\Delta$ to the definition. Note that we assumed there that $f\equiv 0$, but we will not use the results from that section, such that there is no conflict. We then define
\[
\begin{split}
 G_{\mathrm{barrier}} &= \bigcap_{(u,s)\in \mathscr R'^{(0)}_{T+a^2}} \Gamma^{(u,s)} \cap \{R'^{(0)}_{T_1} - R'^{(0)}_{T+a^2} = 0\}\\
& \cap \{Z'^{(1)}_{T_1} \le \ep^2 e^A\} \cap \{Y'^{(1)}_{T_1} \le \ep\eta e^A\} \cap \{R'^{(1)}_{T_1} =0\}.
\end{split}
\]
Now, first note that on $G_{\mathrm{bulk}} \cap G_{\mathrm{fug}}$, we have by \eqref{eq:eta},
\begin{equation}
\label{eq:755} 
Z_T \le 2\ep^{-1}e^A,\quad Y_T \le 2\ep^{-3/2},\quad\tand \quad \forall (u,s)\in \mathscr N_T: s\le T+\zeta.
\end{equation}
It follows from Lemmas \ref{lem:Rt} and \ref{lem:R_f}, \eqref{eq:pB} and \eqref{eq:755} that for large $a$,
\begin{equation}
\label{eq:756}
 \P\Big(G_{\mathrm{bulk}} \cap G_{\mathrm{fug}} \cap \bigcup_{(u,s)\in \mathscr R'^{(0)}_{T+a^2}} (\Gamma^{(u,s)})^c\Big) \le p_B \E[R'^{(0)}_{T+a^2}]
\le C\ep^{-5/2}e^{-A},
\end{equation}
and
\begin{equation}
 \label{eq:757}
\P(G_{\mathrm{bulk}} \cap G_{\mathrm{fug}},\ R'^{(0)}_{T_1} - R'^{(0)}_{T+a^2} \ge 1) = o(1).
\end{equation}
As for the tier 1 particles, if we set
\[
 G' = G_{\mathrm{bulk}} \cap G_{\mathrm{fug}} \cap \bigcap_{(u,s)\in \mathscr R'^{(0)}_{T+a^2}} \Gamma^{(u,s)},
\]
we have by Lemmas \ref{lem:Rt} and \ref{lem:R_f},
\begin{equation}
\label{eq:758}
\begin{split}
 \E\Big[\Ind_{G'}\sum_{(u,s)\in \mathscr R'^{(0)}_{T+a^2}} Z^{(u,s)}\Big] &\le CA\E\Big[\Ind_{G'}R'^{(0)}_{T+a^2}\Big] \le C A (\ep^{-3/2} + o(1)) \le C\ep^{2+5/4}e^A,
\end{split}
\end{equation}
by \eqref{eq:ep_lower}. Equations \eqref{eq:757} and \eqref{eq:758} together with Proposition \ref{prop:quantities}, Lemmas \ref{lem:R_f} and \ref{lem:Rt} and Markov's inequality now give for large $a$,
\begin{equation}
 \label{eq:760}
 \P\Big(G'\cap (\{Z'^{(1)}_{T_1}>\ep^2e^A\}\cup\{Y'^{(1)}_{T_1} > a^{-1/2} e^A\} \cup \{R'^{(1)}_{T_1}\ge 1\})\Big) \le C\ep^{5/4}.
\end{equation}
Equations \eqref{eq:756}, \eqref{eq:757} and \eqref{eq:760}, together with \eqref{eq:ep_lower}, \eqref{eq:740} and \eqref{eq:745} now prove that
\begin{equation}
 \label{eq:750}
\P(G_{\mathrm{bulk}} \cap G_{\mathrm{fug}} \cap G_{\mathrm{barrier}}^c) \le C\ep^{3/2}.
\end{equation}

\paragraph{The particles that stay in the interval $(0,a)$ after the breakout.} On $G_{\mathrm{bulk}} \cap G_{\mathrm{fug}}$, we have $Z_T \ge \kappa e^A$ for large $A$, since $Z^{(\mathscr U,T)}\ge \ep e^A$ on $G_{\mathrm{fug}}$ and $Z_T - Z^{(\mathscr U,T)} \ge \kappa e^A - \ep^{3/2} (e^A + 1/4)$ on $G_{\mathrm{bulk}}$ by the hypothesis on $Z_0$. It follows that
\begin{equation}
\label{eq:767}
\ton G_{\mathrm{bulk}} \cap G_{\mathrm{fug}} : \Delta = \frac 1 {c_0} \log\Big(\frac{Z_T}{\kappa e^A}\Big) \ge 0.
\end{equation}
By \eqref{eq:Delta_f} and Proposition \ref{prop:quantities}, we now have for large $a$,
\begin{equation}
 \label{eq:770}
\ton G_{\mathrm{bulk}}\cap G_{\mathrm{fug}}: |\E[Z'^{(0)}_{T_1}\,|\,\F_{\mathscr N_T}] - \kappa e^A| \le \ep^{3/2}e^A/3,
\end{equation}
as well as
\begin{equation}
 \label{eq:771}
\ton G_{\mathrm{bulk}}\cap G_{\mathrm{fug}}: |\Var[Z'^{(0)}_{T_1}\,|\,\F_{\mathscr N_T}]| \le C\ep^{-3/2} + o(1).
\end{equation}
Equations \eqref{eq:770} and \eqref{eq:771} and the conditional Chebychev inequality now give for large $a$:
\begin{equation}
 \label{eq:772}
\P(G_{\mathrm{bulk}}\cap G_{\mathrm{fug}},\ |Z'^{(0)}_{T_1} - \kappa e^A| > \ep^{3/2}e^A/2) \le C \ep^{-9/2}e^{-2A} \le C\ep^{3/2},
\end{equation}
by \eqref{eq:ep_lower}. Hence, for large $a$, we have by \eqref{eq:772} and \eqref{eq:Yt},
\begin{equation}
 \label{eq:775}
\P(G_{\mathrm{bulk}}\cap G_{\mathrm{fug}} \cap G_{\mathrm{int}}^c) \le C\ep^{3/2},
\end{equation}
where
\[
 G_{\mathrm{int}} = \{|e^{-A}Z'^{(0)}_{T_1} - \kappa| \le \ep^{3/2}/2\}\cap\{Y_{T_1} \le a^{-1/2} e^A\}.
\]

\paragraph{The probability of $G_n$.}
Equations \eqref{eq:745}, \eqref{eq:740}, \eqref{eq:750} and \eqref{eq:775} now give for large $A$ and $a$,
\begin{equation}
\label{eq:776}
 \P\Big((G_{\mathrm{bulk}}\cap G_{\mathrm{fug}} \cap G_{\mathrm{barrier}} \cap G_{\mathrm{int}})^c\Big) \le C A^2\ep^{5/4} + \P(T > \ep^{3/4}e^A) \le \ep^{9/8}/2,
\end{equation}
by \eqref{eq:ep_lower} and Lemma \ref{lem:T0}. Now note that on $G_{\mathrm{fug}}\cap G_{\mathrm{barrier}}$, the first and second points in the definition of $G_1$ from the statement of Proposition \ref{prop:piece} are verified for large $a$, and the third point is verified on $G_{\mathrm{barrier}} \cap G_{\mathrm{int}}$ for large $a$.
In total, we have for large $A$ and $a$,
\begin{equation*}
\P(G_1) \ge 1-\ep^{9/8}.
\end{equation*}
The statement for the probability of $G_n$ now follows readily by induction, since conditioned on the event $G_n$, the process started at the stopping time $T_n$ satisfies the hypotheses of the proposition.

\paragraph{The Fourier transform of the barrier process.}
Define
\begin{align*}
 \Delta_{\mathrm{drift}} &= (\kappa e^A)^{-1} Z_0 - 1 + (\kappa e^A)^{-1} (Z^{(1)}_\emptyset + Z^{(2)}_\emptyset + \widecheck Z^{(1)}_\emptyset),\tand\\
 \Delta_{\mathrm{jump}} &= (\kappa e^A)^{-1} Z^{(\mathscr U,T)},
\end{align*}
and note that $\Delta_{\mathrm{drift}}$ and $\Delta_{\mathrm{jump}}$ are independent random variables.
By the definition of $G_{\mathrm{bulk}}$ and the hypothesis on $Z_0$, we have for large $A$,
\begin{equation}
 \label{eq:778}
\ton G_{\mathrm{bulk}}: \Delta = \frac 1 {c_0} \log\Big(1+\Delta_{\mathrm{drift}}+\Delta_{\mathrm{jump}}\Big) + O(\ep^{3/2})
\end{equation}
Note that by \eqref{eq:714} and \eqref{eq:721} and the hypothesis on $Z_0$, we have
\begin{equation}
 \label{eq:Delta_drift}
\E\left[\Delta_{\mathrm{drift}}\Ind_{(T\le \sqrt\ep a^3)}\right] = \kappa^{-1}c_0\gamma_0(A+\log\ep+c+o_A(1)),
\end{equation}
and by \eqref{eq:716} and \eqref{eq:722} and the inequality $(x+y+z)^2 \le 3(x^2+y^2+z^2)$
\begin{equation}
 \label{eq:Delta_drift_2}
\E\left[(\Delta_{\mathrm{drift}})^2\Ind_{(T\le \sqrt\ep a^3)}\right] = O(\ep^2A^2).
\end{equation}
Note that \eqref{eq:Delta_drift_2} implies
\begin{equation}
 \label{eq:Delta_drift_bound}
\P\left(|\Delta_{\mathrm{drift}}| \ge \ep^{1/3},\ T\le \sqrt\ep a^3\right) = O(\ep^{4/3}A^2) = O(\ep^{7/6}),
\end{equation}
by \eqref{eq:ep_upper}.
Now, since $\log(1+a+b) = \log(1+a)+\log((1+b)/(1+a))$, we have by \eqref{eq:Delta_f}, \eqref{eq:776}, \eqref{eq:778} and \eqref{eq:Delta_drift_bound}, for large $A$ and $a$,
\begin{equation}
\label{eq:782}
\begin{split}
 \E[e^{i\lambda X_{T_1}}] &= \E[e^{i\lambda X_{T_1}}\Ind_{(T\le \sqrt\ep a^3,\ |\Delta_{\mathrm{drift}}| < \ep^{1/3})}\Ind_{G_{\mathrm{bulk}}}] + O(\ep^{9/8})\\
&= \E[e^{(i\lambda /c_0) \log(1+\Delta_{\mathrm{drift}})}\Ind_{(T\le \sqrt\ep a^3,\ |\Delta_{\mathrm{drift}}| < \ep^{1/3})}\Ind_{G_{\mathrm{bulk}}}e^{(i\lambda /c_0)\log(1+\frac{\Delta_{\mathrm{jump}}}{1+\Delta_{\mathrm{drift}}})}] + O(\ep^{9/8}),
\end{split}
\end{equation}
for any $\lambda \in \R$. We will first study the term concerning $\Delta_{\mathrm{jump}}$. Write $Z = Z^{(\mathscr U,T)}$ and let $\rho$ be a real-valued constant with $|\rho| < \ep^{1/3}$. Then,
\begin{equation}
 \label{eq:785}
\begin{split}
\E[e^{(i\lambda /c_0)\log(1+\frac{\Delta_{\mathrm{jump}}}{1+ \rho})}] & = \E[e^{(i\lambda /c_0)\log(1+\frac{\kappa^{-1}e^{-A}Z}{1+ \rho})}\,|\,Z > \ep e^A] + O(\P^a(\tau_{\max}^{(\emptyset,0)} >  \zeta))\\
& = \int_\ep^\infty g(x)\,\P(e^{-A}Z \in \dd x\,|\,Z>\ep e^A) + O(\ep^2),
\end{split}
\end{equation}
where
\[
 g(x) = \exp\Big(\frac {i\lambda}{c_0} \log\Big(1+ \frac{\kappa^{-1}x}{1+ \rho}\Big)\Big).
\]
By definition of $p_B$ and $\eta$, we have for large $A$ and $a$,
\[
\begin{split}
 \int_\ep^\kappa x \P(e^{-A}Z \in \dd x\,|\,Z>\ep e^A) &= (p_B+O(\eta))^{-1}\E[Z\Ind_{(\ep e^A<Z\le \kappa e^A)}]\\
 &= c_0\gamma_0(-\log\ep +\log \kappa + o_A(1)),
\end{split}
\]
by \eqref{eq:W_expec}, \eqref{eq:ep_lower}, \eqref{eq:eta}, \eqref{eq:pB} and \eqref{eq:Z_W}. It follows that
\begin{multline}
\label{eq:786}
\int_\ep^\infty g(x)\,\P(e^{-A}Z \in \dd x\,|\,Z>\ep e^A) = 1 + i\lambda\frac{\kappa^{-1}\gamma_0}{1+ \rho}(-\log\ep+\log \kappa + o_A(1))\\
 + \int_\ep^\infty g(x) - 1 - i\lambda \frac{\kappa^{-1}\gamma_0}{1+ \rho} x \Ind_{(x\le \kappa)}\P(e^{-A}Z \in \dd x\,|\,Z>\ep e^A).
\end{multline}
Now define $h(x) =  g(x) - 1 - i\lambda (\kappa^{-1}\gamma_0/(1+ \rho))x \Ind_{(x\le \kappa)}$ for $x\ge 0$ and denote by $h^-(x)$ its left-hand derivative. Note that  $|h(x)|\le C(1\wedge x^2)$ and $|h^-(x)| \le C(x^{-1}\wedge x^2)$ for $x\ge 0$. Now, by integration by parts,  \eqref{eq:W_tail} and \eqref{eq:Z_W}, we have for large $a$,
\begin{equation}
\label{eq:786a}
\begin{split}
  &\int_\ep^\infty h(x)\P(e^{-A}Z \in \dd x\,|\,Z>\ep e^A)\\
&= h(\ep) + p_B^{-1} (1+o_A(1))\Big(\int_\ep^\infty h^-(x)\P(Z > xe^A)\,\dd x + (h(1)-h(1-)) \P(Z > e^A)\Big)\\
&= c_0\gamma_0 \Big(\int_0^\infty h^-(x)\frac 1 x\,\dd x + (h(1)-h(1-)) + o_A(1)\Big)\\
&= c_0\gamma_0 (1+o_A(1)) \int_0^\infty h(x)\frac 1 {x^2}\,\dd x.
\end{split}
\end{equation}
Now, one readily sees that
\begin{equation}
\label{eq:786b}
 \int_0^\infty h(x)\frac 1 {x^2}\,\dd x = \kappa^{-1} \Big(\frac{i\lambda}{c_0} (c' + o_A(1)+O(\rho))\\
 + \int_0^\infty e^{i\lambda x}-1-i\lambda x\Ind_{(x\le 1)}\,\Lambda(\dd x)\Big),
\end{equation}
where $\Lambda(\dd x)$ is as in the statement of Proposition \ref{prop:piece} and $c'$ is a constant depending only on $c_0$. Equations \eqref{eq:785}, \eqref{eq:786}, \eqref{eq:786a} and \eqref{eq:786b} and the Taylor expansion of $e^{-x}$ at $x=0$ now yield
\begin{multline}
 \label{eq:789}
\E[e^{(i\lambda /c_0)\log(1+\frac{\Delta_{\mathrm{jump}}}{1+ \rho})}] \\
= \exp \Big[\frac {\gamma_0} \kappa \Big(i\lambda (-\log\ep + \log \kappa + c' + o_A(1) + c_0\int_0^\infty e^{i\lambda x}-1-i\lambda x\Ind_{(x\le 1)}\,\Lambda(\dd x)\Big) +O(\ep |\log\ep|\rho)\Big].
\end{multline}
Coming back to \eqref{eq:782}, we have by the Taylor expansion of $(1+x)^{i\lambda/c_0}$ at $x=0$,
\begin{align*}
&\E[e^{(i\lambda /c_0) (\log(1+\Delta_{\mathrm{drift}})+O(\ep |\log\ep| \Delta_{\mathrm{drift}}))}\Ind_{(T\le \sqrt\ep a^3,\ |\Delta_{\mathrm{drift}}| < \ep^{1/3})}\Ind_{G_\mathrm{bulk}}]\\
& = \E[(1+\frac {i\lambda}{c_0} \Delta_{\mathrm{drift}} + O(\Delta_{\mathrm{drift}}^2))\Ind_{(T\le \sqrt\ep a^3)}] + O(\ep^{9/8}) && \text{by \eqref{eq:776} and \eqref{eq:Delta_drift_bound}}\\
& = 1 + i\lambda \kappa^{-1}\gamma_0(A+\log\ep+c+o_A(1)) + \P(T> \sqrt\ep a^3) + O(\ep^{9/8}) && \text{by \eqref{eq:Delta_drift} and \eqref{eq:Delta_drift_2}}\\
& = \exp i\lambda\kappa^{-1}\gamma_0(A+\log\ep+c+o_A(1)+O(\ep^{1/8})),
\end{align*}
where the last equation follows from Lemma \ref{lem:T0} and the Taylor expansion of $e^x$ at $x=0$. This equation, together with
\eqref{eq:782} and \eqref{eq:789} and the fact that $\Delta_{\mathrm{jump}}$ is independent from $\Delta_{\mathrm{drift}}$, $T$ and $G_{\mathrm{bulk}}$, yields \eqref{eq:piece_fourier} in the case $n=1$. For general $n$, we note that
\begin{equation*}
\E[e^{i\lambda X_{T_n}}\Ind_{G_n}] = \E[e^{i\lambda X_{T_n}}\Ind_{G_{n-1}}] - \E[e^{i\lambda X_{T_n}}\Ind_{G_{n-1}\backslash G_n}].
\end{equation*}
Now, by \eqref{eq:piece_fourier} in the case $n=1$, we have
\[
\begin{split}
 \E[e^{i\lambda X_{T_n}}\Ind_{G_{n-1}}] &= \E[\E[e^{i\lambda (X_{T_n}-X_{T_{n-1}})}\,|\,\F_{T_{n-1}}]e^{i\lambda X_{T_{n-1}}}\Ind_{G_{n-1}}]\\
&= \E[e^{i\lambda X_{T_{n-1}}+iO(\ep^{1+\delta})}\Ind_{G_{n-1}}]e^{\ep c_0 \kappa^{-1}(K_\kappa(\lambda) + i\lambda A + o_A(1))}\\
&= \E[e^{i\lambda X_{T_{n-1}}}\Ind_{G_{n-1}}]e^{\ep c_0 \kappa^{-1}(K_\kappa(\lambda) + i\lambda A + o_A(1))} + O(\ep^{1+\delta})
\end{split}
\]
and
\[
 |\E[e^{i\lambda X_{T_n}}\Ind_{G_{n-1}\backslash G_n}]| \le \P(G_{n-1}\backslash G_n).
\]
The statement now follows easily by induction over $n$ and the previously established fact that $\P(G_n) \ge 1-nO(\ep^{1+\delta})$.

\subsection{Proof of Theorems \ref{th:barrier} and \ref{th:barrier2}}
We set $\gamma = \gamma_0/(\kappa c_0)$ and define the process $(X_t'')_{t\ge 0}$ by
\[
X_t'' = X_{T_{\lfloor t\gamma^{-1}a^3\rfloor}} - A t.
\] 
\begin{proposition}
 \label{prop:Xt_second}
The process $(X_t'')_{t\ge 0}$ converges in law (with respect to Skorokhod's topology) to the L\'evy process $(L_t)_{t\ge 0}$ defined in Theorem \ref{th:barrier}.
\end{proposition}
\begin{proof}
Denote by $(\F_t'')_{t\ge0}$ the natural filtration of the process $X_t''$, and note that $\F_t'' = \F_{\gamma \lfloor t\gamma^{-1}\rfloor}'' \subset \F_{T_{\lfloor t\gamma^{-1}\rfloor}}$. In order to show convergence of the finite-dimensional distributions, it is enough to show (see Proposition 3.1 in \cite{Kurtz1975} or Lemma 8.1 in \cite{Ethier1986}, p.\ 225), that for every $\lambda\in\R$ and $t,s\ge 0$,
\begin{equation}
\begin{split}
\label{eq:910}
 \E\Big[\Big|\E[e^{i\lambda X_{t+s}''}\,|\,\F_t''] - e^{i\lambda X_t''}e^{sK_\kappa(\lambda)}\Big|\Big] \to 0,
\end{split}
\end{equation} as $A$ and $a$ go to infinity. Now, define $n := \lfloor t\gamma^{-1}\rfloor$ and $m := \lfloor (t+s)\gamma^{-1}\rfloor$. Then we have by Proposition \ref{prop:piece},
\begin{equation}
\label{eq:912}
 \begin{split}
  \E[e^{i\lambda (X_{t+s}''-X_t'')}\,|\,\F_{T_m}]\Ind_{G_m} &= e^{-i\lambda As} \E[e^{i\lambda (X_{T_m} - X_{T_n})}\,|\,\F_{T_m}]\Ind_{G_n}\\
&= \exp\Big((m-n)\gamma\big(K_\kappa(A) + i\lambda A + o_A(1) + O(\ep^\delta)\big) - i\lambda As\Big)\Ind_{G_n}.\\
&= \exp\Big(s\big(K_\kappa(A) + o_A(1) + O(\ep^\delta)\big)\Big)\Ind_{G_n},
 \end{split}
\end{equation}
because we have $|(m-n)\gamma - s|\le \gamma = A^{-1}o_A(1)+o(1)$, by \eqref{eq:ep_upper} and \eqref{eq:pB}. In total, we get for $A$ and $a$ large enough,
\[
 \E\Big[\Big|\E[e^{i\lambda (X_{t+s}''-X_t'')}\,|\,\F_t''] - e^{sK_\kappa(\lambda)}\Big|\Big] \le e^{s K_\kappa(\lambda)}\E[|e^{s(o_A(1) + O(\ep^\delta))} - 1|] + \P(G_m^c).
\]
By Proposition \ref{prop:piece}, this goes to $0$ as $A$ and $a$ go to infinity, which proves \eqref{eq:910}.

In order to show tightness in Skorokhod's topology, we use Aldous' famous criterion \cite{Aldous1978} (see also \cite{Billingsley1999}, Theorem 16.10): If for every $M > 0$, every family of $(\F_t'')$-stopping times $\tau = \tau(A,a)$ taking only finitely many values, all of which in $[0,M]$ and every $h = h(A,a)\ge 0$ with $h(A,a) \to 0$ as $A$ and $a$ go to infinity, we have
\begin{equation}
\label{eq:920}
 X_{\tau+h}'' - X_{\tau}'' \to 0,\quad\text{ in probability as $A$ and $a$ go to infinity},
\end{equation}
then tightness follows for the processes $X_t''$ (note that the second point in the criterion, namely tightness of $X_t''$ for every fixed $t$, follows from the convergence in finite-dimensional distributions proved above). Now let $\tau$ be such a stopping time and let $V_\tau$ be the (finite) set of values it takes. We first note that since $G_n \supset G_{n+1}$ for every $n\in \N$, we have for every $t\in V_\tau$ and every $A$ and $a$ large enough,
\begin{equation}
\label{eq:922}
\P(G^c_{\lfloor t \gamma^{-1}\rfloor}) \le \P(G^c_{\lfloor M \gamma^{-1}\rfloor}) = O(M \ep^\delta).
\end{equation}
by Proposition \ref{prop:piece}. Moreover, since $\F_t'' \subset \F_{T_{\lfloor t\gamma^{-1}\rfloor}}$ for every $t\ge 0$, we have for every $\lambda > 0$,
\begin{align*}
 \E[e^{i\lambda(X_{\tau+h}'' - X_{\tau}'')}] &= \sum_{t\in V_\tau}\E\Big[e^{i\lambda(X_{t+h}'' - X_t'')}\Ind_{(\tau = t)}\Big]\\
&= \sum_{t\in V_\tau}\E\Big[\E[e^{i\lambda(X_{t+h}'' - X_t'')}\,|\,\F_{T_{\lfloor t\gamma^{-1}\rfloor}}]\Ind_{(\tau = t)}\Ind_{G_{\lfloor t\gamma^{-1}\rfloor}}\Big] + O(M \ep^\delta) && \text{by \eqref{eq:922}}\\
&= e^{h (K_\kappa(\lambda)+o_A(1) + O(\ep^\delta))}(1-O(M\ep^\delta)) + O(M \ep^\delta), && \text{by \eqref{eq:912}},
\end{align*}
which converges to $1$ as $A$ and $a$ go to infinity. This implies \eqref{eq:920} and therefore proves tightness in Skorokhod's topology, since $M$ was arbitrary. Together with the convergence in finite-dimensional distributions proved above, the lemma follows.
\end{proof}

\paragraph{A coupling with a Poisson process.} Let $(V_n)_{n\ge 0}$ be a sequence of independent exponentially distributed random variable with parameter $p_B e^A \pi \kappa$. In order to prove convergence of the processes $X_t'$ and $X_t$, we are going to couple the BBM with the sequence $(V_n)$ in the following way: Suppose we have constructed the BBM until time $T_{n-1}$. Now, on the event $G_{n-1}$, by Lemma \ref{lem:coupling_T}, the strong Markov property of BBM and the transfer theorem (\cite{Kallenberg1997}, Theorem 5.10), we can construct the BBM up to time $T_n$ such that
\begin{equation}
\label{eq:929}
\P(G_{\mathrm{coupl},n}) \ge 1-O(\ep^2),
\end{equation}
where
\[
G_{\mathrm{coupl},n} = \{|(T^{\mathrm{BO}}_n-T_{n-1})/a^3 - V_n| \le \ep^{3/2}\} \cap \{|(T^{\mathrm{BO}}_n-T_{n-1})/(a^3V_n) - 1| \le \sqrt{\ep}) \le \ep^2\},
\]
where $T^{\mathrm{BO}}_n$ is here the time of the first breakout after $T_{n-1}$. On the event $G_{n-1}^c$, we simply let the BBM evolve independently of $(V_j)_{j\ge n}$. Now, define
\[
 G_n' = G_n \cap \bigcap_{1\le j\le n} G_{\mathrm{coupl},j}.
\]
Then, on $G_n'$, we have $T_n = T^{\mathrm{BO}}_n + a^{5/2}$, whence for large $A$ and $a$,
\begin{equation}
\label{eq:930}
 \ton G_n': |(T_n-T_{n-1})/a^3 - V_n| \le 2\ep^{3/2}\quad\tand\quad|(T_n-T_{n-1})/(a^3V_n) - 1|\le 2 \sqrt{\ep}.
\end{equation} 
Furthermore, by \eqref{eq:929} and Proposition \ref{prop:piece}, there is a $\delta > 0$, such that for large $A$ and $a$,
\begin{equation}
 \label{eq:prob_Gnprime}
\P(G_n') \ge 1 - nO(\ep^{1+\delta})
\end{equation}

\begin{proof}[Proof of Theorem \ref{th:barrier2}]
Let $d$ denote the Skorokhod metric on $D([0,\infty))$ (see \cite{Ethier1986}, Section 3.5). Let $\Phi$ be the space of strictly increasing, continuous, maps of $[0,\infty)$ onto itself. Let $x,x_1,x_2,\ldots$ be elements of $D([0,\infty))$. Then (\cite{Ethier1986}, Proposition 3.5.3), $d(x_n,x) \to 0$ as $n\to\infty$ if and only if for every $M > 0$ there exists a sequence $(\varphi_n)$ in $\Phi$, such that
\begin{equation}
 \label{eq:940}
\sup_{t\in [0,M]}|\varphi_n(t) - t| \to 0,
\end{equation}
and
\begin{equation}
 \label{eq:941}
\sup_{t\in [0,M]} |x_n(\varphi_n(t)) - x(t)| \to 0.
\end{equation}
If $(x_n')_{n\in \N}$ is another sequence of functions in $D([0,\infty))$, with $d(x_n',x)\to 0$, then by the triangle inequality and the fact that $\Phi$ is stable under the operations of inverse and convolution, we have $d(x_n,x)\to 0$ if and only if there exists a sequence $(\varphi_n)$ in $\Phi$, such that \eqref{eq:940} holds and
\begin{equation}
 \label{eq:942}
\sup_{t\in [0,M]} |x_n(\varphi_n(t)) - x_n'(t)| \to 0.
\end{equation}

For every $A$ and $a$, we define the (random) map $\varphi_{A,a} \in \Phi$ by
\[
 \varphi_{A,a}(t) = ((1-r) T_n + r T_{n+1})\frac{\pi^2}{c_0^2a^3}, \tif t = \gamma (n+r),\text{ with } n\in\N,\ r\in [0,1].
\]
Let $M > 0$ and define $n_M = \lceil M\gamma \rceil$.
Then we have
\begin{equation}
 \label{eq:944}
\sup_{t\in [0,M]} |\varphi_{A,a}(t) - t| \le \max_{n\in\{0,\ldots,n_M\}} \left|\frac{\pi^2}{c_0^2a^3}T_n - \gamma n\right|,
\end{equation}
and
\begin{equation}
 \label{eq:945}
\sup_{t\in [0,M]}|X''_t - X'_{\varphi_{A,a}(t)}| \le \max_{n\in\{0,\ldots,n_M\}} A \left|\frac{\pi^2}{c_0^2a^3}T_n - \gamma n\right|.
\end{equation}
Now note that $\gamma = \frac{\pi^2}{c_0^2} E[V_1]$, and by Doob's $L^2$ inequality we get
\[
 \P\Big(\max_{n\in\{0,\ldots,n_M\}} \left|\sum_{i=1}^n V_i - nE[V_1]\right| > \ep^{1/3}\Big) \le 4 \ep^{-2/3} n_M \Var(V_i) = O(\ep^{1/3}).
\]
Furthermore, on the set $G_{n_M}'$, we have for every $n\le n_M$,
\[
\left|T_n - \sum_{i=0}^n V_i\right| \le O(n_M \ep^{3/2}) = O(\ep^{1/2}).
\]
In total, we get with \eqref{eq:944} and \eqref{eq:945}, as $A$ and $a$ go to infinity,
\begin{equation}
\label{eq:948}
 \forall M > 0:\sup_{t\in [0,M]} |\varphi_{A,a}(t) - t| \vee |X''_t - X'_{\varphi_{A,a}(t)}| \to 0,\quad\text{in probability},
\end{equation}
which is equivalent to 
\begin{equation}
\label{eq:949}
 \sum_{M\ge 1} 2^{-M}\Big[1\wedge\Big(\sup_{t\in [0,M]} |\varphi_{A,a}(t) - t| \vee |X''_t - X'_{\varphi_{A,a}(t)}|\Big)\Big] \to 0,\quad\text{in probability}.
\end{equation}
Now, suppose that $A$ and $a$ go to infinity along a sequence $(A_n,a_n)_{n\in\N}$ and denote by $X'_{A_n,a_n}$, $X''_{A_n,a_n}$ and $\varphi_{A_n,a_n}$ the processes corresponding to these parameters. By Proposition \ref{prop:Xt_second}, and Skorokhod's representation theorem (\cite{Billingsley1999}, Theorem 6.7), there exists a probability space, on which the sequence $(X''_{A_n,a_n})$ converges almost surely as $n\to\infty$ to the limiting L\'evy process $L = (L_t)_{t\ge 0}$ stated in the theorem. Applying again the representation theorem as well as the transfer theorem, we can transfer the processes $X'_{A_n,a_n}$ and $\varphi_{A_n,a_n}$ to this probability space in such a way that the convergence in \eqref{eq:949} holds almost surely, which implies that the convergence in \eqref{eq:948} holds almost surely as well. By the remarks at the beginning of the proof, it follows that on this new probability space,
\[
 d(X'_{A_n,a_n},L) \le d(X'_{A_n,a_n},X''_{A_n,a_n}) + d(X''_{A_n,a_n},L)\to 0,
\]
almost surely, as $n\to\infty$. This proves the theorem.
\end{proof}

\begin{proof}[Proof of Theorem \ref{th:barrier}]
 By the virtue of Theorem \ref{th:barrier2}, it suffices to show that for every $0 < t_1<t_2<\ldots<t_k$ we have
\begin{equation}
\label{eq:950}
 \P\Big(\forall i: X^{(\infty)}_{t_ia^3} = J_{t_ia^3}\Big) \to 1.
\end{equation}
Let $n := \lceil 2(t_k+2)/E[V_1]\rceil$, such that $n = O(\ep^{-1})$, by \eqref{eq:pB}. By Chebychev's inequality, we then have
\begin{equation}
 \label{eq:952}
\P(\sum_{i=1}^n V_i \le t_k+2) \le \P(\sum_{i=1}^n (V_i - E[V_i]) \le -\frac n 2 E[V_1]) = O(n\Var(V_i)) = O(\ep).
\end{equation}
Furthermore, define the intervals $I_i = t_i + [-2n\ep^{3/2}-a^{-1/2},2n\ep^{3/2}]$, $i=1,\ldots,k$ and denote by $\mathscr P$ the point process on the real line with points at the positions $V_1,V_1+V_2,V_1+V_2+V_3,\ldots$. Then $\mathscr P$ is a Poisson process with intensity $1/E[V_1] = O(\ep^{-1})$ and thus,
\begin{equation}
 \label{eq:954}
\P\Big(\mathscr P \cap \bigcup_{i=1}^k I_i \ne \emptyset\Big) = O(\ep^{1/2}) + o(1).
\end{equation}
We now have
\begin{align*}
 \P\Big(\forall i: X^{(\infty)}_{t_ia^3} = J_{t_ia^3}\Big) &\ge \P\Big(\nexists (i,j): t_ia^3\in[T_j-T^{\mathrm{BO}}_{j-1},T_j]\Big) && \text{by definition}\\
&\ge \P\Big(G_{n}',\ \sum_{i=1}^{n}V_i > t_k+2,\ \mathscr P \cap \bigcup_{i=1}^k I_i = \emptyset\Big) && \text{by definition of $G_n'$}\\
&\ge 1 - O(\ep^\delta) - o(1) && \text{by \eqref{eq:prob_Gnprime}, \eqref{eq:952}, \eqref{eq:954}}.
\end{align*}
Letting $A$ and $a$ go to infinity yields \eqref{eq:950} and thus proves the theorem.
\end{proof}

\section*{Acknowledgments}
I wish to thank my PhD advisor Zhan Shi for his continuous support and encouragements and Julien Berestycki for numerous discussions and his interest in this work.


\begin{thebibliography}{10}

\bibitem{Addario-Berry2009}
{\sc Addario-Berry, L., and Reed, B.}
\newblock {Minima in branching random walks}.
\newblock {\em The Annals of Probability 37}, 3 (May 2009), 1044--1079.

\bibitem{Aidekon2011b}
{\sc A\"{\i}d\'{e}kon, E.}
\newblock {Convergence in law of the minimum of a branching random walk}.
\newblock arXiv:1101.1810, Jan. 2011.

\bibitem{Aidekon2011c}
{\sc A\"{\i}d\'{e}kon, E., Berestycki, J., Brunet, E., and Shi, Z.}
\newblock {The branching Brownian motion seen from its tip}.

\bibitem{Aldous1978}
{\sc Aldous, D.}
\newblock {Stopping Times and Tightness}.
\newblock {\em The Annals of Probability 6}, 2 (Nov. 1978), 335--340.

\bibitem{Arguin2011}
{\sc Arguin, L.-P., Bovier, A., and Kistler, N.}
\newblock {Genealogy of extremal particles of branching Brownian motion}.
\newblock {\em Communications on Pure and Applied Mathematics 64}, 12 (Dec.
  2011), 1647--1676.

\bibitem{Arguin2011a}
{\sc Arguin, L.-P., Bovier, A., and Kistler, N.}
\newblock {The extremal process of branching Brownian motion}.
\newblock arXiv:1103.2322, Mar. 2011.

\bibitem{Athreya1972}
{\sc Athreya, K.~B., and Ney, P.~E.}
\newblock {\em {Branching processes}}.
\newblock Die Grundlehren der mathematischen Wissenschaften, Band 196.
  Springer-Verlag, New York, 1972.

\bibitem{Bellman}
{\sc Bellman, R.}
\newblock {\em {A brief introduction to theta functions}}.
\newblock Athena Series: Selected Topics in Mathematics. Holt, Rinehart and
  Winston, New York, 1961.

\bibitem{Berard2010}
{\sc B\'{e}rard, J., and Gou\'{e}r\'{e}, J.-B.}
\newblock {Brunet-Derrida behavior of branching-selection particle systems on
  the line}.
\newblock {\em Communications in Mathematical Physics 298}, 2 (June 2010),
  323--342.

\bibitem{Berestycki2010}
{\sc Berestycki, J., Berestycki, N., and Schweinsberg, J.}
\newblock {The genealogy of branching Brownian motion with absorption}.
\newblock arXiv:1001.2337, Jan. 2010.

\bibitem{Biggins2004}
{\sc Biggins, J.~D., and Kyprianou, A.~E.}
\newblock {Measure change in multitype branching.}
\newblock {\em Adv. Appl. Probab. 36}, 2 (2004), 544--581.

\bibitem{Billingsley1999}
{\sc Billingsley, P.}
\newblock {\em {Convergence of probability measures}}, second~ed.
\newblock Wiley Series in Probability and Statistics: Probability and
  Statistics. John Wiley \& Sons Inc., New York, 1999.

\bibitem{Bingham1974}
{\sc Bingham, N.~H., and Doney, R.~A.}
\newblock {Asymptotic Properties of Supercritical Branching Processes I: The
  Galton-Watson Process}.
\newblock {\em Advances in Applied Probability 6}, 4 (Jan. 1974), 711--731.

\bibitem{Bingham1987}
{\sc Bingham, N.~H., Goldie, C.~M., and Teugels, J.~L.}
\newblock {\em {Regular variation}}, vol.~27 of {\em Encyclopedia of
  Mathematics and its Applications}.
\newblock Cambridge University Press, Cambridge, 1987.

\bibitem{Bolthausen2001}
{\sc Bolthausen, E., Deuschel, J.-D., and Giacomin, G.}
\newblock {Entropic Repulsion and the Maximum of the two-dimensional harmonic
  crystal}.
\newblock {\em The Annals of Probability 29}, 4 (Oct. 2001), 1670--1692.

\bibitem{Bolthausen2010}
{\sc Bolthausen, E., Deuschel, J.-D., and Zeitouni, O.}
\newblock {Recursions and tightness for the maximum of the discrete, two
  dimensional Gaussian Free Field}.
\newblock arXiv:1005.5417, May 2010.

\bibitem{Bolthausen1998}
{\sc Bolthausen, E., and Sznitman, A.-S.}
\newblock {On Ruelle's Probability Cascades and an Abstract Cavity Method}.
\newblock {\em Communications in Mathematical Physics 197}, 2 (Oct. 1998),
  247--276.

\bibitem{BorodinSalminen}
{\sc Borodin, A.~N., and Salminen, P.}
\newblock {\em {Handbook of Brownian motion---facts and formulae}}, second~ed.
\newblock Probability and its Applications. Birkh\"{a}user Verlag, Basel, 2002.

\bibitem{Bramson2010}
{\sc Bramson, M.~D., and Zeitouni, O.}
\newblock {Tightness of the recentered maximum of the two-dimensional discrete
  Gaussian Free Field}.
\newblock arXiv:1009.3443, Sept. 2010.

\bibitem{Brunet1997}
{\sc Brunet, E., and Derrida, B.}
\newblock {Shift in the velocity of a front due to a cutoff}.
\newblock {\em Physical Review E 56}, 3 (Sept. 1997), 2597--2604.

\bibitem{Brunet2001}
{\sc Brunet, E., and Derrida, B.}
\newblock {Effect of Microscopic Noise on Front Propagation}.
\newblock {\em Journal of Statistical Physics 103}, 1 (2001), 269--282.

\bibitem{Brunet2006}
{\sc Brunet, E., Derrida, B., Mueller, A., and Munier, S.}
\newblock {Phenomenological theory giving the full statistics of the position
  of fluctuating pulled fronts}.
\newblock {\em Physical Review E 73}, 5 (May 2006), 056126.

\bibitem{Brunet2006a}
{\sc Brunet, E., Derrida, B., Mueller, A.~H., and Munier, S.}
\newblock {Noisy traveling waves: Effect of selection on genealogies}.
\newblock {\em Europhysics Letters (EPL) 76}, 1 (Oct. 2006), 1--7.

\bibitem{Buraczewski2009}
{\sc Buraczewski, D.}
\newblock {On tails of fixed points of the smoothing transform in the boundary
  case}.
\newblock {\em Stochastic Processes and their Applications 119}, 11 (Nov.
  2009), 3955--3961.

\bibitem{Chauvin1991}
{\sc Chauvin, B.}
\newblock {Product martingales and stopping lines for branching Brownian
  motion}.
\newblock {\em Ann. Probab. 19}, 3 (1991), 1195--1205.

\bibitem{Chauvin1988}
{\sc Chauvin, B., and Rouault, A.}
\newblock {KPP equation and supercritical branching Brownian motion in the
  subcritical speed area. Application to spatial trees}.
\newblock {\em Probability Theory and Related Fields 80}, 2 (Dec. 1988),
  299--314.

\bibitem{Derrida2007}
{\sc Derrida, B., and Simon, D.}
\newblock {The survival probability of a branching random walk in presence of
  an absorbing wall}.
\newblock {\em Europhysics Letters (EPL) 78}, 6 (June 2007), 60006.

\bibitem{Derrida1988}
{\sc Derrida, B., and Spohn, H.}
\newblock {Polymers on disordered trees, spin glasses, and traveling waves}.
\newblock {\em J. Statist. Phys. 51\/} (1988), 817--840.

\bibitem{Ding2011b}
{\sc Ding, J.}
\newblock {On cover times for 2D lattices}.
\newblock arXiv:1110.3367, Oct. 2011.

\bibitem{Durrett2009}
{\sc Durrett, R., and Remenik, D.}
\newblock {Brunet-Derrida particle systems, free boundary problems and
  Wiener-Hopf equations}.

\bibitem{BatemanII}
{\sc Erd\'{e}lyi, A., Magnus, W., Oberhettinger, F., and Tricomi, F.~G.}
\newblock {\em {Higher transcendental functions. {V}ol. {II}}}.
\newblock The Bateman Manuscript Project. McGraw-Hill Book Company, Inc., New
  York-Toronto-London, 1953.

\bibitem{Ethier1986}
{\sc Ethier, S.~N., and Kurtz, T.~G.}
\newblock {\em {Markov processes}}.
\newblock Wiley Series in Probability and Mathematical Statistics: Probability
  and Mathematical Statistics. John Wiley \& Sons Inc., New York, 1986.

\bibitem{Feller1971}
{\sc Feller, W.}
\newblock {\em {An introduction to probability theory and its applications. Vol
  II. 2nd ed.}}, second~ed.
\newblock John Wiley \& Sons Inc., New York, 1971.

\bibitem{Fisher1937}
{\sc Fisher, R.~A.}
\newblock {The wave of advance of advantageous genes}.
\newblock {\em Annals of Human Genetics 7}, 4 (June 1937), 355--369.

\bibitem{Gantert2011}
{\sc Gantert, N., Hu, Y., and Shi, Z.}
\newblock {Asymptotics for the survival probability in a killed branching
  random walk}.
\newblock {\em Annales de l'Institut Henri Poincar\'{e}, Probabilit\'{e}s et
  Statistiques 47}, 1 (Feb. 2011), 111--129.

\bibitem{Hardy2006}
{\sc Hardy, R., and Harris, S.~C.}
\newblock {A new formulation of the spine approach to branching diffusions}.
\newblock arXiv:math/0611054, Nov. 2006.

\bibitem{HHK2006}
{\sc Harris, J.~W., Harris, S.~C., and Kyprianou, A.~E.}
\newblock {Further probabilistic analysis of the
  Fisher-Kolmogorov-Petrovskii-Piscounov equation: one sided travelling-waves}.
\newblock {\em Ann. Inst. H. Poincar\'{e} Probab. Statist. 42}, 1 (2006),
  125--145.

\bibitem{Harris2011}
{\sc Harris, S.~C., and Roberts, M.~I.}
\newblock {The many-to-few lemma and multiple spines}.
\newblock arXiv:1106.4761, June 2011.

\bibitem{Harris1963}
{\sc Harris, T.~E.}
\newblock {\em {The theory of branching processes}}.
\newblock Die Grundlehren der Mathematischen Wissenschaften, Bd. 119.
  Springer-Verlag, Berlin, 1963.

\bibitem{Hu2009}
{\sc Hu, Y., and Shi, Z.}
\newblock {Minimal position and critical martingale convergence in branching
  random walks, and directed polymers on disordered trees}.
\newblock {\em The Annals of Probability 37}, 2 (Mar. 2009), 742--789.

\bibitem{ItoMcKean}
{\sc It\^{o}, K., and {McKean Jr.}, H.~P.}
\newblock {\em {Diffusion processes and their sample paths}}.
\newblock Second printing, corrected. Die Grundlehren der mathematischen
  Wissenschaften, Band 125. Springer-Verlag, Berlin, 1974.

\bibitem{Jagers1989}
{\sc Jagers, P.}
\newblock {General branching processes as Markov fields.}
\newblock {\em Stochastic Processes Appl. 32}, 2 (1989), 183--212.

\bibitem{Kallenberg1997}
{\sc Kallenberg, O.}
\newblock {\em {Foundations of modern probability}}.
\newblock Probability and its Applications (New York). Springer-Verlag, New
  York, 1997.

\bibitem{Kesten1978}
{\sc Kesten, H.}
\newblock {Branching brownian motion with absorption}.
\newblock {\em Stochastic Processes and their Applications 7}, 1 (Mar. 1978),
  9--47.

\bibitem{Knight1969}
{\sc Knight, F.~B.}
\newblock {Brownian local times and taboo processes}.
\newblock {\em Trans. Amer. Math. Soc 143\/} (1969), 173--185.

\bibitem{KPPshort}
{\sc Kolmogorov, A.~N., Petrovskii, I., and Piskunov, N.}
\newblock {\'{E}tude de l'\'{e}quation de la diffusion avec croissance de la
  quantite de mati\`{e}re et son application \`{a} un probl\`{e}me biologique}.
\newblock {\em Bull. Univ. \'{E}tat Moscou, S\'{e}r. Int., Sect. A: Math. et
  M\'{e}can. 1}, 6 (1937), 1--25.

\bibitem{Kurtz1975}
{\sc Kurtz, T.~G.}
\newblock {Semigroups of Conditioned Shifts and Approximation of Markov
  Processes}.
\newblock {\em The Annals of Probability 3}, 4 (Nov. 1975), 618--642.

\bibitem{Kyprianou2000}
{\sc Kyprianou, A.~E.}
\newblock {Martingale convergence and the stopped branching random walk.}
\newblock {\em Probab. Theory Relat. Fields 116}, 3 (2000), 405--419.

\bibitem{Kyprianou2004}
{\sc Kyprianou, A.~E.}
\newblock {Travelling wave solutions to the K-P-P equation: alternatives to
  Simon Harris' probabilistic analysis}.
\newblock {\em Ann. Inst. H. Poincar\'{e} Probab. Statist. 40}, 1 (2004),
  53--72.

\bibitem{Lambert2000}
{\sc Lambert, A.}
\newblock {Completely asymmetric L\'{e}vy processes confined in a finite
  interval}.
\newblock {\em Annales de l'Institut Henri Poincare (B) Probability and
  Statistics 36}, 2 (Mar. 2000), 251--274.

\bibitem{Liu2000}
{\sc Liu, Q.}
\newblock {On generalized multiplicative cascades}.
\newblock {\em Stochastic Processes and their Applications 86}, 2 (Apr. 2000),
  263--286.

\bibitem{LPP1995}
{\sc Lyons, R., Pemantle, R., and Peres, Y.}
\newblock {Conceptual proofs of $L\log L$ criteria for mean behavior of
  branching processes.}
\newblock {\em Ann. Probab. 23}, 3 (1995), 1125--1138.

\bibitem{Madaule2011}
{\sc Madaule, T.}
\newblock {Convergence in law for the branching random walk seen from its tip}.
\newblock arXiv:1107.2543, July 2011.

\bibitem{Maillard2010}
{\sc Maillard, P.}
\newblock {The number of absorbed individuals in branching Brownian motion with
  a barrier}.
\newblock arXiv:1004.1426v3, to be published in Ann. Inst. H. Poincar\'{e}
  Probab. Statist., Apr. 2010.

\bibitem{McKean1975}
{\sc McKean, H.~P.}
\newblock {Application of Brownian motion to the equation of
  Kolmogorov-Petrovskii-Piskunov}.
\newblock {\em Comm. Pure Appl. Math. 28}, 3 (1975), 323--331.

\bibitem{Mollison1977}
{\sc Mollison, D.}
\newblock {Spatial Contact Models for Ecological and Epidemic Spread}.
\newblock {\em Journal of the Royal Statistical Society. Series B
  (Methodological) 39}, 3 (Oct. 1977), 283--326.

\bibitem{Mueller2010}
{\sc Mueller, C., Mytnik, L., and Quastel, J.}
\newblock {Effect of noise on front propagation in reaction-diffusion equations
  of KPP type}.
\newblock {\em Inventiones mathematicae 184}, 2 (Nov. 2010), 405--453.

\bibitem{Neveu1986}
{\sc Neveu, J.}
\newblock {Arbres et processus de Galton-Watson}.
\newblock {\em Annales de l'Institut Henri Poincar\'{e} (B) Probabilit\'{e}s et
  Statistiques 22}, 2 (1986), 199--207.

\bibitem{Neveu1988}
{\sc Neveu, J.}
\newblock {Multiplicative martingales for spatial branching processes}.
\newblock In {\em Seminar on Stochastic Processes, 1987 (Princeton, NJ, 1987)},
  vol.~15 of {\em Progr. Probab. Statist.} Birkh\"{a}user Boston, Boston, MA,
  1988, pp.~223--242.

\bibitem{Panja2004}
{\sc Panja, D.}
\newblock {Effects of fluctuations on propagating fronts}.
\newblock {\em Physics Reports 393}, 2 (Mar. 2004), 87--174.

\bibitem{Revuz1999}
{\sc Revuz, D., and Yor, M.}
\newblock {\em {Continuous martingales and {B}rownian motion}}, third~ed.,
  vol.~293 of {\em Grundlehren der Mathematischen Wissenschaften [Fundamental
  Principles of Mathematical Sciences]}.
\newblock Springer-Verlag, Berlin, 1999.

\bibitem{Sawyer1976}
{\sc Sawyer, S.}
\newblock {Branching Diffusion Processes in Population Genetics}.
\newblock {\em Advances in Applied Probability 8}, 4 (Jan. 1976), 659--689.

\bibitem{Sevastyanov1958}
{\sc Sevast'yanov, B.~A.}
\newblock {Branching stochastic processes for particles diffusing in a bounded
  domain with absorbing boundaries}.
\newblock {\em Theory Probab. Appl. 3\/} (1958), 111--126.

\bibitem{Shiga1988}
{\sc Shiga, T.}
\newblock {Stepping stone models in population genetics and population
  dynamics}.
\newblock In {\em Stochastic processes in physics and engineering ({B}ielefeld,
  1986)}, vol.~42 of {\em Math. Appl.} Reidel, Dordrecht, 1988, pp.~345--355.

\bibitem{VanSaarloos2003}
{\sc {Van Saarloos}, W.}
\newblock {Front propagation into unstable states}.
\newblock {\em Physics Reports 386}, 2-6 (Nov. 2003), 29--222.

\bibitem{Watanabe1967}
{\sc Watanabe, S.}
\newblock {Limit theorem for a class of branching processes}.
\newblock In {\em Markov {P}rocesses and {P}otential {T}heory ({P}roc.
  {S}ympos. {M}ath. {R}es. {C}enter, {M}adison, {W}is., 1967)}. Wiley, New
  York, 1967, pp.~205--232.

\bibitem{Yang2011}
{\sc Yang, T., and Ren, Y.-X.}
\newblock {Limit theorem for derivative martingale at criticality w.r.t
  branching Brownian motion}.
\newblock {\em Statistics \& Probability Letters 81}, 2 (Feb. 2011), 195--200.

\end{thebibliography}
\end{document}